\documentclass[aos,preprint]{imsart} 
 
\RequirePackage{amsthm,amsmath,amsfonts,amssymb}
\RequirePackage[numbers]{natbib}

\startlocaldefs

 \usepackage[utf8]{inputenc} 
\usepackage[T1]{fontenc} 
\usepackage{mathtools} 
\usepackage{graphicx} 
\usepackage[english]{babel} 
\usepackage{amsmath,mathtools,amsthm,amssymb}
\usepackage{hyperref} 

 \usepackage[table,xcdraw]{xcolor} \usepackage{enumitem}
 \usepackage[colorinlistoftodos,french]{todonotes}
 \usepackage{mdframed}
 \usepackage{dsfont}

 \usepackage{mathabx}
 
 \usepackage{nccmath}

 \usepackage{titletoc}
 
 \usepackage{soul}
 
\def\1{1\!{\rm l}}

\newcommand{\leqa}{\lesssim}

\newcommand{\EM}{\ensuremath}

\newcommand{\al}{\alpha}
\newcommand{\be}{\beta}

\newcommand{\ga}{\gamma}

\newcommand{\la}{\lambda}
\newcommand{\La}{\Lambda}

\newcommand{\si}{\sigma}
\newcommand{\te}{\theta}
\newcommand{\ta}{\tau}
\newcommand{\veps}{\varepsilon}
\newcommand{\vphi}{\varphi}

\newcommand{\cA}{\EM{\mathcal{A}}}

\newcommand{\cC}{\EM{\mathcal{C}}}

\newcommand{\cE}{\EM{\mathcal{E}}}
\newcommand{\cF}{\EM{\mathcal{F}}}

\newcommand{\cN}{\EM{\mathcal{N}}}

\newcommand{\cR}{\EM{\mathcal{R}}}

\newcommand{\psg}{{\langle}}
\newcommand{\psd}{{\rangle}}

\definecolor{blendedblue}{rgb}{0.2,0.2,0.7}

\DeclareMathAlphabet{\mathpzc}{OT1}{pzc}{m}{it}

\newcommand{\RR}{\mathbb{R}}
\newcommand{\CC}{\mathbb{C}}

\newcommand{\mb}{\mathbb{B}}
\newcommand{\mh}{\mathbb{H}}

\newcommand{\N}{\mathbb{N}}

\newcommand{\given}{\,|\,}

\newcommand{\eps}{\varepsilon}

\newcommand{\R}{\mathds{R}}

\newcommand{\bi}{\begin{enumerate}[label=\roman*)]}
\newcommand{\ei}{\end{enumerate}}
\newcommand{\ba}{\begin{array}{rcl}}
\newcommand{\ea}{\end{array}}

\newcommand{\bok}{{\boldsymbol{k}}}

\makeatletter
\newcommand*\bigcdot{\mathpalette\bigcdot@{.5}}
\newcommand*\bigcdot@[2]{\mathbin{\vcenter{\hbox{\scalebox{#2}{$\m@th#1\bullet$}}}}}
\makeatother

\newcommand\restr[2]{{
  \left.\kern-\nulldelimiterspace 
  #1 
  \vphantom{\big|} 
  \right|_{#2} 
  }}

\newcommand{\norm}[1]{\left\lVert#1\right\rVert}

\theoremstyle{plain}
\newtheorem{proposition}{Proposition}
\newtheorem{corollary}{Corollary}
\newtheorem{theorem}{Theorem}
\newtheorem{lemma}{Lemma}

\theoremstyle{remark}
\newtheorem{remark}{Remark}

\newtheorem{example}{Example}

\newtheoremstyle{case}{}{}{}{}{}{:}{ }{}
\theoremstyle{case}

\endlocaldefs

\begin{document}

\begin{frontmatter}

\title{Deep Horseshoe Gaussian Processes}
\runtitle{Deep Horseshoe Gaussian Processes }

\begin{aug}
\author[A]{\fnms{Ismaël}~\snm{Castillo}\ead[label=e1]{ismael.castillo@upmc.fr}},
\and
\author[B]{\fnms{Thibault}~\snm{Randrianarisoa}\ead[label=e3]{t.randrianarisoa@utoronto.ca}}
\address[A]{Laboratoire de Probabilités, Statistique et Modélisation, Sorbonne Université\printead[presep={,\ }]{e1}} 
 
\address[B]{Department of Statistical Sciences, University of Toronto\printead[presep={,\ }]{e3}}
\end{aug}

\begin{abstract}
Deep Gaussian processes 
have recently been  proposed as natural objects to fit, similarly to deep neural networks, possibly complex features present 
 in modern data samples, such as compositional structures. 
Adopting a Bayesian nonparametric approach, it is natural to use deep Gaussian processes as prior distributions, and  use the corresponding posterior distributions for statistical inference.  We introduce the deep
Horseshoe Gaussian process   \textsf{Deep--HGP}, a new simple prior based on deep Gaussian processes with a squared-exponential kernel, that in particular enables data-driven choices of the key lengthscale parameters. For nonparametric regression with random design, we show that the associated posterior distribution recovers the unknown true regression curve optimally in terms of  quadratic loss, up to a logarithmic factor, in an adaptive way. The convergence rates are {\em simultaneously} adaptive to both the smoothness of the regression function and to its structure in terms of compositions. The dependence of the rates in terms of dimension are explicit, allowing in particular for  input spaces of dimension increasing with the number of observations.
\end{abstract}
 


\end{frontmatter}


\section{Introduction}
 
 
Gaussian processes (henceforth GPs) are among the most used machine learning methods, with  applications ranging from inference in regression models to classification, see e.g. \cite{rasmussenbook} for an overview. Due to their flexibility, in recent years GPs have been used as tools for geometric inference and deep learning. Before turning to deep Gaussian processes, and since our results are also relevant for standard GPs, we start with a brief overview of recent results for Gaussian processes.

A particularly natural field of application where there now exists at least partial theory to explain and validate practical successes of GPs is that of {\em Bayesian nonparametrics}:  the posterior distribution corresponding to taking a GP as prior distribution on functions can  be used for function estimation as well as for the practically essential task of {\em uncertainty quantification}. 
In a regression setting, when putting a GP prior distribution on the unknown regression function,  the corresponding posterior distribution can often be efficiently implemented \cite{rasmussenbook}
and comes with theoretical convergence guarantees: the works \cite{aadharry07, aadharry08, ic08} indeed show that the posterior contraction rate in terms of relevant loss functions (e.g. $L^2$--loss for regression) is completely determined (both upper and lower bounds) by the behaviour of its concentration function. Shortly thereafter, van der Vaart and van Zanten also showed that statistical {\em adaptation to smoothness} was possible with GPs with optimal minimax contraction rates by simply drawing at random its scaling parameter \cite{aadharry09} in fixed design regression; see \cite{patietal15, aadharry11} for extensions to random design regression and \cite{aretha20} to inverse problems. Results on uncertainty quantification include \cite{svv15}, \cite{yangetalpreprint} in nonparametric models and \cite{c12, cr15} in semiparametric settings.
 
Let us mention a few applications of posterior distributions arising from GPs that illustrate their flexibility and are related to the setting considered below. 
 
{\em GPs flexibility: geometric settings.} In modern statistical models, it is frequent that data naturally sit on a geometric object such as a compact manifold (one can think of a sphere, a swissroll etc.). It is tempting to use GPs in this setting as well, although  some care is needed in their construction. For instance, the celebrated GP with  squared--exponential kernel (thereafter \textsf{SqExp}) has no immediate analog in a manifold setting, as replacing the euclidian metric in the exponential defining   \textsf{SqExp} with the geodesic distance does not form a covariance kernel. This can be remediated by using a kernel coming from heat equation solutions on the manifold \cite{ckp14}, and this kernel can be shown to be a natural geometric analog of \textsf{SqExp}. Alternatively, one may put a prior directly on the ambient space equipped with the standard euclidean metric: the authors in \cite{yangdunson16} obtain a posterior rate that under some (smoothness) conditions adapts to the unknown dimension of the manifold with a rescaled \textsf{SqExp} exponential GP, when the loss function is the quadratic loss but restricted to sit on the manifold; this is further refined in the recent work \cite{nanetal24}.

{\em GPs flexibility: adaptation to anisotropy and variable selection}. By drawing independent lengthscale parameters along different dimensions, \cite{bpd14} shows that posteriors arising from \textsf{SqExp} GPs contract at near-optimal minimax anisotropic rates. A related problem is that of variable selection in (possibly high-dimensional) regression. The unknown regression function may indeed depend only on a few coordinates (although these are not known in advance). 
 By considering variable selection type priors and then drawing lengthscale parameters of \textsf{SqExp} GPs, \cite{yangtokdar15} and \cite{jiangtokdar21} provide theory for this setting and respectively investigate optimal rates and variable selection properties for the corresponding posterior distributions.
 

Recent years have seen a number of remarkable applications of {\em deep learning} methods, where `deep' typically refers to a certain (often compositional) structure in terms of a number of layers. For instance, deep neural networks are now routinely trained for image or speech processing, giving excellent empirical performance. Theoretical understanding in terms of convergence of statistical procedures is recent and includes \cite{kohlerlanger21, jsh20} for results on empirical risk minimisers for classes of deep neural networks with ReLU activation function in regression settings.  A Bayesian counterpart of the results in \cite{jsh20} with theoretical guarantees is considered in \cite{rp18}, where spike-and-slab priors are placed on network coefficients. Sampling directly from the corresponding posterior can be costly due to the combinatorial nature of the search of nonzero network coefficients; the works \cite{chengetal20,badr20} consider theory and implementation for mean-field variational Bayes versions thereof;  the work \cite{ce24} considers rescaled heavy-tailed priors on weights. Among similarities between GPs and neural networks, it has been shown  in \cite{Neal1996, hazan2015steps, Matthews2018} that both deep and shallow Bayesian neural networks with random parameters (appropriately rescaled to avoid degeneracy) and with all layers of width growing to infinity asymptotically behave like GPs, with covariance kernel depending on the network structure. The Bayesian approach we describe in the next paragraphs avoids the use of large networks using activation functions by modelling layers directly through independent Gaussian processes. 
  
Deep Gaussian processes \cite{damianou13} (henceforth DeepGPs) correspond to iterated compositions of Gaussian processes and broadly speaking can be seen as 
a possible Bayesian analogue of deep neural networks. Figure \ref{fig: DGP sqexp} depicts the sample path of a simple DeepGP obtained from two independent GPs with squared-exponential kernels.  The random paths resulting from DeepGPs have greater modelling flexibility compared to single Gaussian processes, enabling for instance to capture different spatial behaviours; \cite{grsh22} shows that single GPs cannot reach optimal rates for compositional structures, see also \cite{abraham2023deep}. While the infinite-width limits of deep Bayesian neural networks behave like GPs, forcing instead some layers of the network to be of fixed width while letting others grow leads to a deepGP (see Section 7 of \cite{fsh21}). One then indeed obtains in the limit the composition of the limiting GPs in-between the fixed layers. There is a lot of recent activity for providing efficient sampling methods for deepGPs \cite{GPflux, salimbeni2017, salimbeni2019}. Yet, theory is just starting to emerge.
 
 \begin{figure}[h!]
\center
  \caption{Composition of two Gaussian processes with \textsf{SqExp} covariance kernel $K(s,t)=e^{-(s-t)^2}$.}
  \includegraphics[scale=0.44]{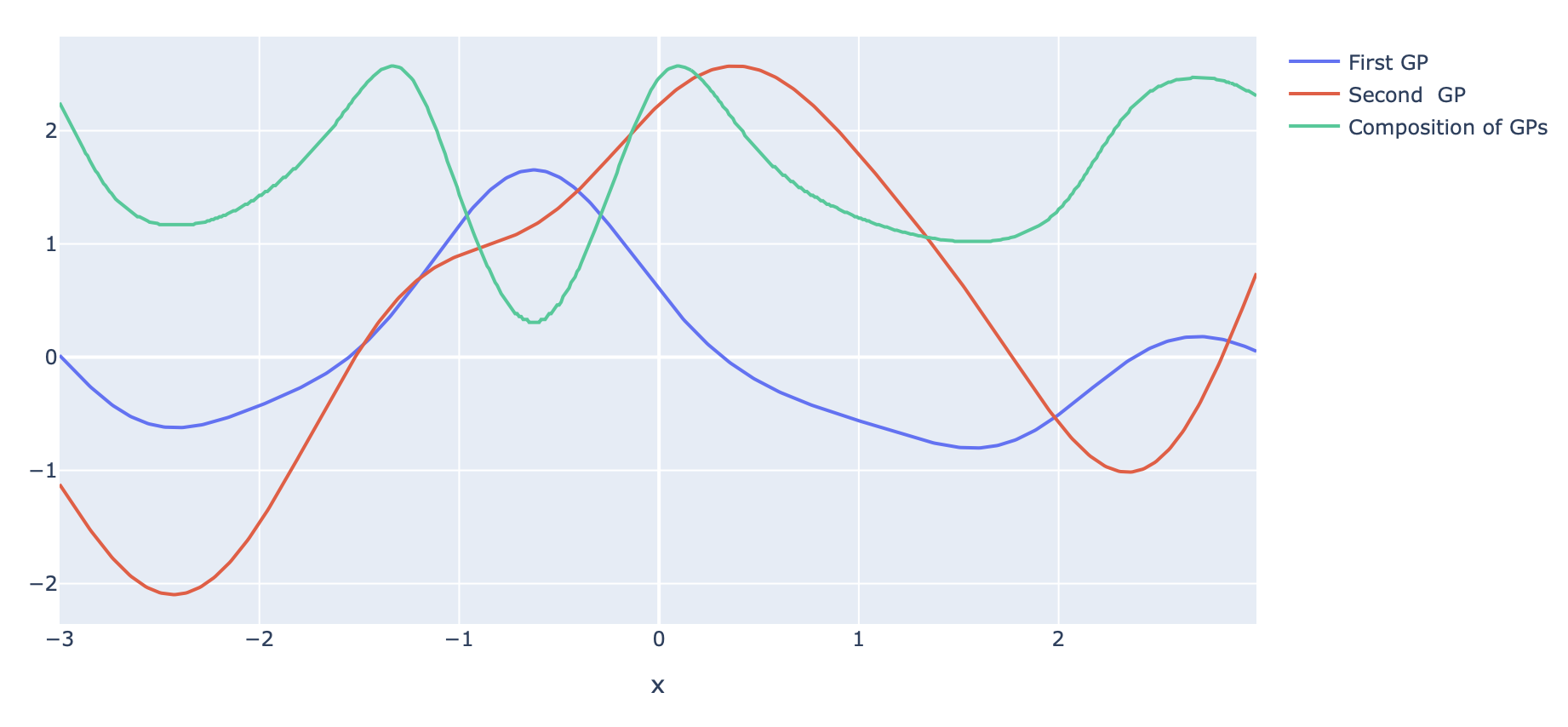}
  \label{fig: DGP sqexp}
\end{figure}
 
The recent seminal work by Finocchio and Schmidt--Hieber \cite{fsh21} on deepGPs shows that using  a model selection prior to select active variables in the successive Gaussian processes, and conditioning individual GP sample paths to verify certain smoothness constraints, the induced posterior distributions contract nearly optimally and adaptively in quadratic loss for compositional structures in regression, for a variety of kernel choices. 
Focusing on compositions of constrained GPs (with bounded sample paths and derivatives), the paper \cite{bl21} uses
 an adapted concentration function for deepGPs and derives near-optimal contraction rates in density estimation and classification. In \cite{mt_dgp23}, using a recursive representation, the authors derive convergence rates  for the posterior mean of deepGPs in a regression setting both in a noiseless case and with noisy data. In \cite{abraham2023deep}, the authors investigate the use of deepGPs for a class of nonlinear inverse problems. 
 
 This work follows the footsteps of \cite{fsh21} and aims at answering the following two questions. The first concerns the possibility to obtain theory and optimality results for a deepGP construction as simple as possible that comes closer to current implementations of deepGPs in practice. The second concerns the possibility to allow for a high-dimensional ambient space as well as a smaller intrinsic dimension.
 \begin{enumerate}
 \item {\em Can deepGP priors avoid an explicit model selection step?}\\
   While the deepGP prior construction in  \cite{fsh21} is completely natural and `canonical' from the theoretical perspective, both the conditioning step (to match smoothness constraints) and the model selection prior (for which the posterior on submodels is often expensive to compute) make posterior sampling more challenging in view of practical implementation. One main objective here is to try to simplify the construction of the prior as much as possible while keeping optimality properties, and thereby come closer to the practically used deepGPs, for which lengthscale parameters are often kept free and then adjusted in an empirical Bayes \cite{damianou13} or hierarchical Bayes  \cite{saueretal23} fashion. In view of this last observation, we propose a prior with a `soft' model selection based on a prior on lengthscales instead of the previous `hard' model selection prior. 
 \item {\em How do deepGPs scale with respect to `dimension'?}\\
 Below we shall allow in some results the input space dimension $d$ to grow with $n$. Even though any method must then face a `curse of dimensionality',  if the effective `intrinsic' dimension of the problem remains fixed or very slowly grows with $n$, it is conceivable that rates of convergence can still be obtained. While recent work on deep methods has shown that convergence rates  only depending on intrinsic dimension(s) can be derived \cite{jsh20}, most results are quite generous in the dependence on dimension of the constant factor in the rate.  In particular, we are not aware of works allowing for input and `intrinsic' dimensions to possibly grow with $n$ (\cite{ohn22} considers an example of Bayesian deep ReLU network with growing $d$ but {\em fixed} intrinsic dimension). 
  We demonstrate below that our construction can adapt to the intrinsic dimension even for a high-dimensional ambient space (with $d$ sublinear in $n$). This requires a careful tracking of the dependence on dimension,  in particular revisiting earlier results in the GP literature to make the dependence on $d$ precise.
 \end{enumerate}
The main contributions of the paper are as follows:
\begin{enumerate}
\item we introduce a new idea of {\em freezing--of--paths} for multi-bandwidth Gaussian processes.  The benefit of random lengthscales of a stationary kernel for adaptation to {\em smoothness} has been established for a while \cite{aadharry09, aadharry11, bpd14, patietal15}. Such an adaptation to the regularity of the underlying truth is made possible  by letting the lengthscales grow polynomially with $n$ in a suitable way with sufficient probability under the prior. 
 In the present paper, we show that letting lengthscales appropriately {\em vanish} (instead of diverge) enables adaptation to {\em structure} (instead of smoothness) or in other words to adapt to sparsity in the covariates dependence, by `freezing' irrelevant dimensions through the corresponding 
 posterior distributions. Intuitively, sample paths become almost constant in the directions with vanishing lengthscales, performing effectively a form of `soft' model selection. 
 \item we show that the previous two effects of lengthscale parameters, namely adaptation to smoothness and to structure (by using respectively diverging and vanishing lengthscales), can be obtained using a {\em single} prior for lengthscales: the {\em horseshoe} distribution \cite{carvalho10}, that both puts a lot of mass near zero and at the tails, is shown to lead to optimal contraction rates with near-optimal scaling in terms of dimension. Our results also include exponential prior distributions on lengthscales as in earlier contributions on Gaussian processes (e.g. \cite{aadharry09}), although dependence on dimension may not be optimal in `large $d$' regimes.
\item 
we study a high-dimensional setting where the input space has growing dimension combined with a compositional structure and functions in the composition having few active coordinates; in particular we allow the input dimension to grow polynomially with $n$ and the number of actually relevant variables in the input layer to grow slowly with $n$. A main technical contribution of the paper consists in deriving dimension-dependent analogues of the inequalities that are at the heart of GP regression theory with a squared-exponential kernel. Namely, we give precise dependence on ambient and intrinsic dimensions of the metric entropy of the unit ball of the RKHS of the covariance kernel, of the small ball probability of the GP and on quantities measuring approximation properties of this RKHS.
\end{enumerate}
We note that the results are not only relevant for deep learning applications, but also already for  shallow (standard) Gaussian processes, for which the freezing-of-paths effect described above is shown to lead to effective `variable selection' in that the achieved convergence rate only depends on the number of truly present variables. Also, from the technical perspective,
in order to leverage the smaller intrinsic dimensionality of the problem, a key new ingredient in the proofs  consists in replacing the prior by a `low-dimensional' oracle GP defined on the relevant coordinates.   Finally in this paper for technical convenience we mostly focus on  {\em tempered} posterior distributions, for which the likelihood in Bayes formula is raised to a fixed power $\rho$ smaller than $1$; we do however also obtain results for the standard posterior ($\rho=1$) when the nonparametric prior is coupled with an appropriate prior on the noise variance. 
 
The paper is organised as follows:  Section \ref{sec: the prior} introduces the statistical model and deep Gaussian process priors. We recall the main elements of the frequentist theoretical analysis of GP regression in Section \ref{sec: key ingredients}. Our main results are split into two parts. Section \ref{sec:mainres1} considers a setting without compositions, and Theorem \ref{thmvs} therein shows that under mild conditions multibandwidth GPs effectively achieve a form of variable selection through a freezing-of-paths effect. Section \ref{sec:mainres2} considers adaptation to compositional structures: Theorems \ref{theorem: posterior contraction rates hdgp} and \ref{theorem: posterior contraction rates hdgp 2} therein demonstrate that deep horseshoe GPs lead to near-minimax optimal contraction rates  both in fixed dimensions and in the high-dimensional case, while Section \ref{sec:truepost} explains how results for tempered posteriors can be transferred to standard posteriors. A discussion follows in  Section \ref{sec:disc}. Proofs are provided in Section \ref{sections: proof rho post}. A key result underlying the proofs and bounding the GPs' concentration function is presented in Section \ref{sec:boundcf}.  Auxiliary lemmas and their proofs can be found in the appendix \cite{AnnexeDGP}.
 
{\em Notation.} For two real numbers $a,b$, we let $a\wedge b=\min(a,b)$, $a\vee b=\max(a,b)$. We denote by $\phi$ the density of a standard normal random variable. The $\varepsilon$-covering number $N\left(\varepsilon, S, D\right)$ of a semimetric space $S$ equipped with a semimetric $D$ is the minimal number of balls of radius $\varepsilon$ needed to cover $S$. For a vector $A=\left(A_1,\dots,A_d\right)\in\mathbb{R}_+^d$, denote 
\[ \bar{A}\coloneqq \underset{i=1,\dots,d}{\max}\ A_i,\qquad |A|=\sum_i A_i,\]
and $A_I=(A_i)_{i\in I}$ for $I\subset [\! [ 1,n]\! ]$, the set of all integers between $1$ and $n$. Also, for any vector $\boldsymbol{x}\in\mathbb{R}^d$, we note $|\boldsymbol{x}|_\infty\coloneqq \max_i |x_i|$ its supnorm. For $f$ integrable on $\R^d$,  let $\hat{f}(t)=\int_{\R^d} e^{-i\langle t, s\rangle} f(s) ds$ denote its Fourier transform, with $\psg \cdot,\cdot\psd$ the euclidean inner product. In the following, $C,C_1, c_1, C_2, c_2,\dots$ denote absolute constants whose values may change from line to line.
 
%

 \section{The deep horseshoe GP prior}\label{sec: the prior}
 
Consider a nonparametric regression model with random design, where one observes $(X,Y):=(X_i,Y_i)_{1\leq i\leq n}$, with $X_1,\dots,X_n$ independent identically distributed design points sampled from a probability measure $\mu$ on $I^d$ for $I$ an interval of $\RR$ chosen for simplicity to be $\left[-1,1\right]$ in the sequel and
\begin{equation} \label{def:rdreg}
Y_i=f_0(X_i)+\veps_i,
\end{equation}
for $f_0:I^d\to \RR$ an unknown regression function and  $\veps_i$  independent $\mathcal{N}(0,\sigma_0^2)$ errors, with $\sigma_0$ assumed known for simplicity. We consider estimation of $f_0$ with respect to the integrated quadratic loss

\[ \norm{f_0-f}_{L^2(\mu)}^2 = \int (f_0-f)^2d\mu.\]
For a given regression function $f$, let $P_f$ denote the distribution of one observation $(X_i,Y_i)$ under model \eqref{def:rdreg}, which has density
\[p_f(x,y)=\left(2\pi\sigma_0^{2}\right)^{-1/2}e^{-(y-f(x))^2/(2\sigma_0^2)}\] with respect to $\mu\otimes \lambda$, for $\lambda$ the Lebesgue measure on $\RR$.  \\


For a real $\be>0$, $\lfloor \beta \rfloor$ the largest integer strictly smaller than $\beta$ and $r$ an integer, 
let $\mathcal{C}^{\beta}[-1,1]^{r}$ denote the classical Hölder space equipped with the norm $\norm{\,\cdot\,}_{\beta,\infty}$. It consists of functions $f:[-1,1]^r\to \RR$ whose norm defined as
\begin{equation} \label{defno}
  \norm{f}_{\beta,\infty}=  \max\left(\max_{ |\boldsymbol{\alpha}|\le \lfloor \beta \rfloor} \norm{\partial^{\boldsymbol{\alpha}} f}_{\infty} , \max_{\boldsymbol{\alpha}: |\boldsymbol{\alpha}|=\lfloor \beta \rfloor} \underset{\boldsymbol{x}, \boldsymbol{y}\in [-1,1]^r,\ \boldsymbol{x}\neq \boldsymbol{y}}{\sup} \frac{\left|\partial^{\boldsymbol{\alpha}} f(\boldsymbol{x}) - \partial^{\boldsymbol{\alpha}} f(\boldsymbol{y})\right|}{\left|\boldsymbol{x}- \boldsymbol{y}\right|^{ \beta-\lfloor \beta \rfloor}_{\infty}}\right) 
\end{equation}  
 is finite, 
with the multi-index notation $\boldsymbol{\alpha}=(\alpha_1,\dots,\alpha_d)\in\N^d$ and $\partial^{\boldsymbol{\alpha}} =\partial^{\boldsymbol{\alpha}_1} \dots \partial^{\boldsymbol{\alpha}_d}$. We note that functions with finite Hölder norm are bounded, for any $\beta>0$.

\subsection{Structural assumptions for multivariate regression}\label{subsec: model}

In order to assess the performance of machine learning methods, a popular benchmark is the regression setting \eqref{def:rdreg} equipped with some `structural' assumptions. In the unconstrained case where only a smoothness condition is assumed on $f_0$, rates for $\beta$--H\"older smooth functions are typically of the form $n^{-\be/(2\be+d)}$, and so are prone to the curse of dimensionality (the rate becomes extremely slow for large $d$). A common approach is to assume that  the multivariate regression function $f_0$ admits a certain {\em unknown} `structure', of `effective dimension $d^*$' possibly much smaller than $d$. For instance, in the simplest setting considered below, $f_0$ may only depend on a small but unknown number of coordinates. The goal is then to find algorithms that are able to achieve  optimal risk bounds that adapt to the unknown underlying structure, and that therefore scale with $d^*$ instead of $d$. \\

{\em A first basic setting: effective variable selection}. 
Let us first consider the simple setting where $f_0:\left[-1,1\right]^d =I^d\to \RR$ only depends on $d^*$ variables, that is
\begin{equation}  \label{vs}
f_0(x_1,\dots,x_d)=g(x_{i_1},\dots,x_{i_{d^*}}), 
\end{equation}
for some $g\in \mathcal{C^{\beta}}(I^{d^*})$ and $\beta>0$. The subset of indices $i_1,\ldots,i_{d^*}$ is unknown to the statistician and the target convergence rate in quadratic loss is $n^{-\beta/(2\beta+d^*)}$. We call this setting {\em effective variable selection}, where by this we mean that one aims at achieving the same performance as if the indices of the truly present variables were known. We note that we do not consider here the problem of {\em actual} variable selection, where the goal would be to recover this set of  indices, and which would require some further conditions; we refer to \cite{jiangtokdar21} for more details on this task.  

Define, for $K, \be>0$ and  $d^*\le d$ two integers, and recalling that $I=[-1,1]$,
\begin{align} 
\cF_{VS}(K,\beta, d,d^*) 
&  = \Big\{ f_0:I^d \to \RR \text{ such that } \eqref{vs} \text{ holds for some }  g\in \cC^\be(I^{d^*}), \label{defvs} \\ 
&\quad \quad\norm{g}_{\infty}\leq 1\text{ and } \| g\|_{\be,\infty} \le K \Big\}. \nonumber
\end{align} 
As in recent works in deep learning and similar to \cite{fsh21}, we assume that an upper-bound $M_0$ is known for the true function $f_0$ and without loss of generality we assume $M_0=1$. 
\\

{\em Compositional structure.}  
Following \cite{jsh20}, suppose that $f$ can be written as a composition
 \begin{equation}\label{compo}
  f=h_q\circ\dots\circ h_0, 
 \end{equation} 
 with $h_i:\, I^{d_i}\to I^{d_{i+1}}$, for $(d_i)$ a sequence of integers such that $d_0=d$ and $d_{q+1}=1$. Since $h_i$ takes values in $\RR^{d_{i+1}}$, one may write $h_i= (h_{ij})$, where $h_{ij}$ for $j=1,\dots,d_{i+1}$ are its univariate coordinate functions. 
 
{\em The compositional class $\mathcal{F}_{\text{deep}}(\lambda,\mathbb{\beta},K)$.} 
Let us further assume that $h_{ij}$'s as above only depend on a subset $\mathcal{S}_{ij}$ of at most $t_i\leq d_i$ variables, and $h_{ij}$ restricted to the variables $\mathcal{S}_{ij}$ belongs to $\mathcal{C}^{\beta_i}(I^{t_i})$. 
For $\lambda=(q,d_1,\dots,d_q,t_0,\dots,t_q)$, 
$\mathbb{\beta}=(\beta_0,\dots,\beta_{q})$ and $K>0$,  let 
\begin{align} 
\mathcal{F}_{deep}(\lambda,\mathbb{\beta},K)
 = \left\{ h_q\circ\dots\circ h_0:\, I^{d}\to I,\ h_i:\, I^{d_i}\to I^{d_{i+1}}, 
\ h_{ij}\in \cF_{VS}(K,\beta_i,d_i,t_i) \right\}.
\label{defcld}
\end{align}

These compositional classes encompass several well-studied structural models in the literature. For instance, in an additive model with a fixed number $D>0$ of covariates,  the regression function can be expressed as
\begin{equation}\label{eq: additive model}f_0(x_1,\dots,x_D)=\sum_{i=1}^D g_i(x_i)=h_1\circ h_0(x),\end{equation}
where $h_0(x)=(g_1(x_1),\dots,g_d(x_D))$ and $h_1(y)=\sum_{i=1}^D y_i$. The original function $f_0$ is then the composition of two functions: one where each component is univariate and depends on a single variable ($h_0$) so that $d_1=t_0=1$, and another that depends on all variables but is infinitely smooth ($h_1$), so $t_1=D$. Therefore, further assuming $g_i\in \mathcal{C}^{\beta}(I)$, $\norm{g_i}_\infty\leq 1$ and $\norm{g_i}_{\beta,\infty}\leq K$ and , we have $f_0\in \mathcal{F}_{\text{deep}}((1,1,1,D), (\beta, \beta'),\max(D,K))$ for arbitrarily large $\beta'$ (taking $M_0=D$ in this case). We refer the reader to Section 4 of \cite{jsh20} for more examples highlighting the link between compositional classes and usual structural constraints in nonparametric regression.

{\em Minimax optimal rate.} The minimax rate of estimation in quadratic loss over this class \[( r_n^* )^2  = \inf_{T} \sup_{f\in \mathcal{F}_{deep}(\lambda,\mathbb{\beta},K)} 
E_f\|T - f\|_2^2,  
\]
for $T$ ranging over the set of estimators of $f$, is, up to logarithmic factors, 
\begin{equation} \label{targetrate}
r_n^* \asymp \underset{i=0,\dots,q}{\max}\ n^{-\frac{\beta_i\alpha_i}{2\beta_i\alpha_i+t_i}},\qquad \text{where }\alpha_i:=\prod_{l=i+1}^q (\beta_l \wedge 1),
\end{equation} 
under the mild condition $t_i\le \min(d_0,\ldots,d_{i-1})$, see \cite{jsh20}. For example, in the above additive model, for $\beta'\geq 1\vee D\beta$, the rate becomes, up to a multiplicative constant depending polynomially on $D$, \[n^{-\frac{\beta}{2\beta+1}} \vee n^{-\frac{\beta'}{2\beta'+D}}=n^{-\frac{\beta}{2\beta+1}}.\] These fast rates are attainable because for both functions in the composition \eqref{eq: additive model}, 
fast rates can be achieved without suffering from the curse of dimensionality ($D$ will still feature as a multiplicative constant in the rate, but importantly not in the exponent of $1/n$).
 
 \subsection{Key ingredients}\label{sec: key ingredients}

{\em Posterior distributions and frequentist analysis.} Given a prior distribution $\Pi$ on regression functions, 
 the posterior mass $\Pi[B \given X,Y]$ of a measurable set $B$ is given by Bayes' formula: this is the next display for $\rho=1$. More generally, one may set, for any $\rho\in(0,1]$,
 \[\Pi_{\rho}[B\given X,Y] = \frac{\int_B  \prod_{1\leq i\leq n} p_f(X_i,Y_i)^\rho d\Pi(f) }{\int \prod_{1\leq i\leq n} p_f(X_i,Y_i)^\rho d\Pi(f)}. \]
When $\rho=1$ this is the usual conditional probability that $f$ belongs to $B$ given the data. If $0<\rho<1$, this quantity is called 
{\em ${\rho}$--posterior} (or tempered posterior). Its use is very much widespread in machine learning, in particular in PAC--Bayesian settings  \cite{catoni04, tzhang2006}. We use the tempered posterior in our main results, and also provide results for the standard posterior $\rho=1$ and an augmented prior in Section \ref{sec:truepost}. Links with the case $\rho=1$ are also further discussed in Section \ref{sec:disc}.
 
 {\em Gaussian process ($\rho$--)\,posteriors: theory.}
  For any centered Gaussian process $W$ on the Banach space of continuous functions on $I^d$ equipped with the $\norm{\,\cdot\,}_\infty$ norm, the probability measure of any ball $\left\{f: \norm{f-g}_\infty<\varepsilon \right\}$ is lower bounded by a quantity depending on the mass of the centered ball of radius $\varepsilon$ and on how well $g$ can be approximated by elements of the RKHS $(\mathbb{H}, \norm{\,\cdot\,}_{\mathbb{H}})$ 
  of the covariance kernel of the process. More precisely, according to Proposition 11.19 of \cite{MR3587782}, for any $\veps>0$,
 \begin{align}\label{eq: concentration function}
 P\left[\norm{W-g}_\infty <\varepsilon\right]&\geq e^{-\varphi_g(\varepsilon/2)},\nonumber\\
 \varphi_g(\varepsilon)&= \underset{h\in\mathbb{H}:\ \norm{h-g}_\infty\leq \varepsilon}{\inf} \frac{1}{2} \norm{h}_{\mathbb{H}}^2- \log P\left[\norm{W}_\infty<\varepsilon\right].
\end{align}
 The function $\vphi_g$ in \eqref{eq: concentration function} is called the {\em concentration function} of the Gaussian process $W$ and plays a key role for contraction rates of GPs \cite{aadharry08, ic08}. It is the sum of two terms: the first term with the infimum is the approximation term whereas the second term is called the small ball term (the probability within the log is the small ball probability of the process $W$).
 
 In nonparametric regression with fixed design, 
  \cite{aadharry09} proved that  posterior contraction rates that are adaptive to the unknown smoothness of the regression function are achievable for stationary Gaussian process priors, with a dilatation parameter of the sample paths distributed as a Gamma variable. 
 As a particular case, consider the {\em squared exponential} process \textsf{SqExp} defined  as the zero-mean Gaussian process with covariance kernel $K(s,t)=\exp(-\norm{t-s}^2)$ (and $\norm{\,\cdot\,}$ the euclidean norm) on $[-1,1]^d$. Next, for $k,\theta>0$, one sets
  \begin{align*}
 A^d &\sim \text{Gamma}\left(k,\theta\right)\\
f\ \big|\ A &\sim \left\{W_{At}:\ t\in [-1,1]^d \right\}.
 \end{align*}
This construction induces a prior on the Banach space of continuous functions for which the posterior concentrates 
in the empirical $L_2-$norm at rate $\varepsilon_n\asymp n^{-\beta/(2\beta+d)}$ (up to a log factor) whenever $f_0$ has $\beta$-Hölder regularity, $\beta>0$. 
%
 
Although this rate coincides with the minimax estimation rate over a ball in $\mathcal{C^{\beta}}[-1,1]^d$, it becomes very slow for large $d$. When the regression function $f_0$ depends on a small number of variables $d^*$ only, a special case of the add-GP prior from \cite{yangtokdar15} gives optimal posterior contraction rates $n^{-\beta/(2\beta+d^*)}$ without the need to estimate $d^*$. This is achieved by the introduction of an additional layer in the prior, drawing via Bernoulli random variables in which direction the Gaussian sample paths have to be dilated (the sample paths being constant in the other directions). From a practical point-of-view, this `hard' selection of variables adds a combinatorial complexity to posterior sampling.
Similarly, if the design points are located on a $d^*$-dimensional Riemannian manifold, $d^*<d$, of the ambient space $[-1,1]^d$,  we expect the faster rate $n^{-\beta/(2\beta+d^*)}$ to be attainable. The work \cite{yangdunson16} achieves it with a dilated Gaussian process as well, the dilatation factor $A$ being distributed as $A^{d^*} \sim \text{Gamma}\left(k,\theta\right)$. One may note that this approach requires an estimate of $d^*$ to be applied and that posterior contraction rates are obtained for local distances {\em on the manifold} (such as the empirical $L_2$-norm) only, but not on the ambient space.

 
Technically, in the works above, adaptation to smoothness is achieved using that the distribution on $A^d$ in the last display selects large values for the dilatation parameters with large enough probability, together with a study of the dependence of the concentration function \eqref{eq: concentration function} on the lengthscale parameter. 
In particular, \cite{aadharry09} (isotropic case) and \cite{bpd14} (anisotropic case) develop a theory of approximation of (H\"older--)smooth functions in $\mathbb{R}^d$ by elements of RKHS from such Gaussian processes, in the case $A\to\infty$. In our results below on `effective' variable selection, the regime $A\to 0$ corresponding to small lengthscale will be particularly relevant, see Section \ref{sec:mainres1} for more on this. We note also that our results on posterior contraction will be expressed in the natural global $L^2$ loss (in contrast to a loss e.g. restricted only on active directions). 


\subsection{Deep Horseshoe Gaussian Process prior}\label{section: Deep Horseshoe Gaussian Process}

We introduce a Gaussian process prior with independent inverse--lengthscales distributed following a half-horseshoe distribution. This distribution possesses two interesting properties for our goals. Its density has a pole at $0$, which allows to `freeze' irrelevant dimensions, drawing small inverse lengthscales with high probability. It also has heavy 
 tails, so that it performs an adequate scaling on the ambient dimensions with sufficiently large probability.


{\em The single-layer case.} 
In the following, we use the real map \[\Psi\colon\ x \mapsto \left(x\wedge 1\right)\ \vee\ (-1).\] 
For $\pi$ a density function  on $\mathbb{R}_+$, consider the following prior $\Pi$ on regression functions $f$ 
  \begin{align}\label{eq: gp prior dilatation}
 A_j &\overset{\text{i.i.d.}}{\sim}\pi\\
 f\ |\ (A_1,\dots,A_d) & \sim \Psi(W^A), \nonumber
 \end{align}
 where \[W^A=\left\{W_{(A_1s_1,\dots,A_ds_d)}:\ s=(s_1,\dots,s_d)\in [-1,1]^d \right\},\] for $W$ the squared exponential process \textsf{SqExp}. For a given value of $A$, we call $W^A$ a multibandwidth Gaussian process.
 
Although the theory below can be applied to arbitrary scaling distributions $\pi$, we consider two main examples in the sequel: an exponential and a horseshoe distribution.
Let us define $\pi_{\tau}$, $\tau>0$, as the (half-) \textit{horseshoe} density (introduced in 
\cite{carvalho10})
 i.e. the density of a random variable $X_\tau$ distributed as
 \begin{align*}
 \xi &\sim C^+(0,1)\\
 X_\tau\ |\ \xi\ &\sim \mathcal{N}^+(0, \tau^2\xi^2),
 \end{align*}
 with $C^+(0,1)$ a standard half-Cauchy distribution and $\mathcal{N}^+(\mu,\sigma^2)$ the half-normal distribution of $|X|$, $X\sim\mathcal{N}(\mu,\sigma^2)$. We refer to \cite{10.1214/14-EJS962} for posterior contraction results in the case of estimation of sparse vectors, with priors based on $\pi_\tau$ and discussions on the influence of $\tau$. 
 
 When $\pi=\pi_\tau$ the horseshoe density on $\R_+$ in \eqref{eq: gp prior dilatation}, we call the above hierarchical prior the {\em Horseshoe Gaussian Process} prior and denote it $\textsf{HGP}(\tau)$.

{\em The multi-layer case.}
In order to perform inference in more complex models, we introduce a deepGP-type prior, mixing ideas from \cite{fsh21} and the just introduced prior \eqref{eq: gp prior dilatation}. 

We first place discrete priors $\Pi_q$ on the number of layers $q$ and $\Pi_d[d_1,\dots,d_q|q]$ on the successive ambient dimensions in the composition \eqref{compo}. We assume that $\Pi_q[q]>0$ and $\Pi_d[d_1,\dots,d_q|q]>0$ for any integers $q\ge 0$ and $d_i\ge 1$.

  Given $q, d_1,\ldots,d_q$, we define a random regression function $f=W_q\circ \dots\circ W_0$ where, for $i=0,\dots,q$, the map $W_i\colon I^{d_i}\to  \RR^{d_{i+1}}$ is a multivariate Gaussian process indexed by $I^{d_i}$ and to which $\Psi$ is applied element-wise. We assume that for $j=1,\dots,d_{i+1}$, the coordinates $(W_{i})_j$ are independently (accross $i$ and $j$) and identically (across $j$) distributed as a prior of the form \eqref{eq: gp prior dilatation}. Constraining the sample paths between $-1$ and $1$ ensures that the composition is well-defined. 
  
  The {\em Deep Horseshoe Gaussian Process} \textsf{Deep--HGP} is a special case of this construction where the prior on lengthscales is a horseshoe prior: it is defined as the hierarchical prior
 \begin{equation}\label{eq: deep gp prior dilatation}
 \begin{split}
 q&\sim\Pi_q \\
 d_1,\dots,d_q\ |\ q&\sim\Pi_d[\cdot|q]\\
g_{ij}\ |\ (q, d_1,\dots,d_q) &\overset{\text{ind.}}{\sim} \textsf{HGP}(\tau_i) \\
 f\ |\ (q, d_1,\dots,d_q,g_{ij}) &=g_q\circ \dots\circ g_0
 \end{split}
 \end{equation} 
for $\tau_i>0$, $i=0,\ldots, q$. In the above, $g_i=(g_{ij})_j$ with $g_{ij}:I^{d_i}\to I$, $1\le j\le d_{i+1}$. 
At the $j$--th level of the composition, there are $d_{j+1}$  coordinate functions $(g_{ij})$ distributed as the \textsf{HGP} process, each of these functions depending on $d_j$ parameters (and setting $d_{q+1}=1$ the output dimension). Note that on each layer $j$, all $d_j$ variables are present simultaneously as input to each Gaussian process component, although the scalings (distributed as horseshoe random variables) calibrate  the `strength' (or `importance') of these variables.  

In the construction \eqref{eq: deep gp prior dilatation}, the depth $q$ and dimensions $d_i$ are given prior distributions, which is perhaps the most natural Bayesian way to model these unknown quantities. However, simulating from the posterior distributions of $q$ and $d$ may be expensive, so it is common practice in this setting to take these parameters to be fixed  `large enough' values, say $q=q_{max}$ and $d_i= d_{max}$ for all $i\le q$. This leads to the following simpler prior  \textsf{Deep--HGP}$(q_{max},d_{max})$
 \begin{equation}\label{eq:dgpmax}
 \begin{split} 
g_{ij}\  &\overset{\text{ind.}}{\sim} \textsf{HGP}(\tau_i) \\
 f\ |\ (g_{ij}) &=g_q\circ \dots\circ g_0
 \end{split}
 \end{equation}
for $\ta_i>0$, $i=0,\ldots, q_{max}$ and now $g_i=(g_{ij})_j$ with $g_{ij}:I^{d_{max}}\to I$, for $1\le j\le d_{max}$ and $i>0$. Our main result on deep Gaussian processes in fixed dimensions, Theorem \ref{theorem: posterior contraction rates hdgp} below, shows that both constructions of \textsf{Deep--HGP}, with either random or fixed $q,d$, lead to near-optimal and adaptive rates for compositions, with the only restriction for  \textsf{Deep--HGP}$(q_{max},d_{max})$ being that  the `true' dimensions are indeed smaller than $q_{max},d_{max}$.

In the following, we use $q,d_i$ to denote both the parameters of the class
$\mathcal{F}_{deep}(\lambda,\mathbb{\beta},K)$ in \eqref{defcld} and the hyperparameters of the prior \eqref{eq: deep gp prior dilatation}, as the context will make it clear what we are referring to.

\section{Main results I: shallow case and ``freezing of paths"}  \label{sec:mainres1}

To gain intuition on our proposed procedure, we now build up progressively the results, starting from a standard smoothness condition in regression (no composition) where the regression function depends on an effective number of coordinates possibly smaller than the input dimension $d$.  Section \ref{sec:freeze} presents a simple oracle result while Section \ref{sec:shallowvs}
considers more precise results allowing for adaptation and growing dimension. The `deep' case of compositional structures is considered in Section \ref{sec:mainres2}. Recall that $\sigma_0^2$ is the variance of the noise in \eqref{def:rdreg} and set
\begin{equation} \label{defxi}
\xi :=  2\sigma_0^2/\sqrt{1+4\sigma_0^2}.
\end{equation}
For simplicity, in Sections \ref{sec:mainres1}
and \ref{sec:mainres2} we mostly focus on $\rho$--posteriors for given $\rho<1$ (e.g. $\rho=1/2$) and assume that $\sigma_0^2$ is known. We refer to Section \ref{sec:truepost} below for results on the standard posterior and possibly unknown noise variance.

\subsection{``Freezing of paths" for effective variable selection: a new property of scalings of Gaussian processes} \label{sec:freeze}

The first result assumes that the true regression function depends on a number $d^*$ of  
coordinates only, and that for now the indices of the active variables are {\em known}.
\begin{theorem}[freezing of paths]\label{thmtoy}  
Let $d\ge 2$ be a fixed integer and  for $K\ge 1, \be>0$, set $\cF(K):=\cF_{VS}(K,\be, d, d^*)$. Fix $\rho\in(0,1)$, let $f_0\in \cF(K)$ and suppose
\[ f_0(x_1,\ldots,x_d) = g_0(x_{i_1},\ldots,x_{i_{d^*}}), \]
with $S_0:=\{i_1,\ldots,i_{d^*}\}\subset\{1,\ldots,n\}$ and some  $1\le d^* \le d$. Let $\Pi$ be a multibandwidth prior 
 \[ f(x) \sim \Psi( W_{(a_1x_1,\ldots,a_dx_d)} ) \]  
with $W$ a $d$--dimensional \textsf{SqExp} Gaussian process with deterministic scaling parameters 
\begin{equation} \label{scale-freeze}
a_ i = 
\begin{cases}
n^{\frac{1}{2\be+d^*}}, & \quad\text{if } i\in S_0 \\
1/\sqrt{n}, & \quad \text{if } i\notin S_0
\end{cases}.
\end{equation}
Then, there exists $M>0$ such that 
 the $\rho$--posterior $\Pi_\rho[\cdot\given X,Y]$ verifies 
 \[ \sup_{f_0\in \cF(K)} E_{f_0} \Pi_\rho\left[f:\ \norm{f-f_0}_{L^2(\mu)} \ge 
 M\ \log^{1+d^*}(n)\ n^{-\frac{\be}{2\be+d^*}} \given X,Y \right]\to 0. \] 
  \end{theorem}
Theorem \ref{thmtoy} shows that by taking GP lengthscale parameters that are very small, of order $1/\sqrt{n}$, for the coordinates $j$ such that $f_0$ in fact does not depend on $x_j$, and taking lengthscales equal to the `standard nonparametric cut-off' $n^{1/(2\be+d^*)}$ (for estimating $\be$--smooth functions in dimension $d^*$) on the other coordinates, leads to an optimal minimax contraction rate $n^{-\be/(2\be+d^*)}$ for the integrated squared loss up to a logarithmic factor for the $\rho$--posterior distribution (the power in the log factor is improved in the next result). 
Inspection of the proof shows that for $i\notin S_0$, one may take $a_i$ to be any value smaller than $C/\sqrt{n}$  for some fixed $C>0$. 

The intuition behind the result is that taking a small lengthscale for coordinate $i$ `freezes' the GP path along this coordinate, making it almost constant in that coordinate, which corresponds to the limiting case $a_i=0$. Note that Theorem \ref{thmtoy} is an `oracle' result in that it assumes both $\be$ and $d^*$ (and even the indices $i_j$) to be {\em known}. Adaptive versions are considered below. While the result is somewhat expected if one sets $a_i=0$ for $i\notin S_0$ (this would correspond to a `hard' variable selection), the fact that the rate remains optimal for small but non-zero values of $a_i$ suggests that there may be room for a `soft' variable selection procedure that would allow for small $a_i$ in a data-driven way: this is the purpose of our next Theorem.


\subsection{Single layer setting: horseshoe GP} \label{sec:shallowvs}

The next statement is our main result on effective variable selection 
for (non-deep) Gaussian processes -- recall that by this we mean achieving the same rates as if active coordinates were known (not recovering the truly active ones) --. It is a non-asymptotic result that allows for dimensions varying with $n$. It is stated for an arbitrary prior $\pi$ on lengthscales. We then particularise it over the next paragraphs by stating simpler asymptotic versions and giving examples of lengthscale priors that satisfy the conditions. We consider both fixed dimensions and the case where both $d^*, d$ vary with $n$.

\begin{theorem}[Single layer, generic result]\label{thmvs}
Let $1\le d^*\le d$ be two  integers and  for $K\ge 1, \be>0$, set $\cF(K):=\cF_{VS}(K,\be, d,d^*)$. Let $\xi$ be as in \eqref{defxi} and fix $0<\rho<1$.
 Let $\Pi$ be a multibandwidth prior \eqref{eq: gp prior dilatation} with density $\pi$ on scaling parameters that satisfies
 \begin{equation} \label{condpr} 
 \left(\int_0^{\frac{\xi}{8d\sqrt{\rho n}}} \pi(a)da\right)^{d-d^*} \left(\int_{a^*}^{2a^*} \pi(a)da\right)^{d^*} \ge 2\exp(-n\rho\veps_n^2/2),
\end{equation}
where $a^*=a_n^*$ and $\veps_n$ verify $a^*\ge 1$, that $1/\sqrt{n}\le \veps_n\le 1$, and 
\begin{equation}\label{opteps}
64\xi^{-2}\geq\veps_n^2 \ge \{ B_1 {a^*}^{-2\be} \} \vee \{ B_2  {a^*}^{d^*}\log^{1+d^*}(n)/n\},
\end{equation}
where $B_1=c_1K^2 c_2^{d^*}$ and $B_2=K^2 (c_3d^{*c_4})^{d^{*}}$, with $c_1,\ldots,c_4\ge 1$ constants depending only on $\be$, $\xi$ and $\rho$.
 Then, there exists $M=M(\rho,\xi)>0$ such that,  
 for $n\ge 3$,
 \[ \sup_{f_0\in \cF(K)} E_{f_0} \Pi_\rho[f:\ \norm{f-f_0}_{L^2(\mu)} \ge M\veps_n\given X,Y]\leq \frac{1}{n\veps_n^2} + e^{-\rho n\veps_n^2}. \] 
  \end{theorem}


Theorem \ref{thmvs} gives a contraction of the $\rho$--posterior distribution at rate  $\veps_n$ around the true $f_0$, provided $n\veps_n^2$ is suitably large.  A more explicit expression of $\veps_n$ is given in the next Corollary. 

\begin{corollary}[Optimal $a^*$ and posterior rate]\label{cor: optim bandwidth and rate}
Optimising $a^*$ in \eqref{opteps}, leads to setting 
\begin{equation} \label{optastar}
 (a^*)^{2\be+d^*}=(B_1n)/(B_2 \log^{1+d^*}(n)) \vee 1.
\end{equation} 
Condition \eqref{opteps} then becomes, for $a^*$ as in \eqref{optastar},
\begin{align} 
64\xi^{-2} \ge \veps_n^2 & 
 \ge \left[ 
 B_3 (\log{n})^{\frac{2\be(1+ d^*)}{2\be+d^*}} 
n^{-\frac{2\be}{2\be+d^*}}\right] \vee \left[B_2\log^{1+d^*}(n)/n\right], \label{vepsfinal}
\end{align}
where 
$B_3 = K^2c_5^{d^*}(d^*)^{c_6d^*/(2\beta+d^*)}$,
  recalling $B_2=K^2 (c_3d^{*c_4})^{d^{*}}$, and $c_3,\ldots,c_6\ge 1$ are constants only depending on $\beta,\xi, \rho$. If additionally \eqref{condpr} holds (conditions for this are given below), then the $\rho$--posterior rate can be taken as the right-hand side of \eqref{vepsfinal}. 
\end{corollary} 
In the asymptotic regime $n\to\infty$, for $\veps_n$ taken equal to the right hand-side of \eqref{vepsfinal}, it follows that $\veps_n\to 0$ and $n\veps_n^2\to \infty$, so that provided \eqref{condpr} holds, the posterior mass in the last display of Theorem \ref{thmvs} goes to $0$, and the posterior contracts to $f_0$ at rate $\veps_n$ asymptotically. For reasonable (i.e. fixed or growing slowly with $n$) values of $d^*$, the first term in \eqref{vepsfinal} dominates and, again under \eqref{condpr}, the resulting rate $\veps_n$ goes to $0$ with $n$. Next we investigate a few examples of priors for which \eqref{condpr} holds with a resulting $\veps_n$ given by \eqref{vepsfinal}.

The proof of Theorem \ref{thmvs} is based on 
considering an oracle process defined on the $d^*$ relevant dimensions. The rates can then carry over to the full prior thanks to Condition \eqref{condpr}, which ensures that the lengthscale prior $\pi$ tunes down irrelevant dimensions, so that the difference between the processes is small with high probability. These deviations are controlled via new dimension--dependent estimates of characteristics of the squared-exponential Gaussian process (via its RKHS), see Theorem \ref{concd} and Lemmas \ref{lemma:: approx rkhs bound}--\ref{lemma bound entropy unit ball rkhs} \cite{AnnexeDGP}, coupled with concentration of measures tools. This theorem is a main theoretical novelty of the paper: a lengthscale prior which puts large mass on small values allows to `freeze' irrelevant directions, so that the overall prior behaves like a smaller-dimensional one. 

The case $A\to 0$ corresponding to a \textit{freezing-of-paths} effect has not been studied so far to the best of our knowledge; although it is conceivable that a study of the RKHS of the \textsf{SqExp} process in the small $A$ regime (similar to the case $A\to\infty$ discussed in Section \ref{sec: key ingredients}) would also lead to a proof of Theorem \ref{thmvs}, the case $A\to 0$ proves to be quite challenging: for instance, one may note that constant functions do not belong to the RKHS of the squared-exponential process; one should then find the best possible rate of approximation of constants by elements of the RKHS. This is why, in the present work, we have followed a different route by directly comparing to an oracle process, as explained in the previous paragraph. 
\\

  
 
{\em Fixed dimensions.} Let us now examine the case where the dimensions $d, d^*$ are fixed, independent of $n$. We derive conditions on two natural priors: an exponential prior, as used e.g. in \cite{aadharry08} for adaptation to smoothness, and a horseshoe prior.
 
\begin{example}[Exponential prior with fixed scaling $\la$] \label{exa:exp}
Consider $\pi(a)=\la e^{-\la a}\1_{a>0}$ an exponential prior of parameter $\la>0$. A simple calculation, see Lemma \ref{lemexp} \cite{AnnexeDGP}, shows that \eqref{condpr} is verified if 
\begin{equation} \label{condexp}
n\veps_n^2 \ge (2/\rho)\left[d\log\left(\frac{16d\sqrt{\rho n}}{\xi \la} \right)+2\la d^*a^* +\log 2\right],
\end{equation}
as long as $\la\in[1/a^*, 8d\sqrt{\rho n }/\xi]$. As a particular case, for fixed $d, d^*, K, \la$, and $a^*$ as in \eqref{optastar}, this condition is automatically satisfied for large enough $n$ if \eqref{vepsfinal} holds.
\end{example}  
 
\begin{example}[Horseshoe prior with fixed parameter $\tau$] \label{exa:hs1}
 Consider $\pi=\pi_\ta$ a horseshoe prior of parameter $\tau>0$. Then  \eqref{condpr} is verified if, setting $e_0=2/(2\pi)^{3/2}$, 
 \begin{equation} \label{condhs}
n\veps_n^2 \ge (2/\rho)\left[d\log\left(\frac{8d\sqrt{\rho n}}{\xi e_0 \ta} \right)+ d^*\log(10a^*/\ta) + \log 2\right],
\end{equation}
as long as $\ta\in[\xi/(8d\sqrt{\rho n}),a^*]$, see Lemma \ref{lem-hs-fix} \cite{AnnexeDGP}. In particular, for fixed $d, d^*,K, \ta$, and $a^*$ as in \eqref{optastar}, this condition is automatically satisfied for large enough $n$  if \eqref{vepsfinal} holds.
\end{example}
To obtain \eqref{condhs}, one simply uses that the horseshoe density is bounded from below by a constant on the integration interval; this is sensible for a fixed $\tau$ but  can be significantly improved for small $\ta$, as seen in Corollary \ref{cor-hdhs}. Corollary \ref{corfinite} is a direct consequence of above examples and Corollary \ref{cor: optim bandwidth and rate}.

\begin{corollary}[Fixed dimensions] \label{corfinite}
In the setting of Theorem \ref{thmvs}, suppose the input dimension $d$ is fixed (independent of $n$). Let $\pi$ be either the exponential prior or the horseshoe prior with fixed (independent of $n$) respective parameters $\la$ and $\ta$. 
Then for large enough $M>0$ (depending on $\be,d^*$ only), as $n\to\infty$, 
 \[ \sup_{f_0\in \cF(K)} E_{f_0} \Pi_\rho[f:\ \norm{f-f_0}_{L^2(\mu)} \ge M\veps_n\given X,Y]\to 0, \]
 where $\veps_n$ is given by 
$\veps_n= (\log{n})^{\frac{2\be(1+ d^*)}{2\be+d^*}}  n^{-\frac{2\be}{2\be+d^*}}$. In particular, the posterior distribution achieves the minimax convergence rate up to a logarithmic factor. 
\end{corollary}

An important consequence of Corollary \ref{corfinite} is that it is possible to derive a (near-)optimal rate adapting to the unknown number $d^*$ and coordinates of the active variables  with continuous priors, that is,  even without setting the scaling exactly to zero on certain coordinates (i.e. without performing a `hard model selection'). Even more surprisingly at first, such `soft model selection' is possible (at least with tempered posteriors) using a prior not putting a particularly large amount of mass near $0$, such as an exponential prior. In particular, simple priors on scalings such as exponentials or gamma distributions considered in \cite{aadharry08}  have prior mass permitting for `enough' variable selection in order for their (tempered)--posterior distribution to contract at near optimal rate, without using oracle knowledge of which coordinates are active or not. 

At this point it may seem as if effective variable selection can be achieved at no cost with just simple random scalings on coordinates. This is not (completely)  so, the reason being that the dependence on the input dimension $d$ in the convergence rate that arises from e.g. putting exponential priors on scalings is far from optimal. This can be seen from \eqref{condexp}, or similarly \eqref{condhs} for $\pi_\ta$ with fixed $\ta$, as follows. Recall that \eqref{condexp}--\eqref{condhs} are non-asymptotic conditions, so one may let $d, d^*$ depend on $n$. Suppose for instance $d=n^\delta$ for some $\delta>0$ and $d^*$ is fixed, say $d^*=1$ to fix ideas. Then \eqref{condexp} cannot hold with $\veps_n$ the lower bound in \eqref{vepsfinal}: indeed, the latter is up to a log-factor of order $n^{1/(2\be+1)}$, so is a $o(n^\delta)$ as soon as $\delta>1/(2\be+1)$,  which shows that in this setting the rate is suboptimal for large enough $\be$'s. 


The previous comments naturally make one wonder if 
effective variable selection is still feasible with a better dependence on dimensions with continuous scaling. The next section investigates this, in a setting where $d$ can go to $\infty$ with $n$.\\

 
{\em High-dimensional variable selection. 
} Let us now study the problem of inference for a small number of truly active covariates if $d$ is possibly allowed to depend on $n$. In the high-dimensional sparsity adaptation problem as in the first setting of Section \ref{subsec: model}, the work \cite{yangtokdar15} derives up to constants the minimax rate of estimation for the squared $\norm{\, \cdot\, }_{L^2(\mu)}$ loss which is 
 up to a constant factor {\em depending on} $K$, $\beta$ and $d^*$, 
\begin{equation*}
(\epsilon_n^*)^2 \asymp  n^{-\frac{2\beta}{2\beta+d^*}}+\frac{d^*}{n}\log(d/d^*).
\end{equation*}
The first term corresponds to the rate of estimation of a low-dimensional function $g\in\mathcal{C}^\beta[-1,1]^{d^*}$ and the second is the rate for the variable selection problem. Under Condition \eqref{dims} below, the first term dominates.
Note however that, as the dependence of the constants in $d^*$ is not explicit in the last display, this result allows for $d\to\infty$ and a fixed $d^*$ but not  both $d^*, d$ going to infinity.

Let us consider the sparse high-dimensional setting where $d$ can go to infinity;  we also allow the effective dimension $d^*$ to slowly grow with $n$: more precisely for some $\delta<1/2$ and $C_1, C_2>0$ suppose
\begin{equation} \label{dims}
1\le d^* \le (\log{n})^{1/2-\delta},\qquad 1\le d^*\le d\le C_1n^{C_2}.
\end{equation}
One may hope  to obtain a convergence rate that depends on the effective dimension $d^*$ only, not on $d$. In Appendix \ref{app-rate-d} \cite{AnnexeDGP}, Corollary \ref{cor-vs}, we derive a lower bound result that shows that under mild conditions on the design distribution, and if the radius of the considered H\"older ball is not too large, the minimax rate for the integrated quadratic risk is bounded below by $C_2D_2^{d^*}n^{-2\be/(2\be+d^*)}$, for constants $C_2, D_2$ independent of $d^*$. We show below that this rate can be achieved by a well-chosen horseshoe GP in the regime \eqref{dims}, up to a slowly-varying term $D_3^{d^*}(\log{n})^c$. To do so, one first determines a horseshoe scaling parameter $\tau$ for which condition \eqref{condpr} holds, and then we state the sparse high-dimensional result as Corollary \ref{cor-hdhs} (more details on optimality can be found below in Appendix \ref{app-rate-d} \cite{AnnexeDGP}, Corollary \ref{cor-vs}).



\begin{example}[Horseshoe prior with vanishing parameter $\tau$] \label{exa:hs2}
Consider $\pi=\pi_\ta$ a horseshoe prior of parameter $\tau>0$. Then  \eqref{condpr} is verified for  large enough $n$ if
\begin{equation}
10a^*e^{-n\rho\varepsilon_n^2/2d^*}  \le \tau \le 
\frac{\xi}{d^2}\frac{1}{\sqrt{\rho n}}.
\end{equation}
For $a^*, \veps_n$ as in \eqref{optastar}--\eqref{vepsfinal}, the last display holds for large enough $n$ and fixed $K$ (or more generally $K\le C^{d^*}$ for some $C>1$) if one sets for some $\delta>0$
\begin{equation*} 
\ta = \ta^*:=(n^{1+\delta} d^4)^{-1/2}.
\end{equation*}
For a proof of both claims, see Lemma \ref{lem-hs-van} \cite{AnnexeDGP}.
\end{example}  
  
 \begin{corollary}[High-dimensional horseshoe GP]\label{cor-hdhs}
 In the setting of Theorem \ref{thmvs}, suppose $d^*,d$ verify Condition \eqref{dims}.  
Let $\Pi$ be the multibandwidth prior \eqref{eq: gp prior dilatation} with horseshoe scaling density $\pi_{\ta^*}$ and $\ta^*=(n^{1+\delta} d^4)^{-1/2}$, $\delta>0$. Then, for $K\ge 1$, there exists $M=M(\xi,\rho)>0$ such that 
 \[ \sup_{f_0\in \cF(K)} E_{f_0} \Pi_\rho[f:\ \norm{f-f_0}_{L^2(\mu)} \ge M\veps_n\given X,Y]
 \to 0, 
 \]   
 as $n\to\infty$   
 where, for some constant $C$ that depends on $\beta, \xi,\rho$ only,
  \[ \veps_n^2 = K^2C^{d^*}
 (\log{n})^{\frac{2\be(1+ d^*)}{2\be+d^*}} 
n^{-\frac{2\be}{2\be+d^*}}. \] 
In particular, for $K^2\le C^{d^*}$, the rate $\veps_n^2$ is of order $n^{-2\be/(2\be+d^*)}$, up to a smaller order term at most of order $C^{2d^*}(\log{n})^{\frac{2\be(1+ d^*)}{2\be+d^*}}$. 
 \end{corollary}  
 The rate $\veps_n^2$ obtained in Corollary \ref{cor-hdhs} is optimal in the minimax sense up to the smaller order term $K^2C^{d^*}$ (up to a log factor). 
As mentioned above, as long as $K$ does not grow faster than $C_4^{d^*}$ this is a slower order term compared to the main term $n^{-2\be/(2\be+d^*)}$ (in regime \eqref{dims}) in the minimax lower bound from Corollary \ref{cor-vs} \cite{AnnexeDGP}: in this case the rate is minimax optimal up to a slower order term $C_5^{d^*}$. We also note that a growth in $C^{d^*}$ for the radius $K$ of the H\"older ball is typical for functions in H\"older spaces of dimension $d^*$, see Appendix \ref{app:hold} \cite{AnnexeDGP}, where this is checked for functions of product form.

The idea behind Corollary \ref{cor-hdhs} is that for small values of the parameter $\tau$, the horseshoe distribution becomes very `sparse' in the sense that most nonzero values are very close to $0$: this is reminiscent of the high-dimensional statistics literature for sparse models, see e.g. \cite{pashs14, pasdisc17}, where near-optimal posterior rates for horseshoe posteriors are derived in sparse settings. We now turn to a deep learning setting, where the prior process is allowed to have several Gaussian compositional layers.


\section{Main results II: deep simultaneous adaptation to structure and smoothness}    \label{sec:mainres2}
   
\subsection{Multilayer setting: Deep Horseshoe GP} \label{sec:deep}
We now consider the problem of adaptation to an unknown compositional structure, first in the fixed dimensional case.  The following result shows that, assuming the regression function can be expressed as a composition \eqref{compo}, such adaptation can be achieved with a prior mimicking this structure and organizing Gaussian processes in layers.
In particular, in the \textsf{Deep--HGP} prior, the distribution on the scalings of the individual Gaussian processes allows for adaptation to the regularity as we have seen above, but also adaptation to a sparse network of compositions.

 \begin{theorem}\label{theorem: posterior contraction rates hdgp}
 Let $\lambda=(q,d_1,\dots,d_q,t_0,\dots,t_q)$, $\mathbb{\beta}=(\beta_0,\dots,\beta_{q})$, $d\geq1$, $K\geq1$ and suppose $f_0\in \mathcal{F}_{\text{deep}}(\lambda,\mathbb{\beta},K)$.  Let $\Pi$ be the \textsf{Deep--HGP} prior \eqref{eq: deep gp prior dilatation} with fixed parameters $\tau_i>0$.
 Then, for any $0<\rho<1$, $\Pi_\rho[\cdot\given X,Y]$ contracts to $f_0$ at the rate \[\veps_n = \max_{i=0,\dots,q} (\log{n})^{\frac{\alpha_i\be_i(1+ t_i)}{2\alpha_i\be_i+t_i}} 
n^{-\frac{\alpha_i\be_i}{2\alpha_i\be_i+t_i}} \] in $\norm{\,\cdot\,}_{L^2(\mu)}$ distance, where $\alpha_i = \prod_{l=i+1}^q (\beta_l\wedge 1)$: for any $M_n \to \infty$
  \[ E_{f_0} \Pi_\rho[f:\ \norm{f-f_0}_{L^2(\mu)} \geq M_n \varepsilon_n\given X,Y]\to 0. \]
The same conclusion holds for the prior \textsf{Deep--HGP}$(q_{max},d_{max})$ as in \eqref{eq:dgpmax}, provided the parameters of $\la$ verify $q\le q_{max}$ and $d_i\le d_{max}$ for all $i\le d_{max}$. 
 \end{theorem}

Theorem \ref{theorem: posterior contraction rates hdgp} shows that the fractional posterior attains the minimax rate of convergence  of contraction  \eqref{targetrate} over the class $\mathcal{F}_{\text{deep}}(\lambda,\beta,K)$ up to the logarithmic factor $\log^{\gamma} n$, $\gamma=\max_{i=0,\dots,q}\ 2\alpha_i\be_i(1+ t_i)/(2\alpha_i\be_i+t_i)$. For simplicity, we only considered the situation where inverse-bandwidths are distributed as horseshoe random variables. As above in the fixed $d$ setting with GPs, one can derive similar results for exponential priors on scalings. However, given the benefits of the horseshoe prior in high-dimensional settings (see also below), we focus on this choice.

Theorem \ref{theorem: posterior contraction rates hdgp} can be compared to Theorem 2 of \cite{fsh21}, providing rates for a deepGP construction over a compositional functional class. The present result on the \textsf{Deep--HGP} prior now shows that adaptation to the structure can be achieved with $\rho$--posteriors without imposing a hyperprior on all the parameters describing the structure. Even if a prior can still be put on the depth and width of the composition, as in the first part of the statement,  it is enough to chose these deterministically under the mild condition above. As in our results of Section \ref{sec:shallowvs} on (single-layer) GPs, instead of imposing a `hard' selection of relevant variables on each layer, a continuous distribution on the lengthscales, with sufficient mass on small values, is enough for {\em simultaneous} adaptation to smoothness and  sparse compositional structure.
The proof of Theorem \ref{theorem: posterior contraction rates hdgp} can be found in Section \ref{proofthm3} for random $q,d$ (see Appendix  \ref{app: fixed dmax and qmax} \cite{AnnexeDGP} for the version for deterministic $q,d$).

\begin{remark}[Benign overfitting]\label{rk: fixed depth and width}
The prior with deterministic $q,d$ in the last part of the statement of Theorem \ref{theorem: posterior contraction rates hdgp} closely matches current practice in deep GPs implementation: recent works show that a depth of just a few layers already enables important gains compared to traditional (single-layer, or `shallow') GPs, see e.g. \cite{wade23} ($q=2$), \cite{zhao_et_al_21} ($q=2$ and $q=3$), \cite{saueretal23} ($q=2$). Interestingly, our result shows that `overestimating' $q$ and $d$ (i.e. choosing $q_{max}, d_{max}$ `too large' in the above statement) does not prevent one to still obtain an {\em adaptive} rate (i.e. knowledge of `true' $q,d$ is not needed). In this sense we see that a form of {\em benign overfitting} is at play, with an overfitted architecture specified by $(q_{max},d_{max})$ still leading to the optimal minimax rate with adaptation both to smoothness and structure.
\end{remark}

In view of Corollary \ref{cor-hdhs}, one naturally wonders if the \textsf{Deep--HGP} prior is able to perform adaptation to the structure and the regularity in a high-dimensional framework as well. More precisely, suppose $f_0:\left[-1,1\right]^d \to \RR$ belongs to $\mathcal{F}_{\text{deep}}(\lambda,\mathbb{\beta},K)$  with $d=d_0$ and $t_0$ possibly depending on $n$. As in \eqref{dims}, we suppose $t_0,d_0$ verify, for some $\delta<1/2$ and $C_1, C_2>0$,
\begin{equation} \label{dims2}
1\le t_0 \le (\log{n})^{1/2-\delta},\qquad 1\le t_0\le d_0=d\le C_1n^{C_2}.
\end{equation}
The next result shows that letting the horseshoe scale parameter $\tau_0$ of the first layer in the composition vanish in the \textsf{Deep--HGP} prior, while keeping the other scale parameters $\tau_i, i=1,\ldots q$ across the other layers, constant, is enough to still obtain a near-minimax rate of contraction for the fractional posterior. Choosing $\ta_0$ appropriately small (although independent of the true unknown $t_0$) allows one to obtain sparsity on the first layer, mitigating the effect of the growing input dimension $d$.

 \begin{theorem}\label{theorem: posterior contraction rates hdgp 2}
 Under the same assumptions as in Theorem \ref{theorem: posterior contraction rates hdgp}, now suppose $t_0$ and $d_0$ satisfy Condition \eqref{dims2}, and that $d_1,\ldots, d_q$ are fixed, non $n$--dependent integers. 
Let $\Pi$ be the \textsf{Deep--HGP} prior with $\tau_0=d^{-2}n^{-1}$ and $\tau_i>0, i=1,\ldots, q$ be fixed.
For any $0<\rho<1$, there exists $c_1,c_2$ depending on $\beta_0$ and $\rho$ such that the $\rho$--posterior contracts to $f_0$ at the rate \[
\veps_n^2 = \left[c_1K^2c_2^{t_0}
 (\log{n})^{\frac{2\be_0\alpha_0(1+ t_0)}{2\be_0\alpha_0+t_0}} 
n^{-\frac{2\be_0\alpha_0}{2\be_0\alpha_0+t_0}}\right] \vee \max_{i=1,\dots,q} (\log{n})^{\frac{2\alpha_i\be_i(1+ t_i)}{2\alpha_i\be_i+t_i}} 
n^{-\frac{2\alpha_i\be_i}{2\alpha_i\be_i+t_i}}\] in $\norm{\,\cdot\,}_{L^2(\mu)}$ distance, that is, for any $M_n\to\infty$,
 \[ E_{f_0} \Pi_\rho[f:\ \norm{f-f_0}_{L^2(\mu)} \geq M_n \veps_n \given X,Y]\to 0. 
 \]
 \end{theorem}

To the best of our knowledge, Theorem \ref{theorem: posterior contraction rates hdgp 2} is the first result on deep methods in high-dimensional regression where both the input dimension $d_0=d_0(n)$ and first effective dimension $t_0=t_0(n)$ are allowed to grow with $n$. It combines both the ability of the horseshoe prior to select relevant dimensions in the input space and its ability to perform model selection in presence of a compositional parameter. It is particularly interesting given that these methods are most often applied to high-dimensional problems.

Compared to the condition on $\tau$ in Corollary \ref{cor-hdhs}, the preceding result requires a smaller $\tau_0$. In the present setting, given the flexibility of the sample paths from deep Gaussian processes, it is necessary to `stabilize' each GP to avoid `wild' behavior. From a technical point of view, it translates into more restrictive prior mass conditions for these GPs and the need for more efficient variable selection. This is achieved with a horseshoe prior that is more peaked near $0$, given the choice of the smaller parameter.

\begin{remark}
    In Theorem \ref{theorem: posterior contraction rates hdgp 2}, we let the dimensions indexed by $i=0$ on the first layer of the composition to possibly grow with $n$. Extending this result to a situation where $d_i, i>0$ can also grow with $n$ is not straightforward. Indeed, inspection of the proof of Theorems \ref{theorem: posterior contraction rates hdgp} and \ref{theorem: posterior contraction rates hdgp 2} shows that the rate involves a multiplicative factor $\sum_{i=1}^{q+1} d_i$ whose dependence on the inner dimensions of the composition is linear and thus prevents polynomially growing $d_i$'s. As this factor does not involve $d_0$, it still allows for $t_0, d_0$ as in \eqref{dims2}. 
\end{remark}

\subsection{Results for standard posteriors} \label{sec:truepost} 
We now derive results for the posterior distribution (that is $\rho=1$ in the fractional posterior). To do so one keeps the same prior on the regression function $f$ as in the previous results, the only difference being that we now also put a prior on the noise variance $\sigma^2$. The following Proposition \ref{proppost} is inspired by an idea of The Thien Mai  \cite{mai24}, who in high-dimensional regression under sparsity, notices the link between standard posterior and fractional posterior for an updated prior in that sparse setting. Here, in the different context of nonparametric regression, we use a different $n$-dependent prior on $\si^2$. The model now allows for possibly unknown noise variance: the observations $(Y_i,X_i)$ are an independent sample
\begin{equation} \label{modelsi}
 Y_i = f(X_i) + \sigma\eta_i,
\end{equation} 
with $X_i\sim \mu$ and $\eta_i\sim\cN(0,1)$, where we treat both $f$ and the variance parameter $\sigma^2$ as unknown. We denote by $P_{f,\si^2}$ the distribution of $(X_1,Y_1)$ from model \eqref{modelsi} given $f,\si^2$.

Let $\Pi$ be a prior on $(f,\sigma^2)$ defined as a product $\Pi_f\otimes \pi_{\si^2}$. We take $\pi_{\si^2}=\pi_{\sigma^2,n}$  to be, for some fixed $b\in(0,1)$,
\begin{equation}\label{priorsig}
\pi_{\si^2} = \text{Gamma}\Big( \Big\{\frac{1-b}{2}\Big\}n+1,b \big),
\end{equation}
where $\text{Gamma}(a_1,a_2)$ denotes a Gamma distribution with shape parameter $a_1>0$ and rate parameter $a_2>0$, of density proportional to $x\to x^{a_1-1}e^{-a_2 x}$ on $(0,\infty)$. This prior induces a posterior distribution $\Pi[\cdot\given (X,Y)]$ jointly on $f$ and $\si^2$, and the next result examines the  behaviour of its marginal on $f$, that is $\Pi[f\in \cdot\,,\,\sigma^2\in \RR^+\given (X,Y)]$. 



\begin{proposition}  \label{proppost}
Let $\Pi$ be a prior on $(f,\si^2)$ of product form $\Pi_f\otimes \pi_{\si^2}$, with $\pi_{\si^2}$ given by \eqref{priorsig}. Suppose, for some rate $\veps_n=o(1)$ with $n\veps_n^2\to\infty$ and a constant $D>0$, that
\begin{equation}\label{eq:small ball full post} \Pi_f[ \|f-f_0\|_\infty \le \veps_n ] \ge  e^{-Dn\veps_n^2}.\end{equation}
Then, for a constant $M>0$ large enough, as $n\to\infty$, for $D_b$  the $b$--R\'enyi divergence, the marginal posterior distribution on $f$ given observations from \eqref{modelsi} verifies
\[ E_{f_0,\sigma_0^2}\Pi\left[ D_b(p_{f,2\si_0^2},p_{f_0,2\si_0^2}) \le M \veps_n \given (X,Y)\right] \to 1. \]
If $\Pi[\|f\|_\infty\le K\given (X,Y)]=1+o(1)$ for some fixed constant $K>0$, we also have 
\[ E_{f_0,\sigma_0^2}\Pi\left[ \|f-f_0\|^2 \le M \veps_n \given (X,Y)\right] \to 1. \]
\end{proposition} 
This result shows that, modulo defining an appropriate prior on $\si^2$, results for fractional posteriors (that effectively only require the prior mass condition as in the statement), also hold for the (marginal in $f$ of the) classical posterior; see Corollary \ref{cor: from frac to actual post} below. Proposition \ref{proppost} is also of independent interest, as it holds for any choice of prior on $\Pi_f$ in random design regression; its proof can be found in the Appendix \cite{AnnexeDGP}, Section I. We also note that sampling from the marginal posterior on $f$ from the above augmented prior does not add any fundamental difficulty, see Section \ref{sec:disc}.


\begin{corollary}\label{cor: from frac to actual post}
Provided the respective priors $\Pi_f$ are replaced with augmented priors $\Pi_f\otimes \pi_{\sigma^2}$, with $\pi_{\sigma^2}$  the  prior \eqref{priorsig} on $\sigma^2$, the results stated in Theorems \ref{thmtoy}–\ref{theorem: posterior contraction rates hdgp 2} still hold for the corresponding standard marginal posterior in $f$ (instead of the fractional posterior corresponding to $\Pi_f$ in these statements). Specifically, the nonparametric convergence rates derived in Theorems \ref{thmtoy}–\ref{theorem: posterior contraction rates hdgp 2} remain valid under these modifications.
\end{corollary}

To prove this result, note that Theorems \ref{thmtoy}–\ref{theorem: posterior contraction rates hdgp 2} are established by proving the corresponding inequality \eqref{eq:small ball full post}, as outlined in the preliminaries of Section \ref{sections: proof rho post} (the slightly different constants can be accommodated by checking the condition \eqref{eq: supball condition} therein for $\veps_n'=R\veps_n$ and a large enough constant $R>0$). Given that the prior distributions have samples almost surely bounded by $1$, Corollary \ref{cor: from frac to actual post}  follows from Proposition \ref{proppost}.

\section{Discussion and open questions} \label{sec:disc}


In this work, we provide theoretical guarantees for the convergence of  posterior distributions using deep Gaussian process priors. One key insight is in the role played by lengthscale parameters: not only do these enable adaptation to smoothness, but they can also at the same time perform an effective variable selection (adaptation to sparsity) 
 by `freezing' the Gaussian process paths in suitable directions, a point relevant also for standard (non-deep) Gaussian processes; this has not been recognised so far in the literature to the best of our knowledge. The fact that adaptation to smoothness and structure can be performed {\em simultaneously} is particularly appealing computationally in that there is no need to include a model selection part in the building of the prior (if that was the case, posterior sampling would require to have access to the posterior distribution over models, which is often costly to implement). The obtained deepGP prior is then simple enough so that it corresponds to recently proposed algorithms, see below for more on this.

 We also derive new results on deep methods for high-dimensional input spaces, and on the way obtain explicit dimension-dependent constants for the characteristics of the involved GPs.

{\em The use of fractional posteriors.}  Many of our results consider fractional posterior distributions, where the parameter $\rho$ can be taken to be any constant in $(0,1)$. The main reason for this is of technical nature: in order for one to use the general theory of convergence of Bayesian posteriors in \cite{ggv00}, one needs to build sieve sets, capturing most prior mass, whose entropy or `complexity' is well controlled.
 However, especially in complex settings such as deep learning models, sieves can be difficult to construct, in particular since the probability of the complement of the sieve is  required to have a form of exponentially fast decrease, with at the same time the requirement to control the entropy of the sieve set. This difficulty leads \cite{fsh21} to condition sample paths of Gaussian processes to verify certain smoothness constraints. This can be avoided using $\rho$--posteriors, since convergence of these can be guaranteed under prior mass conditions only 
\cite{10.1214/18-AOS1712, MR3202636, alik22, tzhang2006}, so we do not need to condition on boundedness of derivatives in our prior construction. This is an advantage also computationally, as adding more conditioning constraints may typically slow down MCMC samplers. 
 
Beyond fractional posteriors, we have also obtained results for the standard posterior ($\rho=1$) and avoid the just-mentioned technical difficulties: the idea then has been to use an augmented prior that also models the noise variance. 
For the original prior (without a prior on $\sigma^2$) one can conjecture that the standard posterior also achieves optimal rates; although it seems delicate to prove this using the current tools available for proving posterior convergence, given that construction of sieves (while keeping the prior simple) looks particularly difficult, it is conceivable that, at least in say regression settings, one may be able to state an adapted version of the generic theorem of \cite{ggv00}. Although beyond the scope of the present contribution, one may note that promising results are obtained in \cite{ac23}, where in regression with heavy-tailed priors both standard and fractional posteriors are shown to converge at the same rate up to constants, for a prior for which the construction of sieves seems also presently out of reach, which suggests that posterior convergence for $\rho=1$ under  prior mass conditions only is not exceptional. 
  
 On the other hand, we argue that, at least for the set of applications of Bayesian (possibly tempered) posteriors considered here, one does not loose much with fractional posteriors, except slightly in the constants in the convergence rate: here as $\rho$ is fixed (and can be chosen e.g. to be $\rho=1/2$) we did not keep the dependence in $\rho$ in the constants, but it is shown in \cite{alik22} that nonparametric rates arising from $\rho$--fractional posteriors are typically the same as for the usual posterior but with effective sample size $n'=n\rho$; for $\rho=1/2$ the constant in terms of $\rho$ in the rate is not particularly large then, so is not a main concern. Also, regarding sampling algorithms in practice, most sampling methods such as MCMC are of similar difficulty with the fractional or the original likelihood, so this is not a main concern computationally (we discuss below sampling from deepGP fractional posteriors from Example \ref{exa:exp}). One loses, though, the interpretation of the posterior as a conditional distribution and possibly efficiency for $\sqrt{n}$--estimable parameters that comes with the Bernstein--von Mises theorem, which will not hold as such for $\rho$--posteriors but again typically with a variance inflated by $1/\rho$, see \cite{alik22} for more discussion. 
 However, again, this is not a main concern here, as we are interested in estimation rates up to constants. Another interesting topic not covered in the present work is uncertainty quantification. It should be possible to prove that, modulo certain (unavoidable) structural conditions on the regression function such as self-similarity, certain credible sets from the deepGP (possibly, fractional) posterior are also confidence sets. In case one uses a fractional posterior, this may slightly increase the radius compared to the classical posterior, although, we expect, not in a significant way if $\rho$ is kept far away from $0$; we refer to \cite{ac23} for an empirical study of the influence of $\rho$ on nonparametric credible sets.

 \vspace{.5cm}
  
\noindent{\em Implementation.} Although our main focus here is on theoretical guarantees, we note that sampling from the deep Gaussian process fractional posteriors with exponential priors on GP lengthscales (Example \ref{exa:exp}) is readily available using the R package {\tt deepgp} \cite{deepgpR}. The later provides MCMC samples for standard posteriors ($\rho=1$) using Vecchia approximation; one can similarly obtain MCMC samples from fractional posteriors with any $\rho<1$ using the following remark. Note that using a fractional likelihood with a given $\rho\in(0,1)$ to form the fractional posterior in Gaussian regression with independent errors $\cN(0,\sigma^2)$ is equivalent to using the standard likelihood in the case the errors are {\em misspecified as} independent $\cN(0,\sigma^2/\rho)$. Since posterior sampling is conditional on the observed values, and the  {\tt deepgp} package allows for specifying a given noise level, it suffices to specify it to the misspecified value $\sigma^2/\rho$ (while data will truly be generated with errors of variance $\sigma^2$). We also note that for the implementation of the augmented prior $\Pi_f\otimes \pi_{\si^2}$ discussed in Section \ref{sec:truepost}, it is enough to sample from the marginal posterior in $f$, and {\tt deepgp} allows for a Gamma prior on $\sigma^2$. 
We refer to \cite{saueretal23} for illustrations and details on the sampling schemes (we note that it should also be possible to modify the code, which presently allows for Gamma priors, to include horseshoe priors on lengthscales as in Examples \ref{exa:hs1}--\ref{exa:hs2}).

Considering fixed deterministic composition depth $q$ and width $d$ (in the spirit of Remark \ref{rk: fixed depth and width}
)
 and given lengthscale parameters, the prior considered in the present work (but without $\Psi$) coincides with that considered in the paper \cite{damianou13} introducing deepGPs, where the kernel is termed ARD (Automatic Relevance  Determination) and the lengthscales are called weights. In \cite{damianou13}, the weights/lengthscales are then calibrated using a variational approach. Many different posterior approximating schemes for deepGPs have been proposed over the last few years, using in particular variational approximations; we refer to \cite{saueretal23} for an overview and discussion. Obtaining theoretical guarantees for these different approaches, in particular for the Deep HGP posteriors introduced here, is an interesting avenue for future work.\\ 

%


\section{Proof of the main results}\label{sections: proof rho post}

We denote by $\nu$ the spectral measure of the squared-exponential \textsf{SqExp} process. Let us recall that this process is stationary with covariance $K(s,t)=\exp(-\|s-t\|^2)=k(s-t)$ and that by Bochner's theorem $k(t)=\int e^{-t\psg u,t \psd} \nu(du)$; in particular it follows that $\nu$ has Lebesgue density $u\to \exp(-\|u\|^2/4)/(2^d\pi^{d/2})$.

{\em Preliminaries: reducing the problem to a prior mass condition.}
 Given a rate $(\varepsilon_n)$, Theorem \ref{bat_result} and Proposition \ref{prop: sup in KL} in Appendix \ref{app: Frac post} \cite{AnnexeDGP} ensure that if $\Pi$ satisfies, for $f_0\in C[-1,1]^d$ and $\rho\in (0,1)$,
 \begin{equation}\label{eq: supball condition}
     \Pi\left[\norm{f-f_0}_{\infty} \leq \frac{2\sigma_0^2}{\sqrt{1+4\sigma_0^2}}\varepsilon_n\right] \geq e^{-n\rho \varepsilon_n^2},
 \end{equation} 
then the fractional posterior is such that, for $D_\rho$ the $\rho$--R\'enyi divergence as in \eqref{dal},
 \[
E_0\Pi_\rho \left( \eta: \: \frac{1}{n}D_{\rho}(p_{\eta}^n, p_{\eta_0}^n) \ge 4 \frac{\rho\eps_n^2}{1-\rho}   \given  X \right) \le e^{- n\rho\eps_n^2}  + (n\eps_n^2)^{-1}.
 \]
Let us now focus on the considered regression model. Assuming that the regression functions we consider are bounded, say $\norm{f}_\infty, \norm{g}_\infty \leq 1$, using  $1-x\leq -\log x$ and $1-e^{-x}\geq xe^{-x}$ both for $x\ge 0$ and the additivity of the Rényi divergence for densities of independent observations, it follows that
 \begin{align}\label{eq: from renyi to l2}
\frac1n D_\rho(p_f^{\otimes n},p_g^{\otimes n})&=-\frac{1}{1-\rho}\log \int e^{-\frac{\rho-\rho^2}{2\sigma_0^2}(f-g)^2}d\mu
 \ge \frac{1}{1-\rho}\left[1- \int e^{-\frac{\rho-\rho^2}{2\sigma_0^2}(f-g)^2}d\mu\right]\nonumber\\
&\ge  \frac{\rho}{2\sigma_0^2} e^{-2\frac{\rho-\rho^2}{\sigma_0^2}}\norm{f-g}^2_{L_2(\mu)}.
\end{align}
We note that under the regularity assumptions on $f_0$ and with the use of the `link' function $\Psi$, the boundedness assumption is satisfied for both  $f_0$ and a draw $f$ from the posterior in our different theorems.
Consequently, in the following, the proofs consist in proving  \eqref{eq: supball condition} for $f_0$ as in the different statements and for the different priors considered. This will imply
 \[
E_0\Pi_\rho \left( f: \: \norm{f-g}^2_{L_2(\mu)} \ge \frac{8\sigma_0^2}{1-\rho}e^{2\frac{\rho-\rho^2}{\sigma_0^2}}\eps_n^2   \given  X \right) \le e^{- n\rho\eps_n^2}  + (n\eps_n^2)^{-1},
 \]
which suffices to conclude.
 
 \subsection{Proof of Theorem \ref{thmtoy}}

The proof of the theorem is a (simpler) variation on the proof of  Theorem \ref{thmvs} to follow: therein, it suffices to work with the fixed values of scaling parameters specified in \eqref{scale-freeze}. For easy reference the precise argument is given below the end of the proof of Theorem \ref{thmvs}.
 

\subsection{Proof of Theorem \ref{thmvs}}
  
Given $A\in \R_+^d$, let us denote by $\varphi_{f_0}^{A}$ the concentration function of the Gaussian process $W^A$ as in \eqref{eq: gp prior dilatation} and its RKHS by $\mathbb{H}^A$. To derive the result, we  prove below that \eqref{eq: supball condition} is satisfied, since any $f_0\in \mathcal{F}(K)$ satisfies $\norm{f_0}_\infty\leq 1$  and $\norm{\Psi(W^A)}_\infty\leq 1$ by definition. Also, since $\norm{f_0-\Psi(W^A)}_\infty\leq \norm{f_0-W^A}_\infty$, we bound this last quantity instead as it then implies \eqref{eq: supball condition} for $f=\Psi(W^A)$.

Since we assume $f_0\in\cF_{VS}(K,\beta,d,d^*)$, we have $f_0(x_1,\dots,x_d)=g_0(x_{i_1},\dots,x_{i_{d^*}})$ for some $g_0\in \cF(K,\be,d^*)$.  
Let us set, for  $S_0=\{1,\ldots, d \} \setminus \{i_1,\dots,i_{d^*}\}$ and $\xi = 2\sigma_0^2/\sqrt{1+4\sigma_0^2}$, 
 \[
I_i=
\begin{cases}
[0,\xi/(8d\sqrt{\rho n})],&\ \text{if }i \in S_0,\\
[a^*, 2a^*],&\ \text{otherwise},
\end{cases}
\]
for  $a^*\in[1,n]$ as in the statement of the Theorem. For a given vector $A$, let us introduce $\Tilde{A}=(\Tilde{A}_1,\ldots,\Tilde{A}_d)$ with $\Tilde{A}_i=A_i\mathds{1}_{i\in\left\{i_{1},\dots,i_{d^*}\right\}}$ for  $i=1,\ldots,d$. 
For any $\veps>0$,
\begin{align*} 
\lefteqn{\Pi\left[f:\ \norm{f-f_0}_\infty <\xi\veps \right]  = P\left[\norm{W^A-f_0}_\infty<\xi\varepsilon  \right]} & \\
& \quad \ge \int_{I_1} \dots \int_{I_d} P\left[\norm{W^A-f_0}_\infty<\xi\varepsilon \given A \right] \prod_{i=1}^d \pi(A_i) dA_i. 
\end{align*}
One may now split the contribution of $A_i$'s into subsets of indices as follows
 \begin{align*}
P\left[\norm{W^A-f_0}_\infty<\xi\varepsilon \given A \right] \ &\ge \ P\left[\norm{W^{\Tilde{A}}-f_0}_\infty<\xi\varepsilon/2\,,\, \norm{W^{\Tilde{A}}-W^A}_\infty<\xi\varepsilon/2 \given A \right]\\
 &\ge \  P\left[\norm{W^{\Tilde{A}}-f_0}_\infty<\xi\varepsilon/2  \given A \right]
  - P\left[\norm{W^{\Tilde{A}}-W^A}_\infty>\xi\varepsilon/2  \given A \right].
 \end{align*}
In what follows we bound from below and above respectively the quantities, for $\eta=\xi\veps/2$,
\begin{align}
P_1(\eta) & =  \int_{I_1} \dots \int_{I_d}  P\left[\norm{W^{\Tilde{A}}-f_0}_\infty<\eta  \given A \right]\prod_{i=1}^d \pi(A_i) dA_i,  \label{pun} \\
P_2(\eta) & = \int_{I_1} \dots \int_{I_d}  P\left[\norm{W^{\Tilde{A}}-W^A}_\infty>\eta  \given A \right]\prod_{i=1}^d \pi(A_i) dA_i. \label{pdeux}
\end{align}
In the bounds $P_1, P_2$ to follow, we use that the involved scaling parameters $A_i$ belong to the respective intervals $I_i$  defined above.

Starting with \eqref{pun}, denote  $A^*=\left(A_{i_1},\dots,A_{i_{d^*}}\right)\in \R_+^{d^*}$. Conditionally on the $A_i$'s, the Gaussian process $W^{\Tilde{A}}$, interpreted as a process on variables indexed by $i_1,\ldots,i_{d^*}$ only, has the same distribution as the Gaussian process $W^{A^*}$ on  $\left[-1,1\right]^{d*}$ (they are both centered with same covariance kernel; in slight abuse of notation we denote also $W$ for the process in the $d^*$--dimensional space). 
Since $W^{A^*}$ is independent of $A_i$ for $i\in S_0$, for  $\eta>0$,
\[ P\left[ \norm{W^{\Tilde{A}}-f_0}_\infty <\eta \given A \right] = P\left[ \norm{W^{A^*}-g_0}_\infty <\eta \given A^* \right].
\] 
The term $P_1$ can then be bounded from below by
\begin{align*} 
P_1(\eta) & =  \prod_{i\in S_0} \pi(I_i) \cdot \int_{I_{i_1}} \dots \int_{I_{i_{d^*}}} 
P\left[\norm{W^{A^*}-g_0}_\infty < \eta \given A^* \right] \prod_{j=1}^{d^*} \pi(A_{i_j}) dA_{i_j} \\
& \ge
\prod_{i\in S_0} \pi(I_i)  \cdot \int_{I_{i_1}} \dots \int_{I_{i_{d^*}}} 
e^{-\vphi_{g_0}^{A^*}(\eta/2)} \prod_{j=1}^{d^*} \pi(A_{i_j}) dA_{i_j},
\end{align*}    
where we use Lemma \ref{lem: concentration} to bound from below the probability in the display.

 
We now use Theorem \ref{concd} applied to $g_0$, a function with input dimension $d^*$. We set $\veps=\veps_n$ to be chosen so that $\eta=\xi\veps_n/2$. Suppose,
\begin{alignat}{2} 
8\geq&\quad \xi\veps_n && \ge 4\cC_1
K(a^*)^{-\be}, \label{condieps}\\
\ &\rho n\veps_n^2/2 && \ge \cC_2
K^2  (a^*)^{d^*} + (C{d^*}^{c}a^*)^{d^*}\log^{1+d^*}(8a^*/\veps_n), \label{condiepsb}
\end{alignat}
where $\cC_i=\cC_i(\be,d^*), i=1,2$ are constants of the form $c_i C_i^{d^*}$ as in the statement of Theorem \ref{concd}. 
Up to making $\cC_1, \cC_2$ larger, one can always assume $\cC_1, \cC_2\ge 1$. Below we will also use that if $\veps_n\le 1, a^*\ge 1$ verifying the last display exist, then $a^*\le n$. Indeed, the last term in the second inequality is bounded from below by $(C{d^*}^{c}a^*)^{d^*}\log^{1+d^*}(8a^*)$. As $8a^*\geq e$ one must have $(C{d^*}^{c}a^*)^{d^*}\le\rho n\veps_n^2/2$ where $C\ge 1$, so that $a^{*}\le (n\veps_n^2)^{1/d^*}\le n$.  

By Theorem \ref{concd} we have $\vphi_{g_0}^{A^*}(\veps_n) \le \rho n\veps_n^2/2$, uniformly for $a^*\le A_{i_j}\le 2a^*$. One deduces
\[ P_1(\xi \veps_n/2) \ge \prod_{i\in S_0} \pi(I_i)  \cdot e^{-\rho n\veps_n^2/2} \cdot \prod_{i\notin S_0} \pi(I_i). \] 
Let us now deal with the term $P_2$ in \eqref{pdeux}. First one notes that, given A,
\[ X :=W^{\Tilde{A}}-W^A \] 
is a centered Gaussian process.  In order to bound it, one first computes 
 \begin{align*}
 \sigma^2 &= \underset{t\in [-1,1]^d}{\sup} \mathbb{E}[X^2(t)\given A]
 =  \underset{t\in [-1,1]^d}{\sup}  2\left(1-e^{-\sum_{i\in S_0}A_i^2 t_i^2}\right)\\
 &= 2\left(1-e^{-\sum_{i\in S_0}A_i^2 }\right)\leq 2 \left(d-d^*\right)\  \underset{i\in S_0}{\max}\ A_i^2,
 \end{align*}
using $e^{-u}\geq 1-u$ for any $u$ and $\left|S_0\right|=d-d^*$.
Setting $M\coloneq e^{-\sum_{i\in S_0}A_i^2s_i^2}$, $N\coloneq e^{-\sum_{i\in S_0}A_i^2t_i^2}$, $P\coloneq e^{-\sum_{i\not\in S_0}A_i^2(s_i-t_i)^2}$, $Q\coloneq e^{-\sum_{i\in S_0}A_i^2(s_i-t_i)^2}$,
  \begin{align*}
D\left(s,t\right)^2 &\coloneqq \mathbb{E} [\left|X(s)-X(t)\right|^2\given A]\\
 &=  \mathbb{E}\left[ \left(W^{\Tilde{A}}(t)-W^A(t)-W^{\Tilde{A}}(s)+W^A(s)\right)^2 \given A\right]\\
 &= 4 + 2PM - 2PQ - 2M -2N + 2PN -2P\\
 &= 2\left(1-Q\right) + 2\left(1-P\right)\left(1+Q-M-N\right).
 \end{align*}
 For $s,t\in[-1,1]^d$, using again $e^{-u}\geq 1-u$,
 \[ 1-Q\leq \left|s-t\right|_\infty^2\ \left(d-d^*\right)\ \underset{i\in S_0}{\max}\ A_i^2,\quad 1-P\leq \left|s-t\right|_\infty^2\ d^*\ \underset{i\notin S_0}{\max}\ A_i^2.\]
For any $x,y,z\ge 0$, we have $1+e^{-z}-e^{-x}-e^{-y}\le 2 - e^{-x}-e^{-y}\le x+y$, using the inequalities $e^{-z}\le 1$ and $e^{-x}\ge 1-x$. Deduce
\[ 1+Q-M-N \le \sum_{i\in S_0} (s_i^2+t_i^2) A_i^2  \leq  2\left(d-d^*\right)\underset{i\in S_0}{\max}\ A_i^2.\]
Combining the previous bounds one obtains, for any $s,t\in [-1,1]^d$,
\[ D(s,t)^2\leq  2\left|s-t\right|_\infty^2\left(d-d^*\right)\underset{i\in S_0}{\max}\ A_i^2\ \left[1+ 2d^*\ \underset{i\notin S_0}{\max}\ A_i^2\right]. \]
On the other hand, using $D(s,t)^2\le 2\mathbb{E}[X(s)^2+X(t)^2\given A]\le 4\sigma^2$ for any $s,t$, we have
\[ \underset{s,t\in[0,1]^d}{\sup}\ D(s,t) \leq 2\sigma. \] 
According to Lemma \ref{lemma: dudley bound} \cite{AnnexeDGP}, since $X(0)=0$ almost surely, we have
\[ \mathbb{E} \left[ \norm{X}_\infty \given A \right] \le 4\sqrt{2}\int_0^{\sigma} \sqrt{\log\ 2N\left(\epsilon, [-1,1]^d,D\right)}d\epsilon.\]
Writing $Z^2=2\left(d-d^*\right)\underset{i\in S_0}{\max}\ A_i^2\ \left[1+ 2d^*\ \underset{i\notin S_0}{\max}\ A_i^2\right]$, this quantity is upper bounded by
\[ 4\sqrt{2}\int_0^{\sigma} \sqrt{\log\ 2N\left(\epsilon/Z, [-1,1]^d,|\cdot|_\infty\right)}d\epsilon \leq 4\sqrt{2d}\int_0^{\sigma} \sqrt{\log\ \frac{4Z}{\epsilon}}d\epsilon,\]
using $2N(\eta,[-1,1]^d,|\cdot|_\infty)\le (4/\eta)^d$ for $\eta\le 1$,  which is the case here since $\sigma/Z\le 1$ (see also below).
By a change of variable $v = \sqrt{\log(4Z/\epsilon)}$ 
, the upper bound is, with $C'=32\sqrt{2}$,
\begin{align*}
4\sqrt{2d}\int_{\sqrt{\log\ \frac{4Z}{\sigma}}}^{\infty} 8Z v^2 e^{-v^2} dv = C' \sqrt{d} Z \int_{\sqrt{\log\ \frac{4Z}{\sigma}}}^{\infty} v^2 e^{-v^2} dv.
\end{align*}
Integrating by parts, $\int_a^{\infty} v^2 e^{-v^2} dv = a e^{-a^2}/2 + \int_a^{\infty} e^{-v^2} dv/2$. For $a\ge 1$ we have $\int_a^{\infty} e^{-v^2} dv\le \int_a^{\infty} v^2 e^{-v^2} dv$ so that  $\int_a^{\infty} v^2 e^{-v^2} dv\le a e^{-a^2}$. Let us apply this to the previous bound, noting that $\log4Z/\sigma\ge \log 4 \ge 1$, using the upper bound on $\sigma$ obtained above and the definition of $Z$. One obtains, using that $\sigma\to \sigma\sqrt{\log(4Z/\sigma)}$ is increasing on $[0,4Z/\sqrt{e}]$, and that $\sigma^2\le\bar{\sigma}^2:=2(d-d^*)\max_{i\in S_0} A_i^2\le Z^2\le (4Z/e^{1/2})^2,$
\[  \mathbb{E} \left[ \norm{X}_\infty \given A \right] \le 8\sqrt{2d}\sigma \sqrt{\log 4Z/\sigma} 
 \le 16 d\ \underset{i\in S_0}{\max}\ A_i\  \sqrt{\frac12 \log\left( 16[ 1+ 2d^*\ \underset{i\notin S_0}{\max}\ A_i^2]\right)}.\]
Assuming $\max_{i\notin S_0} A_i^2\ge 1$, we obtain for some universal $c_1>0$,
\[  \mathbb{E} \left[ \norm{X}_\infty \given A \right] \le c_1 d \cdot \underset{i\in S_0}{\max}\ A_i \cdot
\sqrt{\log\left(1+ 2d^*\ \underset{i\notin S_0}{\max}\ A_i^2\right)}.
\]
One can now use Gaussian concentration to control the deviations of $\|X\|_\infty$ from its expectation. Suppose
\begin{equation} \label{condieps2}  
\veps_n \ge 4\xi^{-1}c_1 d \cdot \underset{i\in S_0}{\max}\ A_i \cdot
\sqrt{\log\left(1+ 2d^*\ \underset{i\notin S_0}{\max}\ A_i^2\right)}.
\end{equation}
Combining Lemma \ref{lemma: gaussian sup concentration} \cite{AnnexeDGP} and the above bound on $ \mathbb{E} \left[ \norm{X}_\infty \given A \right]$ gives
 \begin{equation}\label{eq: Dudley ineq} 
 P\left[\norm{X}_\infty>\xi \varepsilon_n/2 \given A \right]
 \le P\left[\norm{X}_\infty- \mathbb{E}\norm{X}_\infty >\xi \varepsilon_n/4 \given A \right]  
 \le e^{- \xi^2 \varepsilon_n^2/32\sigma^2}.
 \end{equation}
Recall that, for $A_i\in I_i$,
\begin{equation*}
 \max_{i\in S_0} A_i \le \xi/(8d\sqrt{\rho n}),
\end{equation*} 
which gives $\sigma^2\le \xi^2/(32\rho n)$, so that the last but one display is bounded from above by $e^{-\rho n\veps_n^2}$, uniformly for $A_i$ in the corresponding interval $I_i$. One deduces
\[ P_2(\xi \veps_n/2) \le e^{-\rho n\veps_n^2}  \prod_{i=1}^d \pi(I_i).
 \]
Combining this with the obtained lower-bound on $P_1(\xi \veps_n/2)$, one gets, using $e^{-\rho n\veps_n^2/ 2}\ge 2e^{-\rho n\veps_n^2}$ if $n\veps_n^2\ge 1/4\rho$, that $P_1(\xi \veps_n/2)-P_2(\xi \veps_n/2)\ge e^{-\rho n\veps_n^2/2}\prod \pi(I_i)/2$, so that
\begin{equation}\label{eq: final lower bound} \Pi\left[f:\ \norm{f-f_0}_\infty <\xi \veps_n \right] \ge e^{ -\rho n\veps_n^2/2} \prod_{i=1}^d \pi(I_i)/2\geq e^{ -\rho n\veps_n^2}, \end{equation}
where we used \eqref{condpr}.
 
Let us now optimise in terms of $\veps_n$ verifying the conditions \eqref{condieps}--\eqref{condiepsb}--\eqref{condieps2}. 
Since
\begin{equation*}
 \max_{i\in S_0} A_i \le \xi/(8d\sqrt{\rho n}),\qquad \max_{i\notin S_0} A_i \le n
\end{equation*}
for $A_i\in I_i$, we have that \eqref{condieps2} holds if, for some $c_2>0$ depending on $\rho$,
\begin{equation} \label{condieps3}
\veps_n^2 \ge c_2 \frac{\log(1+2d^*n^2)}{n}.
\end{equation}
Now turning to $\eqref{condieps}-\eqref{condiepsb}$, recalling $\cC_i=\cC_i(\be,d^*)\ge 1$ and $K\geq 1$, it suffices to have, using $\veps_n\ge 1/\sqrt{n}$ and  $a^*\le n$ as noted earlier,
\begin{equation}\label{trveps}
\veps_n^2 \ge \{ B_1 {a^*}^{-2\be} \} \vee \{ B_2  {a^*}^{d^*}\log^{1+d^*}(n)/n\},
\end{equation}
where $B_1=C(\xi^{-1}\cC_1 K)^2$ and $B_2=C\rho^{-1}{K}^2\cC_2 (c_1d^{*c_2})^{d^*}$, with $c_1,c_2,C$ universal constants. We note that \eqref{trveps} implies \eqref{condieps3} for $C$ large enough, using $a^*\geq 1$ and $n\ge 3$ (which implies $\log(d^*)\leqa \log^{1+d^*}(n)$). This concludes the proof of Theorem \ref{thmvs}, provided $\veps_n\leq 8\xi^{-1}$.\\

{\em Proof of Theorem \ref{thmtoy}.}  One follows the 
proof of  Theorem \ref{thmvs}, noting that the fixed values of scaling parameters specified in \eqref{scale-freeze} belong to the intervals $I_i$ from the proof of Theorem \ref{thmvs}, and where now there is no conditioning on $A_i$ (those have given values now).

Let us set  $a^* = n^{1/(2\beta+s)}$, $\veps_n^2 = M\log^{1+s}(n) n^{-2\beta/(2\beta+s)}$   and $d^*=s$. Then the conditions \eqref{condieps}, \eqref{condiepsb}, \eqref{condieps2}  and \eqref{trveps} arising on $\veps_n$ in the proof of Theorem \ref{thmvs} are satisfied for $M$ large enough depending on $s$, $K$, $\beta$, $\xi$ and $\rho$. This gives 
\[ \Pi\left[f:\ \norm{f-f_0}_\infty <\xi \veps_n \right] \ge e^{ -\rho n\veps_n^2}.\]
One concludes similarly as for the proof of Theorem 2, using the discussion  following \eqref{eq: supball condition}.

\subsection{Proof of Theorem \ref{theorem: posterior contraction rates hdgp}}
\label{proofthm3} 
The proof of Theorem \ref{thmvs} needs to be suitably generalized and modified: as we shall see below, the considered $L^\infty$ balls for the various layers of the composition need to have carefully chosen radii. To obtain the results, one needs to verify the prior mass condition \eqref{eq: supball condition} for $\veps_n$ as in the statement of the theorem. For any $f_0\in \mathcal{F}_{\text{deep}}(\lambda,\mathbb{\beta},K)$ where $\lambda=(q,d_1,\dots,d_q,t_0,\dots,t_q)$, we now have
 \begin{align*}
 \Pi\left[f:\ \norm{f-f_0}_\infty <\xi \varepsilon_n\right]\geq  &\Pi_q[\{q\}]\Pi[\{d_1,\dots,d_q\}|\ q]\\
& \Pi\left[\norm{g_q\circ \dots\circ g_0-h_q\circ\dots\circ h_0}_\infty <\xi\varepsilon_n\ |\ q, d_1,\dots,d_q\right].\\
 \end{align*}
 Let us now fix $0\leq i\leq q$ and $1\leq j\leq d_i$, such that we focus on the marginal distribution of $g_{ij}$. Lemma \ref{lemma: structure regression distance} \cite{AnnexeDGP} indeed ensures that for any $\varepsilon_n(i)>0$, denoting $\mathbf{d}=(d_1,\ldots,d_q)$,
 \begin{align}
 \Pi\left[\norm{g_q\circ \dots\circ g_0-h_q\circ\dots\circ h_0}_\infty <\xi q \prod_{i=0}^q [2^{|\beta_{i}-1|} t_{i}K\vee 1]\max_{i=0,\dots,q} \varepsilon_n(i)^{\alpha_i}\ |\ q, \mathbf{d} \right]\nonumber\\ 
 \geq \ \prod_{i=0}^q \prod_{j=1}^{d_{i+1}} \Pi\left[\norm{W^{A_{ij}}-h_{ij}}_\infty \leq \xi^{1/\alpha_i} \varepsilon_n(i)\ |\ q,\mathbf{d}\right],\label{eq: sep lower bound}
 \end{align}
where $W^{A_{ij}}$ is as in \eqref{eq: gp prior dilatation} (with random bandwidths $A_{ij}$) and $\alpha_i=\prod_{l=i+1}^q (\beta_l\wedge 1)$. Since the $h_{ij}$ are bounded by $1$ in supnorm, the factors in the above product are lower bounded by \[\Pi\left[\norm{g_{ij}-h_{ij}}_\infty \leq \xi^{1/\alpha_i}\varepsilon_n(i)\ |\ q,d_1,\dots,d_q\right].\] If we can find $\varepsilon_n(i)$ such that the above quantity is lower bounded by $e^{-\rho n \varepsilon_n(i)^{2\alpha_i}}$, then we can verify that $\varepsilon_n=C(K, \lambda)\max_{i=0,\dots,q} \varepsilon_n(i)^{\alpha_i}$ such that $n\varepsilon_n^2\to\infty$, for \[C(K,\lambda)= \left[q \prod_{i=0}^q [2^{|\beta_{i}-1|} t_{i}K\vee 1]\right]\vee \sum_{i=0}^{q}d_{i+1},\] is a posterior contraction rate thanks to \eqref{eq: supball condition}. Having $n\veps_n^2\to\infty$ ensures that the remaining mass in Theorem \ref{bat_result} in Appendix \ref{app: Frac post} \cite{AnnexeDGP} is vanishing, so that $\veps_n$ is indeed a contraction rate. Indeed, up to the constant factor \[L=\Pi_q[\{q\}]\Pi[\{d_1,\dots,d_q\}|\ q]\] independent of $n$, we can derive \eqref{eq: supball condition} from the lower bound on the right-hand side of \eqref{eq: sep lower bound}
 \[Le^{-\rho n\veps_n^2}.\]Indeed, as long as $n\varepsilon_n^{2}\to \infty$, we could replace $\veps_n$ with $C\veps_n$, $C\geq1$, for $C$ such that $Le^{-\rho n \veps_n^2}\geq e^{-\rho n C\veps_n^2}$. Since the left side of \eqref{eq: supball condition} increases when replacing $\veps_n$ with $C\veps_n$, \eqref{eq: supball condition} would be satisfied with $C\veps_n$. This is enough as we seek to express $\veps_n$ up to a large enough constant.
 
 From here, we can continue as in the proof of Theorem \ref{thmvs}. Since we assume $h_{ij}\in\cF_{VS}(K,\beta_i,d_i,t_i)$, we have $h_{ij}(x_1,\dots,x_d)=f_{ij}(x_{k_1},\dots,x_{k_{t_i}})$ for some $f_{ij}\in \cF(K,\be_i,t_i)$.  
Let us set, for  $S_0=\{1,\ldots, d_i \} \setminus \{k_1,\dots,k_{t_i}\}$, $\xi = 2\sigma_0^2/\sqrt{1+4\sigma_0^2}$, and \[v_{i,n}\coloneqq \frac{\xi \veps_n(i)^{1-\al_i}}{8\sqrt{\rho n d_i}},\]the intervals
 \[
I_k=
\begin{cases}
[0,v_{i,n}],&\ \text{if }k \in S_0,\\
[a^*, 2a^*],&\ \text{otherwise},
\end{cases}
\]
for  $a^*\in[1,n/2]$. Let's also consider an arbitrary vector $A_{ij}$ such that $(A_{ij})_k\in I_k$ for $1\leq k\leq d$ (in the following, we note $\mathcal{A}_k=(A_{ij})_k$ for simplicity). If we can show $\prod_{k=1}^{d_i} \pi(I_k) \geq 2e^{-\rho n \varepsilon_n(i)^{2\alpha_i}/2}$, and, $\cC_{1,i},\cC_{2,i}$ constants of the form $c_{j,i} C_{j,i}^{t_i}$, $j=1,2$, as in the statement of Theorem \ref{concd},

\begin{align} 
4\geq \xi^{1/\alpha_i}\veps_n(i)/2 & \ge 4\cC_{1,i}
K(a^*)^{-\be_i}, \label{condiepsdeep}\\
\rho n\veps_n(i)^{2\alpha_i}/2 & \ge \cC_{2,i}
K^2  (a^*)^{t_i} + (Ct_i^{c}a^*)^{t_i}\log^{1+t_i}(8a^*/\xi^{1/\alpha_i}\veps_n(i)),
 \label{condiepsbdeep}
 \end{align}
 (counterparts of \eqref{condieps} and \eqref{condiepsb}) and, for some $c_1>0$,
 \begin{align} 
\veps_n(i) & \ge 4\xi^{-1/\alpha_i}c_1 d_i \cdot \underset{k\in S_0}{\max}\ \mathcal{A}_k \cdot
\sqrt{\log\left(1+ 2t_i\ \underset{k\notin S_0}{\max}\ \mathcal{A}_k^2\right)}, \label{condiepsbdeep2}\\
\max_{k\in S_0} \cA_k & \le \frac{\xi \veps_n(i)^{1-\al_i}}{8\sqrt{\rho n d_i}}\label{condiAdeep2}
\end{align} 
(the first is a  counterpart of \eqref{condieps2}, the second ensures that an upper bound $e^{- \xi^2 \varepsilon_n(i)^2/32\sigma_i^2}$ obtained as in \eqref{eq: Dudley ineq} is further upper bounded by $\exp(-\rho n\veps_n(i)^{2\al_i})$, as $\sigma_i^2\le 2d_i\max_k \cA_k^2$). Under these conditions, we can conclude 
\[ \Pi\left[\norm{W^{A_{ij}}-h_{ij}}_\infty \leq \xi^{1/\alpha_i}\varepsilon_n(i)\ |\ q,d_1,\dots,d_q\right] \geq e^{-\rho n \varepsilon_n(i)^{2\alpha_i}}\]
in the same way we obtained \eqref{eq: final lower bound}.

From \eqref{condiAdeep2}, which is satisfied by definition of $v_{i,n}$, and $\underset{k\notin S_0}{\max}\ A_k \leq n$, \eqref{condiepsbdeep2} is satisfied whenever
\begin{equation}\label{condiAdeep23}
   \veps_n(i)^{4\alpha_i-2}  \ge c_2 \frac{\log(1+2t_in^2)}{n} 
\end{equation}
for some $c_2>0$. Now turning to $\eqref{condiepsdeep}-\eqref{condiepsbdeep}$, recalling $\cC_{j,i}=\cC_{j,i}(\be_i,t_i)\ge 1$ and $K_+=K\vee 1$, it suffices to have, using $\veps_n(i)\ge 1/\sqrt{n}$ and  $a^*\le n$ as noted earlier,
\begin{equation}\label{trveps_deep}
\veps_n(i)^{2\alpha_i} \ge \{ B_1 {a^*}^{-2\alpha_i\be_i} \} \vee \{ B_2  {a^*}^{t_i}\log^{1+t_i}(n)/n\},
\end{equation}
where $B_1^{1/\alpha_i}=C(\cC_{1,i} K_+t_i^{2})^2$ and $B_2=C\rho^{-1}{K_+}^2\cC_{2,i} (c_1t_i^{c_2})^{t_i}$, with $c_1,c_2,C$ universal constants. If $\veps_n(i)\leq 1$, we note that \eqref{trveps_deep} implies \eqref{condiAdeep23} for $C$ large enough, using $\alpha_i\leq 1$, $a^*\geq 1$ and $n\ge 3$ (which implies $\log(t_i)\leqa \log^{1+t_i}(n)$).

Optimising $a^*$ in \eqref{trveps}, leads to setting 
\begin{equation} \label{optastardeep}
 (a^*)^{2\alpha_i\be_i+t_i}=(B_1n)/(B_2 \log^{1+t_i}(n)) \vee 1.
\end{equation} 
Condition \eqref{trveps_deep} then becomes, for $a^*$ as in \eqref{optastardeep},
\begin{equation} \label{almosttherate}
\veps_n(i)^{2\alpha_i} 
 \ge \left[ B_3 (\log{n})^{\frac{2\alpha_i\be_i(1+ t_i)}{2\alpha_i\be_i+t_i}} 
n^{-\frac{2\alpha_i\be_i}{2\alpha_i\be_i+t_i}}\right] \vee \left[B_2\log^{1+t_i}(n)/n\right], 
\end{equation}
where $B_3 = K^2c_5^{t_i}t_i^{t_i\frac{c_6}{2\alpha_i\beta_i+t_i}}$, recalling $B_2=K^2 (c_3t_i^{c_4})^{t_i}$, and $c_3,\ldots,c_6\ge 1$ are constants only depending on $\beta_i, \rho$. For $\veps_n(i)$ equal to the lower bound in \eqref{almosttherate}, we indeed have $\veps_n(i)\leq 1$ and, for $n$ large enough, $\veps_n(i)^2=B_3^{1/\alpha_i}(\log{n})^{\frac{2\be_i(1+ t_i)}{2\alpha_i\be_i+t_i}} 
n^{-\frac{2\be_i}{2\alpha_i\be_i+t_i}}$. Condition \eqref{condiAdeep2} is then satisfied by definition.

It now remains to prove that $\prod_{k=1}^{d_i} \pi(I_k) \geq e^{-\rho n \varepsilon_n(i)^{2\alpha_i}/2}$ given the definition of $v_{i,n}$ and  condition \eqref{optastardeep}.
Using the fact that $1\leq a^*\leq n/2$, a straightforward modification of the proof of Lemma \ref{lem-hs-fix} \cite{AnnexeDGP} gives that it is satisfied for a parameter $\tau_i>0$ satisfying
\begin{equation}\label{eq: modif_cond_growth}
    n\veps_n(i)^{2\alpha_i} \ge (2/\rho)\left[ t_i\log(10a^*/\ta_i)-d_i\log\left(v_{i,n} e_0 \ta_i \right) +\log 2
\right],
\end{equation}
whenever $v_{i,n}<\tau_i<a^*$. This last condition is satisfied for any fixed $\tau_i>0$ as $v_{i,n}\to0$ and $a^*\to\infty$. 
Also, equation \eqref{eq: modif_cond_growth} is satisfied for large enough $n$ as the left-hand side has a polynomial growth and the right-hand side has a logarithmic growth in $n$.

This concludes the proof of Theorem \ref{theorem: posterior contraction rates hdgp} in the case of the prior \textsf{Deep--HGP}. The proof for the prior \textsf{Deep--HGP}$(q_{max},d_{max})$ is very similar and can be found in Appendix \ref{app: fixed dmax and qmax} \cite{AnnexeDGP}.

\subsection{Proof of Theorem \ref{theorem: posterior contraction rates hdgp 2}}

We proceed as in the proof of Theorem \ref{theorem: posterior contraction rates hdgp} but with the new horseshoe prior with shrinking parameter $\tau_0$ on the lengthscales of the first layer, with special attention to that layer of GPs, $i=0$, as $d_0,t_0$ may now go to infinity. As in Corollary \ref{cor-hdhs}, for $i=0$, we now have 
\begin{equation} \label{almosttherate2}
\veps_n(0)
 \ge \left[ B_3 (\log{n})^{\frac{2\alpha_0\be_0(1+ t_0)}{2\alpha_0\be_0+t_0}} 
n^{-\frac{2\alpha_0\be_0}{2\alpha_0\be_0+t_0}}\right] \vee \left[B_2\log^{1+t_0}(n)/n\right], 
\end{equation}
which as $n\to\infty$, under \eqref{dims2}, is equal to 
$\veps_n(0)=C^{t_0}n^{-\frac{2\alpha_0\be_0}{2\alpha_0\be_0+t_0}}(\log{n})^{\frac{2\alpha_0\be_0(1+t_0)}{2\alpha_0\be_0+t_0}} $, for $C$ depending on $K$, $\rho$ and $\beta_i$, $i\geq0$.
Also the condition on $\tau_0$ becomes 
\begin{equation}
10a^*e^{-n\rho\varepsilon_n(0)^{2\alpha_0}/4t_0}  \le \tau_0 \le 
C_0\frac{\varepsilon_n(0)^{1-\alpha_0}}{\sqrt{nd_0^4}}
\end{equation}
for $a^*$ as in \eqref{optastardeep} and $C_0$ depending $\xi$ and $\rho$,
via a slight modification of Lemma \ref{lem-hs-van} \cite{AnnexeDGP}. As it is satisfied under the assumption of the theorem, this concludes the proof.

\section{Dimension-dependent bounds for multibandwidth \textsf{SqExp} Gaussian processes} \label{sec:boundcf}

In order to prove posterior contraction rates for deep GPs, 
a key step is to derive an upper bound for the concentration function \eqref{eq: concentration function}. The next Theorem enables us in particular to revisit Lemmas 4.2 and 4.3 from \cite{bpd14}, with explicit multiplicative constants depending on the ambient dimension $d$ in the result. This is a novel contribution to the literature on squared-exponential GPs, to the best of our knowledge. Also, these results allow us to deploy the \textsf{HGP} and \textsf{Deep--HGP} priors in the high-dimensional setting. For simplicity we do not consider here the anisotropic case in which the function $f_0$ can have varying smoothness across coordinates, although this could be done as well, following the approach of \cite{bpd14}. We focus on the variable selection aspect of the problem, assuming the same regularity on the active directions of $f_0$.
\begin{theorem} \label{concd}
Let $W^A$ be a \textsf{SqExp} Gaussian process in dimension $d\ge 1$ with deterministic vector of scalings $A=(A_1,\ldots,A_d)$ with $a\le A_i\le 2a$ for $i=1,\ldots,d$ and some $a\ge \sqrt{\log(2)/d}/2$. Let $\vphi^A_{f_0}(\veps)$ be the concentration function of $W^A$. Suppose $f_0\in \cF(\be,K,d)$ for some $\be, K>0$. There exist constants $\cC_1(\be,d)$ and $\cC_2(\be,d)$ depending only on $\be, d$ and a universal $c, C>0$ such that, if 
\[ \cC_1(\be,d)K^2 a^{-\be} \leq \veps \leq 4,\]
then
\[
\vphi_{f_0}^A(\veps) \le  \cC_2(\be,d) K^2 a^d + (Ca)^d d^{cd} \log^{1+d}(2a/\veps).
\]
Moreover, for $i=1,2$ one can take $\cC_i(\be,d)=c_i(\be) C_i(\be)^d$ for some constants $c_i(\be), C_i(\be)$ that depend only on $\be$.
\end{theorem}
 The proof of this result can be found in Appendix \ref{app: proof dim-dep} \cite{AnnexeDGP}.

\begin{acks}[Acknowledgments]
 The authors would like to thank François Bachoc, Agnès Lagnoux, Johannes Schmidt-Hieber, Aad van der Vaart as well as two referees for insightful comments.
\end{acks}
\begin{funding}
%
This work is co-funded by ANR GAP project (ANR-21-CE40-0007) and the European Union (ERC, BigBayesUQ, project number: 101041064). 
\end{funding}



\bibliographystyle{imsart-number} 
 \bibliography{biblio}

\pagebreak


\begin{appendix}

\title{Supplementary material to `Deep horseshoe Gaussian processes'}

This supplementary material contains the proof of Theorem \ref{concd} as well as the proofs of a number of technical lemmas. It also contains results on H\"older spaces in increasing dimensions (Appendices \ref{app:hold} and \ref{app-rate-d}), including a lower bound for the minimax rate in regression in this setting. Appendix \ref{app: Frac post} is concerned with fractional  posteriors, while Appendix \ref{sec:verif}  verifies the conditions for different priors on scalings.

\section{Proof of Theorem \ref{concd}}\label{app: proof dim-dep}

In order to prove Theorem \ref{concd}, we need to bound from above the concentration function of the \textsf{SqExp} Gaussian process $W^A$. 

Recall that the concentration function of the Gaussian process $W$ with RKHS $\mathbb{H}$ is 
\[ \varphi_g(\varepsilon)\coloneqq \underset{h\in\mathbb{H}:\ \norm{h-g}_\infty\leq \varepsilon}{\inf}\, \frac{1}{2} \norm{h}_{\mathbb{H}}^2- \log P\left[\norm{W}_\infty<\varepsilon\right].\]
The form of the RKHS for the \textsf{SqExp} Gaussian process is recalled in Section \ref{app: RKHS}. 
Then, in Lemma \ref{lemma:: approx rkhs bound} and \ref{lemma: small ball proba}, we bound successively the two parts of the $\vphi_g$ in the last display for $W=W^A$ in terms of $\veps, d$, the scale parameters $A$ and the smoothness of $g$. Combining the conclusions of the two Lemmas, one obtains Theorem \ref{concd}. \qed

\subsection{RKHS of multi-bandwidth \textsf{SqExp} Gaussian process}\label{app: RKHS}
\[\ \]

The RKHS $\mh^A$ of the \textsf{SqExp} process $W^A(t), t\in [-1,1]^d$ with bandwidth sequence $(A_1,\ldots,A_d)$ has been characterised in \cite{aadharry09}, Lemma 4.1, in case of a common bandwidth $A_1=\cdots=A_d$, result extended in \cite{bpd14}, Lemma 4.1, to the case of different bandwidths. 

The space  $\mh^A$ is the set of real parts of functions
\begin{equation} \label{rkhs-elem}
 t\to \int e^{i \psg u , t \psd} g(u) \nu_A(u)=:h_g(t),
\end{equation} 
where the spectral measure $\nu_A$ admits a density $f_{\nu,A}(u)=f_\nu(u_1/A_1,\ldots,u_d/A_d)/(A_1\cdots A_d)$ with respect to Lebesgue measure, and $g$ runs over the complex Hilbert space $L^2(\nu_A)$. We recall that in the case of the \textsf{SqExp} process $f_{\nu}$ is given by $f_\nu\colon s\mapsto \prod_i e^{-s_i^2/4}/\sqrt{4\pi}$. 

Also, the results of \cite{aadharry09,bpd14} show, since $[-1,1]^d$ has an interior point, that the squared-RKHS norm of the element $h_g$ of $\mh^A$ in the last display is 
\begin{equation} \label{rkhs-norm}
 \| h_g \|_{\mh^A}^2 =\|g\|_{L^2(\nu_A)}^2=\int g(u)^2 f_{\nu,A}(u)du. \\
\end{equation}

The key role played by the  RKHS  in the characterization of posterior contraction rates for Gaussian processes priors  is explained by 
 the following lemma.

\begin{lemma}[Proposition 11.19 in \cite{MR3587782}]\label{lem: concentration}
For any mean zero Gaussian random element $W$ in a separable Banach space $\mb$ with norm $\|\cdot\|$, any $f$ in the closure in $\mb$ of its RKHS and any $\epsilon>0$,
\[ P(\norm{W-f}<\epsilon)> e^{-\varphi_f(\epsilon/2)},\] where $\varphi_f$ if the concentration function of $W$ at point $f$ as defined  in \eqref{eq: concentration function}.
\end{lemma}

\subsection{Dimension-dependent upper bounds on the concentration function}\label{sec: control small ball prob}

\[\ \]
  
{\em Dimension-dependent bound on the approximation term.}
Let us  denote  for $K, \be>0$ and  $d$ an integer,
\begin{align} 
\cF(K,\beta,d) 
 = \left\{ f:I^d \to \RR:\ \ f\in \cC^\be(I^{d})\text{ and } \| f \|_{\be,\infty} \le K \right\}.
\label{defcld_holder}
\end{align}
The following lemma deals with the "approximation" part of the concentration function. 
  
\begin{lemma}\label{lemma:: approx rkhs bound}
 Suppose $f_0\in \cF(\be,K,d)$ as in \eqref{defcld_holder} for some $\be,K>0$ and $d$ an integer. Let $A=\left(A_1,\dots,A_d\right)$ be such that $a\leq A_i\leq 2a$ for $i=1,\ldots,d$ and $a>0$.    
Then, there exist positive constants $\cC_1=\cC_1(\be,d)$ and $\cC_2=\cC_2(\be,d)$ that depend only on $\be, d$ such that 
\[ \underset{\substack{h\in\mathbb{H}^A:\\ \norm{h-f_0}_\infty\leq \cC_1 K^2 \ a^{-\be}}}{\inf}  \norm{h}_{\mathbb{H}^A}^2 \le  \cC_2 K^2 a^d.\] 
Moreover, for $i=1,2$ one can take $\cC_i=c_i(\be) C_i(\be)^d$ for some constants $c_i(\be), C_i(\be)$ that depend only on $\be$.
 \end{lemma}

\begin{proof}
We revisit the proof of Lemma 4.2 from \cite{bpd14}, making the dependence in  dimension $d$ explicit.   Let $v$ be a complex-valued function such that $\int v(t)dt = 1$, $\int |t|^\beta |v(t)|dt <\infty$,  
\[ \int t^k v(t)dt = 0 \ \text{ for }\ k=0,\dots,\lfloor \beta \rfloor, \] 
and its Fourier transform $\widehat{v}$ is compactly supported (all integrals being over $\RR$).

Define $V:\ \RR^d\to \CC$ by $V(t)=v(t_1)\cdots v(t_d)$ and 
\[V_A(t) := V(A_1t_1,\dots,A_dt_d)\prod_{i=1}^d A_i.\] 
By Lemma \ref{lem: Whitney}, we  extend $f_0$ to a function $f_0:\ \R^d\to\R$ with  $\norm{f_0}_{\beta,\infty}\le K\cE_\be(d)$. Next we set $\tilde{f_0}=f_0 \cdot h$, where $h: \RR^d\to [0,1]$ is a $\cC^\infty$--function, equal to $1$ on $[-1,1]^d$, to $0$ outside of $[-2,2]^d$  and with $\|h\|_{\be,\infty}\le 2d C_{1,\be}^d$ (from Lemma \ref{lemder}, it suffices to build such a function, say $\eta$, for the case $d=1$, and to set $h(t_1,\ldots,t_d)=\eta(t_1)\cdots \eta(t_d)$), where $C_{1,\be}$ only depend on $\be$. By construction, the function $\tilde{f_0}$ equals $f_0$ on $[-1,1]^d$, has compact support within $[-2,2]^d$, and by Leibniz's rule and the same argument as in the proof of Lemma \ref{lemder}
\[ \|\tilde{f_0}\|_{\be,\infty}\le 2^{\lfloor \beta \rfloor+1}dKC_{1,\be}^d\cE_\be(d),\]
where $C_{1,\be}$ only depends on $\be$. 

Additionally, $\mathbb{H}^A$ is the set of real parts of functions $t\mapsto \int e^{i \langle s, t \rangle} \eta(s) d\nu_A(s)$ for $\nu_A$ the measure with Lebesgue density $s\mapsto \prod_i \frac{1}{ \sqrt{4\pi}A_i}e^{-s_i^2/4A_i^2}$ and $\eta\in L^2\left(\nu_A\right)$ (see Appendix \ref{app: RKHS}). The corresponding RKHS norm is also equal to the $L^2(\nu_A)$--norm of $\eta$. The convolution $h\coloneqq V_A \ast \tilde{f}_0$ is then an element of $\mathbb{H}^A$ since for any $t\in[-1,1]^d$
\begin{align*}
h(t) &= (2\pi)^{-d} \int_{\R^d} e^{i\ \langle s, t\rangle} \widehat{\tilde{f}}_0(s) \widehat{V}_A(s) ds\\
&= (2\pi)^{-d}  \int_{\R^d} e^{i\ \langle s, t\rangle} \widehat{\tilde{f}}_0(s) \widehat{V}_A(s) \prod_{j=1}^{d} (\sqrt{4\pi}A_{i}e^{s_{i}^2/4A_{i}^2} )d\nu_A(s).
\end{align*}
with its square norm bounded by $ (2\pi)^{-2d} \norm{\tilde{f_0}}^2_{2} \norm{\frac{|\widehat{v}|^2}{\sqrt{2}^{-1/2}\phi(\cdot\ /\sqrt{2})}}_\infty^d \prod_{i=1}^{d} A_i  $ according to the Plancherel theorem. Also, $\|\tilde{f_0}\|_2^2\le 4^d \|\tilde{f_0}\|_\infty^2\le 4^d \|\tilde{f_0}\|_{\be,\infty}^2$, so that an upper bound for this last quantity is $c_{1,\beta}C_{1,\beta}^d K^2 \prod_i A_i$, using Lemma \ref{lem: Whitney}, which provides an estimate of $\cE_\be(d)$ of the form $c_\be C_\be^d$.

It remains to show that $h$ approximates $f_0$ well enough on the unit cube. It follows by Taylor expansion at order $\lfloor\beta\rfloor$ of $\tilde{f}_0$ around $t\in[-1,1]^{d}$. For any $s\in\R^{d}$,
\[ \tilde{f}_0(t+s) = \sum_{\boldsymbol{\alpha}: |\boldsymbol{\alpha}|\leq \lfloor \beta \rfloor} \partial^{\boldsymbol{\alpha}} \tilde{f}_0(t) \frac{s^{\prod \alpha_i}}{\prod \alpha_i !} + S(t,s),\]
such that $|S(t,s)|\leq C \|\tilde{f_0}\|_{\be,\infty} \norm{s}_\infty^\beta$, independently of $t$ and for $C$ depending on $\be$ only. By assumptions on $v$, we find that for $t\in[-1,1]^{d}$
\[ V_A \ast \tilde{f}_0(t)-f_0(t) = \int_{\R^{d}} V(s) (\tilde{f}_0(t-s/A) -\tilde{f}_0(t))ds = \int_{\R^{d}} V(s) S(t,-s/A)ds.\]
As a consequence, $\left|h(t)-f_0(t)\right|\leq  C \|\tilde{f_0}\|_{\be,\infty} \int_{\R^{d}} V(s) \norm{s}_\infty^\beta ds\ a^{-\beta}$. Observing that \[\int_{\R^{d}} V(s) \norm{s}_\infty^\beta ds \leq \int_{\R^{d}} V(s) \sum_{j=1}^d |s_i|^\beta ds = d\int_\R v(t)|t|^\beta dt\] concludes the proof.
\end{proof}

{\em Dimension-dependent lower bounds on the small ball probability.} We now deal with the small ball probability in the concentration function and bound it. We note that, for $\nu_A$ the spectral measure of the process $W^A$, $\int e^{ \norm{t}_2/2}d\nu_A(t)<\infty$ as verified in Lemma \ref{lem-v} below. 

 \begin{lemma}\label{lemma: small ball proba}
There exists absolute constants $C$, $c$ such that for $0<\varepsilon\leq4$ and any given $A_i\geq 1/(96\sqrt{d})$, with $\bar{A}=\max_i A_i> \sqrt{\log 2}/(2\sqrt{d})$,
 \[ \varphi_0^A(\varepsilon)\coloneqq - \log P\left[\norm{W^A}_\infty<\varepsilon\given A\right] \leq  C^d d^{c d} \log\left(\frac{\bar{A}}{\varepsilon}\right)^{1+d}\ \prod_{i=1}^d A_i.\]
 \end{lemma}
 
 \begin{proof}
We extend the proof of Lemma 4.6 in \cite{aadharry09}, keeping all the dependencies on the dimension explicit. We start with formula (3.19) in \cite{kuelbsli93} which states that, for any $\varepsilon,\lambda>0$,
 \[ \varphi_0^A(2\varepsilon) + \log\Phi\left(\lambda +\Phi^{-1}(e^{-\varphi_0^A(\varepsilon)})\right)\leq \log\ N\left(\varepsilon/\lambda, \mathbb{H}^A_1, \norm{\cdot}_\infty\right),\]
 with $\mathbb{H}^A_1$ the unit ball of $\mathbb{H}^A$ and $\Phi$ the standard normal distribution function. For the choice $\lambda=\sqrt{2\varphi_0^A(\varepsilon)}$ and with the inequality $\Phi\left(\sqrt{2x}+\Phi^{-1}(e^{-x})\right)\geq 1/2$ for any $x>0$ (see Lemma 4.10 in \cite{aadharry09}), we get
 \begin{equation}\label{eq: crude ineq for small ball}
 \varphi_0^A(2\varepsilon) \leq \log\ 2N\left(\varepsilon/\sqrt{2\varphi_0^A(\varepsilon)}, \mathbb{H}^A_1, \norm{\cdot}_\infty\right).
 \end{equation}
 Before going further, it is necessary to prove a crude bound of the form \[\varphi_0^A(\varepsilon)\lesssim (\bar{A}/\varepsilon)^\tau,\] for some $\tau>0$. The Karhunen-Loève expansion of $W^A$ ensures that for an orthonormal basis $h_1,h_2,\dots$ of $\mathbb{H}^A$ and the map $u_A:\mathbb{H}^A\mapsto C\left([-1,1]^d,\ \norm{\cdot}_\infty\right)$ defined by $u_A(h)=h$, $W^A$ is equal in distribution to $\sum_{j=1}^{\infty}\xi_ju_A(h_j)$, for $\xi_j$ i.i.d. $\mathcal{N}\left(0,1\right)$ variables and the series converging almost surely in $C\left([-1,1]^d,\ \norm{\cdot}_\infty\right)$. As in \cite{lilinde99}, we introduce the functional $e_l$ defined as
 \[ e_l(v) \coloneqq \inf \left\{ \eta>0\colon\ N(\eta, v(B_E), \norm{\cdot}_F) \leq 2^{l-1} \right\}\] defined for any compact operator $v$ from a Banach space $\left(E,\norm{\cdot}_E\right)$, with unit ball $B_E$, into another one with norm $\norm{\cdot}_F$.
In our case, we obtain
 \[ e_l\left(u_A\right) = \inf \left\{\eta>0:\  \log\ N\left(\eta, \mathbb{H}^A_1, \norm{\cdot}_\infty\right) \leq (l-1)\log 2\right\}.\]
 Lemma \ref{lemma bound entropy unit ball rkhs} implies that, for $B^2=\int e^{ \norm{t}_2/2}d\nu(t)$ and $\varepsilon<2B(C_1d)^d$,
 \begin{equation}\label{eq: bound entropy dim} \log\ N\left(\varepsilon, \mathbb{H}^A_1, \norm{\cdot}_\infty\right) \leq C_2^d d^{4d} \log\left(\frac{1}{\varepsilon}\right)^{1+d} \prod_{i=1}^d A_i,\end{equation}
 and $\log\ N\left(\varepsilon, \mathbb{H}^A_1, \norm{\cdot}_\infty\right)=0$ for $\epsilon\geq 2B(C_1d)^d$.
By definition, $e_l\left(u_A\right)$ is smaller than the solution $\eta^*$ of  $C_2^d d^{4d} \log\left(\frac{1}{\eta^*}\right)^{1+d} \bar{A}^d=(l-1)\log 2$, provided it is smaller than $2B(C_1d)^d$, and smaller than $2B(C_1d)^d$ otherwise, which gives
 \[  \eta^* = \begin{cases} e^{-\left(C_2d^4\right)^{-\frac{d}{1+d}} \bar{A}^{-\frac{d}{1+d}} \left[(n-1)\log 2\right]^{\frac{1}{1+d}} }, & \text{if } l>1, \\ 2B(C_1d)^d, & \text{if } l=1, \end{cases}\]noting that $B\geq 1$.
For $u_A^*$ the dual of $u_A$, we rewrite the first equation on p.300 in \cite{tj87} as \begin{equation}\label{eq: TJ87}\underset{l\leq m}{\sup}\ l^\alpha e_l\left(u_A^*\right) \leq 32\ \underset{l\leq m}{\sup}\ l^\alpha e_l\left(u_A\right)\end{equation} for any $m\geq 1$ and $\alpha>0$. Also, for any $l\geq 2$,
\begin{align*}
le_l\left(u_A\right) &\leq l e^{-\left(C_2d^4\right)^{-\frac{d}{1+d}} \bar{A}^{-\frac{d}{1+d}} \left[(l-1)\log 2\right]^{\frac{1}{1+d}} }\\
&\leq \frac{2}{\log 2} \left(C_2d^4\right)^d \bar{A}^d \left(\frac{(l-1)\log 2}{\left(C_2d^4\right)^d \bar{A}^d}\right)e^{-\left(C_2d^4\right)^{-\frac{d}{1+d}}\bar{A}^{-\frac{d}{1+d}} \left[(l-1)\log 2\right]^{\frac{1}{1+d}} }\\
&\leq \frac{2}{\log 2}\left(\frac{1+d}{e}\right)^{1+d} \left(C_2 d^4 \bar{A}\right)^d\leq \left(Cd^6\right)^d\bar{A}^d,
\end{align*}
since $h\colon\ x\to xe^{-x^{\frac{1}{1+d}}}$ is upper bounded by $(1+d)^{1+d}e^{-(1+d)}$ on $\R_+^*$ ($h(0)=0$, $h'(x)=0$ has $x=\left(1+d\right)^{1+d}$ for solution, $h$ is non-negative on the positive real line and vanishing) .
Taking $\alpha=1$ in \eqref{eq: TJ87} and using the bound from Lemma \ref{lem-v}, this implies that, for $m\geq 1$,
 \begin{align*}
me_m\left(u_A^*\right) &\leq\underset{l\leq m}{\sup}\ l e_l\left(u_A^*\right)  \\
&\leq 32 \underset{l\leq m}{\sup}\ l e_l\left(u_A\right) \\
&\leq \left(64B(C_1d)^d\right)\vee \left(32(Cd^6)^d\bar{A}^d\right)\leq  C_3^{d} d^{6 d}\bar{A}^d .
\end{align*}
 From Lemma 2.1 in \cite{lilinde99}, itself cited from \cite{pisier89}, and this last upper bound, we obtain the following upper bound on \[l_n(u)=\inf \left\{\Bigg(E \norm{\sum_{j=n}^\infty \xi_j x_j}_\infty^2\Bigg)^{1/2}\colon\ W^A=_d\sum_{j=1}^\infty \xi_j x_j,\ \xi_j\overset{\text{i.i.d.}}{\sim}\mathcal{N}(0,1),\ x_j \in C([-1,1]^d,\ \norm{\cdot}_\infty)\right\},\]the $n$--th approximation number of $u$ as defined in Section 2 of \cite{lilinde99}, p.1562, $n\geq 1$. For $c_1,c_2$ universal constants independent of $d$:
 \[ l_n(u)\leq c_1\sum_{m\geq c_2n} e_m\left(u^*\right) m^{-1/2}\left(1+\log m\right) \leq C_3^d d^{6d} \bar{A}^d n^{-1/2} \log(n).\]
From the proof of Proposition 2.3 in \cite{lilinde99}, we find that, for $\varepsilon>0$, $\sigma\coloneqq\mathbb{E}\left[\norm{W^A}_\infty^2\right]^{1/2}$ and \[n(\varepsilon)\coloneqq \max\left\{n:\ 4l_n(u)\geq \varepsilon\right\},\] the following bound stands
 \[ P\left[\norm{W^A}_\infty<\varepsilon\right] \geq \frac{3}{4}\left(\frac{\varepsilon}{6\sigma n(\varepsilon)}\right)^{n(\varepsilon)},\] which implies
 \begin{equation}\label{eq: bound conc fonction with neps} \varphi_0^A(\varepsilon) \leq n(\varepsilon)\log \left(\frac{8\sigma n(\varepsilon)}{\varepsilon}\right).\end{equation}
 We note that $n(\varepsilon)$ is well-defined as long as $\varepsilon\leq 4$: indeed, $l_n(u)$ decreases with $n$ and $l_1(u)\coloneqq\sigma>\mathbb{E}\left[\big(W^{A}_0\big)^2\right]^{1/2}=1$. The above bound on $ l_n(u)$ ensures that $n(\varepsilon)\leq M^{2} \log^3 M$, for $M=\left( C_3d^{6} \right)^d \bar{A}^{d} \varepsilon^{-1}$. This bound combined with Lemma \ref{lemma: bound on sigma} then gives
\[ \varphi_0^A(\varepsilon) \leq C^dd^{c_1d}\Bar{A}^{c_2d}\varepsilon^{-c_3},\]
which we plug into \eqref{eq: crude ineq for small ball} with \eqref{eq: bound entropy dim}
\[  \varphi_0^A(\varepsilon) \leq C^d d^{c d} \log\left(\frac{\bar{A}}{\varepsilon}\right)^{1+d}\ \prod_{i=1}^d A_i, \]
 for $0<\varepsilon\leq 4$. This concludes the proof.
  \end{proof}

\subsection{Auxiliary lemmas for the proof of Lemmas \ref{lemma:: approx rkhs bound} and \ref{lemma: small ball proba}}
  
 \begin{lemma}\label{lemma bound entropy unit ball rkhs}
 Let $\mathbb{H}^A_1$ be the unit ball of $\mathbb{H}^A$ and $B^2\coloneqq  \int e^{ \norm{t}_2/2}d\nu(t)$. Assume $A_i\geq 1/(96\sqrt{d})$, for $i=1,\dots,d$. There exist absolute constants $C_1,C_2$ such that, for $0<\varepsilon< 2B(C_1d)^d$,
 \[ \log\ N\left(\varepsilon, \mathbb{H}^A_1, \norm{\cdot}_\infty\right)\leq C_2^d d^{4d} \log\left(\frac{1}{\varepsilon}\right)^{1+d} \prod_{i=1}^d A_i.\] 
 and, for $\varepsilon\geq 2B(C_1d)^d$, \[ \log\ N\left(\varepsilon, \mathbb{H}^A_1, \norm{\cdot}_\infty\right)=0.\]
 \end{lemma}
\begin{proof}

 Lemma 4.1 in \cite{bpd14} states that the elements from $\mathbb{H}_1^A$ are the real parts of functions $t\mapsto \int e^{i \langle s, t \rangle} \eta(s) d\nu_A(s)$, for $\nu_A$ the measure with Lebesgues density $s\mapsto \prod_i \frac{1}{ \sqrt{4\pi}A_i}e^{-s_i^2/4A_i^2}$, defined on $[-1,1]^d$ and such that $\eta$ takes values in $\mathbb{C}$ and satisfies $\int |\eta(s)|^2 d\nu_A(s)\leq 1$. 
 
Let's take an arbitrary function as above and write $h$ is extension to $\mathbb{C}^d$ (defined as above, for $t\in\mathbb{C}^d$). Cauchy-Schwartz inequality and a change of variable ensure that 
\[ |h(t)|^2 \leq \int_{\R^d} e^{\langle s, 2 A\cdot \text{Re}(t)\rangle} d\nu(s).\]
 Here, Re is the real part applied coordinate-wise to the vector $t$ of complex coordinates and $A\cdot \text{Re}(t) = (A_1\text{Re}(t_1),\dots,A_d\text{Re}(t_d))$. As in the proof of Lemma 4.4 in \cite{bpd14}, this ensures, along with the dominated convergence theorem and $\int_{\R_d} e^{\frac{1}{2}\norm{s}_2}d\nu(s)<\infty$, that $h$ can be extended analytically to 
 \[ U = \left\{z\in\mathbb{C}^d\colon\ \norm{2 A\cdot \text{Re}(z)}_2< 1/2\right\},\]
 a set which contains the strip \[ \mathcal{S} = \left\{z\in\mathbb{C}^d\colon\ |\text{Re}(z_j)|\leq R_j = \frac{1}{12A_j\sqrt{d}},\ j=1,\dots,d\right\} .\]For any $z\in \mathcal{S}$, we can see that $ |h(z)|^2 \leq \int_{\R_d} e^{\frac{1}{2}\norm{s}_2}d\nu(s)=B^2$.
 With $R=(R_1,\dots,R_d)$, we partition $[-1,1]^d$ into rectangles $V_1,\dots,V_m$, with centers $t_1,\dots,t_m$, such that for any $z\in[-1,1]^d$ there exists $t_j$ such that $|z_i-(t_{j})_i|\leq R_i/4$, $i=1,\dots,d$. We can find such partition with $m\leq \prod_{j=1}^d \left(1\vee \lceil8/R_j\rceil\right) = \prod_{j=1}^d \left(1\vee \lceil96A_j \sqrt{d}\rceil\right) $. We define the piecewise polynomials $P= \sum_{j=1}^m P_{j,p_j} \mathds{1}_{V_j}$ of arbitrary degree $k$ such that \[ P_{j,p_j}(t)=\sum_{|n|\leq k} p_{j,n}(t-t_j)^n.\]
Here the sum ranges over all multi-index vectors $n = (n_1,\dots,n_d) \in \N^d$ with $|n|=n_1+\dots+n_d \leq k$, and for $s=(s_1,\dots,s_d)\in\R^d$ the notation $s^n$ denotes $s_1^{n_1}\cdots s_d^{n_d}$. We obtain a finite number of such functions by discretizing the coefficients $p_{j,n}$ for any $j$ and $n$ over a grid of mesh $\varepsilon/R^n$ in $[-B/R^n;B/R^n]$, for arbitrary $\varepsilon>0$ and where $B$ is as above and $R^n=\prod_{j} R_j^{n_j}$. The log cardinality of this set is then 
\[ \log\left(\prod_{j=1}^m \prod_{|n|\leq k} \left\lceil\frac{2B/R^n}{\varepsilon/R^n} \vee 1\right\rceil\right) \leq mk^d \log(2B/\varepsilon\ \vee 1).\]
Applying the Cauchy formula $d$ times in the circles $O_j$ of radius $R_j$ and center $t_{ij}$ (i.e., the coordinates of the point $t_i$) in the complex plane and noting $\partial^n$ the partial derivative of orders $n = (n_1,\dots,n_d)$ and $n!=n_1!n_2!\dots n_d!$,

\[ \left|\frac{\partial^n h(t_i)}{n!}\right| \leq \frac{1}{(2\pi)^d}\left|\oint_{O_1}\cdots \oint_{O_d} \frac{h(z)}{(z-t_i)^{n+1}} dz_1\dots dz_d\right|\leq B/ R^n.\]
Consequently, for $z\in V_j$, and appropriately chosen $p_j$, we find 

\[ \left|\sum_{|n|> k}\frac{\partial^n h(t_i)}{n!} (z-t_i)^n \right| \leq B \sum_{|n|> k} \frac{1}{R^n}(R/4)^{n} \]
and 
\[ \left|\sum_{|n|\leq  k}\frac{\partial^n h(t_i)}{n!} (z-t_i)^n  - P_{j,p_j}(z)\right| \leq \sum_{|n|\leq  k} \frac{\varepsilon }{R^n}(R/2)^n.\]
The sum of these two terms is upper bounded by 
\[B\sum_{l=k+1}^\infty \frac{(l+1)^{d-1}}{2^l} + \varepsilon \sum_{l=1}^k \frac{(l+1)^{d-1}}{2^l}\leq 2B\sum_{l=k+2}^\infty \frac{l^{d-1}}{2^l} + \varepsilon \sum_{l=1}^k \frac{(l+1)^{d-1}}{2^l}.\]
One first checks that  $l^{d-1}2^{-l} \leq (2/3)^l$ for  $l\geq \left(d/\log(4/3)\right)^2$ (since  $\log l\le \sqrt{l}$ for $l\ge 1$). So, as long as   
\begin{equation} \label{ineqk}
 k \ge \log(3/2)^{-1}\log\left(\varepsilon^{-1}\right)\vee  \left(d/\log(4/3)\right)^2, 
\end{equation} 
we have $\sum_{l=k+1}^\infty l^{d-1} 2^{-l} \le 2 (2/3)^{k}\leq 2\varepsilon$. On the other hand, computing the $q$-th derivative of $(1-x)^{-1}$ at $x=1/2$, one finds that  
\[ \sum_{l=1}^k \frac{l^{d-1}}{2^l}\leq \sum_{l=1}^\infty (l+1)\dots(l+d-1)2^{-l}\leq (d-1)!2^d.\] 
Combining the previous bounds and choosing $k$ as the smallest integer verifying inequality \eqref{ineqk}, one deduces that there exists an $\veps(Cd)^d$--covering, i.e. with radius
\[ \varepsilon\left(4B+(d-1)!2^d\right)\leq \varepsilon\left(4e^{1/4}2^d/\sqrt{\pi}+(d-1)!2^d\right)\leq \varepsilon (Cd)^d,\]
where we used Lemma \ref{lem-v} in the first inequality. Therefore, we have constructed an $\varepsilon(Cd)^d$ covering of $\mathbb{H}^A_1$ and, writing $\Tilde{\varepsilon} = \varepsilon(Cd)^d$, we conclude that, for $\Tilde{\varepsilon}$ small enough,
 \begin{align*}
 &\log\ N\left(\Tilde{\varepsilon}, \mathbb{H}^A_1, \norm{\cdot}_\infty\right)\\&\leq \left\lceil \log(3/2)^{-1}\log\left(\frac{(Cd)^d}{\Tilde{\varepsilon}}\right)\vee  \left(d/\log(4/3)\right)^2 \right\rceil^d \log(2B(Cd)^d/\Tilde{\varepsilon}\ \vee 1) \prod_{j=1}^d \left(1\vee \lceil96A_j \sqrt{d}\rceil\right)
 \end{align*}
 This bound is null for $\Tilde{\varepsilon}\geq 2B(Cd)^d$ and otherwise upper bounded by
 \[ C_2^d d^{4d}\log\left(\Tilde{\varepsilon}^{-1}\right)^{1+d}\prod_{j=1}^d A_j,\]
 for $C_2$ an absolute constant, which proves the assertion in the lemma. In the last display, we used Lemma \ref{lem-v} to bound $B$.

 
\end{proof}

  \begin{lemma}\label{lemma: bound on sigma}
  Let $d\geq1$, $A\in \R_+^d$ with $A_i\geq 1$, $W^A$ as in \eqref{eq: gp prior dilatation}  and $\sigma=\mathbb{E}\left[\norm{W^A}_\infty^2 \given A \right]^{1/2}$.
  Then, there exists a universal constant $C>0$ such that, for $\Bar{A}=\max_i A_i>\sqrt{\log 2}/(2\sqrt{d})$,
  \[ \sigma^2\leq C d\log(d\Bar{A}).\]
  \end{lemma}
  \begin{proof}
   Since $\sigma^2= \mathbb{E}\left[\norm{W^A}_\infty\given A \right]^2 + \mathbb{V}\left[\norm{W^A}_\infty\given A \right]$, we bound these two terms. First \[\sup_{t\in[-1,1]^d}\ \mathbb{E} [W_{At}^2\given A]=1\] and Lemma \ref{lemma: gaussian sup concentration} gives the bound on the tail probability \[P\left(\big|\norm{W^A}_\infty-\mathbb{E}\norm{W^A}_\infty\big|> u \given A \right)\leq 2e^{-u^2/2}.\] Then
 \[ \mathbb{V}[\norm{W^A}_\infty\given A] = \int_0^\infty 2x P\left(\big|\norm{W^A}_\infty-\mathbb{E}\norm{W^A}_\infty\big|> x\given A \right)dx \leq 4 \int_0^\infty x e^{-x^2/2}dx=4.\]
We control the other term via Lemma \ref{lemma: dudley bound} as follows
\[ \mathbb{E}[\norm{W^A}_\infty \given A]\leq \mathbb{E}\left|X\right| + 4\sqrt{2}\int_0^{S/2} \sqrt{\log\left(2N(\varepsilon,[-1,1]^d,D)\right)}d\varepsilon\]
 for $X\sim\mathcal{N}(0,1)$, $D(s,t)=2\left(1-e^{-\sum_i A_i^2 (t_i-s_i)^2}\right)$ and \[S\coloneqq\sup\left\{s,t\in[-1,1]^d:\ D(s,t)\right\}=2\left(1-e^{-4\sum_iA_i^2}\right).\] 
 We note that, for $\varepsilon<1$, $\norm{s-t}_{2}\leq \sqrt{-\log(1-\varepsilon/2)}/\Bar{A}$ implies $D(s,t)\leq \varepsilon$, and as a result, for $B_0(r)$ the euclidean ball in $\R^d$ of radius $r>0$ around $0$, \[N\left(\varepsilon,[-1,1]^d,D\right)\leq N\left(\sqrt{-\log(1-\varepsilon/2)}/\Bar{A},B_0(2\sqrt{d}),\norm{\cdot}_{2}\right)\leq \left(\frac{6\sqrt{d}\Bar{A}}{\sqrt{-\log(1-\varepsilon/2)}}\right)^d\]
 by standard arguments (see Lemma C.2 in \cite{MR3587782}, using that $\sqrt{-\log(1-\varepsilon/2)}/\Bar{A}\leq 2\sqrt{d}$ by assumptions). The above integral is then bounded by
 \begin{align*} 
 &\int_0^{1-e^{-4d\Bar{A}^2}} \sqrt{\log\left(2N(\varepsilon,[-1,1]^d,D)\right)}d\varepsilon\\
  &\le \sqrt{d} \left[\sqrt{\log(12d\Bar{A})} + \int_0^{1-e^{-4d\Bar{A}^2}} \sqrt{\log\left(\frac{1}{\sqrt{-\log(1-\varepsilon/2)}}\right)}d\varepsilon \right]\\
  & \le \sqrt{d} \left[\sqrt{\log(12d\Bar{A})} + c_0\right]
  \le C \sqrt{d\log(d\Bar{A})},
 \end{align*}
 for universal constants $c_0, C$, where we have used $\sqrt{a+b}\le \sqrt{a}+\sqrt{b}$ for $a,b>0$.  
Therefore $\sigma^2\lesssim d\log(d\Bar{A})$ as announced.
  \end{proof}

\section{The horseshoe density}\label{appendix A}

 
 Recall that $\pi_\tau$ denotes the horseshoe density (Section \ref{section: Deep Horseshoe Gaussian Process}). Then, for any $t>0$,
 \begin{align}\label{eq: horseshoe exact}
 \pi_\tau(t)&=\frac{4}{\pi}\frac{1}{\sqrt{2\pi}\tau}\ \int_{\R^+} \frac{1}{\lambda(1+\lambda^2)} e^{-\frac{t^2}{2\lambda^2\tau^2}} d\lambda\\
 &\underset{v=\lambda^{-2}}{=} \frac{2}{\sqrt{2\pi^3}\tau} \int_{\R^+} \frac{1}{v+1} e^{-\frac{t^2}{2\tau^2}v}dv= 2\frac{e^{t^2/(2\tau^2)}}{\sqrt{2\pi^3}\tau}  \underbrace{\int_1^{+\infty} \frac{1}{v} e^{-\frac{t^2}{2\tau^2}v}dv }_{\eqqcolon E_1\left(\frac{t^2}{2\tau^2}\right)}\nonumber.
 \end{align}
It is known that (see Chapter 5 in \cite{MR1225604}), for $x>0$,
\[ \frac{1}{2}e^{-x}\log\left(1+\frac{2}{x}\right)<E_1(x)<e^{-x}\log\left(1+\frac{1}{x}\right),\]
 so that we have the bound (see also the appendix of \cite{carvalho10}), for $t>0$,
 \begin{equation}\label{eq: lower bound horseshoe}
 \frac{2}{(2\pi)^{3/2}\tau}\log\left(1+\frac{4\tau^2}{t^2}\right)<\pi_\tau(t)<\frac{2}{\sqrt{2\pi^3}\tau}\log\left(1+\frac{2\tau^2}{t^2}\right).
 \end{equation}
The following lemma gives bounds on the probabilities of events of interest in the study of the different priors we study in the paper.
  
  \begin{lemma}\label{lemma: bound on the horseshoe}
Let $d^*\geq 1$ and $\beta>0, \ta>0$. Then for any $0<\delta\le \ta$,
\begin{equation} \label{hs-smalltau}
\int_0^{\delta} \pi_\tau(t) dt\ge e_0 (\delta/\ta),
\end{equation}
where $e_0=2(\log{5})/(2\pi)^{3/2}$. 
Also, for $\delta>0$,
\begin{equation}\label{eq: prob bound small K}\int_0^{\delta} \pi_\tau(t) dt\geq 1-\frac{4\tau}{\sqrt{2\pi^3} \delta}.\end{equation}
For any $\tau \leq 1 \le a$,
\begin{equation}\label{eq: prob bound smoothing high}
 \int_{a}^{2a} \pi_\ta(t)dt \ge \exp(-\log(10a/\ta)). 
\end{equation} 

  \end{lemma}
  \begin{proof}
The first inequality follows from the lower bound in  \eqref{eq: lower bound horseshoe} noting that the logarithm in the integral is at least $\log{5}$ if $\delta\le\ta$. 
For the second inequality, the upper bound in \eqref{eq: lower bound horseshoe} gives
   \begin{align*}
\int_0^{\delta} \pi_\tau(t) dt &\geq 1- \int_\delta^{\infty} \frac{2}{\sqrt{2\pi^3}\tau}\log\left(1+ \frac{2\tau^2}{t^2}\right)  dt \nonumber\\
  &\geq 1-  \frac{4\tau}{\sqrt{2\pi^3}} \int_\delta^{\infty} t^{-2} dt \geq  1-  \frac{4\tau}{\sqrt{2\pi^3}\delta}.
  \end{align*}
 For the third inequality, one first bounds $\pi_\ta$ from below
using \eqref{eq: lower bound horseshoe} and next use $\log\left(1+x\right)\geq  \log(5)x/4$ for $x\le 4$ by concavity, noting that $4\ta^2/t^2\le \ta^2/a^2\le 1$ for $t\in[a,2a]$ since one assumes $\ta\le a$.  One deduces
\[  \int_{a}^{2a} \pi_\ta(t)dt \ge (2/(2\pi)^{3/2})\ta^{-1} \log{5} \int_a^{2a} (\ta^2/t^2)dt
\ge \exp(-\log(Ca/\ta)),\]
where $C=(2\pi)^{3/2}/\log{5}<10$, which concludes the proof.

  \end{proof}

\section{Additional lemmas}



\begin{lemma}\label{lemma: structure regression distance}
Let $q\geq0$ and $d_i\geq 1$ for $0\leq i \leq q+1$, with $d_0=d$ and $d_{q+1}=1$. For $0\leq i\leq q,\ 1\leq j \leq d_{i+1},\ 1\leq t_i\leq d_i$ and some $\beta_i>0$, $K>0$, let $h_{ij}:\ [-1,1]^{d_i} \to [-1,1]\in \cF_{VS}(K,\beta_i,d_i,t_i)$ be a function that depends on a subset $S_{ij}$ of $t_i$ coordinates and such that the restriction $\restr{h_{ij}}{S_{ij}}$ satisfies $\norm{\restr{h_{ij}}{S_{ij}}}_{\beta_i,\infty}\leq K$.

Then, the maps $h_i=\left(h_{ij}\right)_{j=1,\dots,d_{i+1}}^T$ satisfy for any $\Tilde{h}_i=\left(\Tilde{h}_{ij}\right)_{j=1,\dots,d_{i+1}}^T$, with $\Tilde{h}_{ij}:[-1,1]^{d_i}\to [-1,1]$,
\[ \norm{h_q\circ\dots h_0 - \Tilde{h}_q \circ\dots \Tilde{h}_0}_{\infty}\leq \prod_{i=0}^q [2^{|\beta_{i}-1|} t_{i}K\vee 1]\sum_{i=0}^q \norm{\left|h_i-\Tilde{h}_i\right|_\infty}_{\infty}^{\alpha_i}\]
with $\alpha_i=\prod_{l=i+1}^q \beta_l\wedge 1$ (and $\alpha_q=1$ by convention).
\end{lemma}
\begin{proof}
We follow the proof of Lemma 16 in \cite{fsh21} and prove the assertion by induction. For $q=0$, the bound is trivial. 
For $q=k+1>0$, assume that the statement is true for the nonnegative integer $k$. We write $H_k=h_k\circ\dots h_0$ and $\Tilde{H}_k=\Tilde{h}_k\circ\dots \Tilde{h}_0$ and use the triangle inequality so that
\begin{align*}
&\left|h_{k+1}\circ H_k(x) - \Tilde{h}_{k+1}\circ \Tilde{H}_k(x) \right|_\infty\\
&\leq \left|h_{k+1}\circ H_k(x) - h_{k+1}\circ \Tilde{H}_k(x) \right|_\infty + \left|h_{k+1}\circ \Tilde{H}_k(x)- \Tilde{h}_{k+1}\circ \Tilde{H}_k(x) \right|_\infty
\end{align*}
To bound the first term, we note that for any $1\leq j\leq d_{k+2}$, if $\beta_{k+1}\leq 1$, 
\[ \left|h_{(k+1)j}\circ H_k(x) - h_{(k+1)j}\circ \Tilde{H}_k(x) \right|\leq K \left|H_k(x) - \Tilde{H}_k(x) \right|_\infty^{\beta_{k+1}},\]
while for $\beta_{k+1}> 1$, we use that for any $s\in I^{d_{k+1}}$, $h_{(k+1)j}(s)$ is equal to $g(s')$ for $g\colon I^{t_{k+1}}\mapsto [-1,1]$ and $s'$ consisting of elements of $s$ whose indices are in $S_{(k+1)j}$. By the mean-value theorem $h_{(k+1)j}\circ H_k(x) - h_{(k+1)j}\circ \Tilde{H}_k(x)=\nabla{g}(c)^T\cdot (H_k(x)-\Tilde{H}_k(x))$ for some $c\in I^{t_{k+1}}$, implying
\[ \left|h_{(k+1)j}\circ H_k(x) - h_{(k+1)j}\circ \Tilde{H}_k(x) \right|\leq t_{k+1}K \left|H_k(x) - \Tilde{H}_k(x) \right|_\infty,\]
where we used the regularity assumption on $g$.
Then, we note that $\left|H_k(x) - \Tilde{H}_k(x) \right|_\infty\leq 2$ and, for any $\beta>0$ and $0\leq x\leq 2$,  $\max(x,x^\beta)\leq 2^{|\beta-1|} x^{1\wedge\beta}$. Therefore,
\begin{align*}
&\left|h_{k+1}\circ H_k(x) - \Tilde{h}_{k+1}\circ \Tilde{H}_k(x) \right|_\infty\\
&\leq 2^{|\beta_{k+1}-1|}t_{k+1}K\left|H_k(x) - \Tilde{H}_k(x) \right|_\infty^{\beta_{k+1}\wedge1} + \left|h_{k+1}\circ \Tilde{H}_k(x)- \Tilde{h}_{k+1}\circ \Tilde{H}_k(x) \right|_\infty\\
&\leq  2^{|\beta_{k+1}-1|} t_{k+1}K\left(\prod_{i=0}^k [2^{|\beta_{i}-1|} t_{i}K\vee 1] \sum_{i=0}^k \norm{\left|h_i-\Tilde{h}_i\right|_\infty}_{\infty}^{\prod_{l=i+1}^k \beta_l\wedge 1} \right)^{\beta_{k+1}\wedge1}  + \norm{\left|h_{k+1}- \Tilde{h}_{k+1}\right|_\infty}_\infty\\
&\leq \prod_{i=0}^{k+1} [2^{|\beta_{i}-1|} t_{i}K\vee 1]\sum_{i=0}^k \norm{\left|h_i-\Tilde{h}_i\right|_\infty}_{\infty}^{\prod_{l=i+1}^{k+1} \beta_l\wedge 1} + \prod_{i=0}^{k+1} [2^{|\beta_{i}-1|} t_{i}K\vee 1]\norm{\left|h_{k+1}- \Tilde{h}_{k+1} \right|_\infty}_\infty\\
&= \prod_{i=0}^{k+1} [2^{|\beta_{i}-1|} t_{i}K\vee 1]\sum_{i=0}^{k+1} \norm{\left|h_i-\Tilde{h}_i\right|_\infty}_{\infty}^{\alpha_i}
\end{align*}
where we use that $(y+z)^\alpha\leq y^\alpha+x^\alpha$ for $y,z\geq 0$, $\alpha\in[0,1]$.
\end{proof}

In the paper, we also use the following results on the concentration of the supremum of a Gaussian process around its mean, and a bound on it. It applies to the different Gaussian processes  considered  as they are defined on $\mathbb{R}^d$ and their paths are almost-surely continuous.
We recall that the sample paths of the {\em squared exponential} process \textsf{SqExp} are almost surely continuous. As a consequence, the separability condition (as defined in \cite{ ginenickl_book}) of the next lemma is satisfied with continuous transformation and combinations of such processes.

\begin{lemma}[Gaussian supremum concentration, Theorem 2.5.8 from \cite{ ginenickl_book}]\label{lemma: gaussian sup concentration}
Let $W(t)$, $t\in T$ be a separable centred Gaussian process whose supremum is finite with positive probability. Let $\sigma^2$ be the supremum of the variances $E W(t)^2$ and $\norm{W}_\infty = \underset{t\in T}{\sup} |W(t)|$. Then,
\[ P\left(\norm{W}_\infty \geq E \norm{W}_\infty + u \right) \leq e^{-u^2/2\sigma^2},\quad P\left(\norm{W}_\infty \leq E \norm{W}_\infty - u \right) \leq e^{-u^2/2\sigma^2}. \]
\end{lemma}

\begin{lemma}[Expected supnorm bound on the Gaussian process, Theorem 2.3.7 of \cite{ginenickl_book}]\label{lemma: dudley bound}
Let $W(t)$, $t\in T$, be a Gaussian process, defining the metric $d\left(s,t\right)^2 \coloneqq \mathbb{E} \left|X(s)-X(t)\right|^2$ on $T$ and such that
\[ \int_0^{\infty} \sqrt{\log\ N\left(u, T,d\right)}du<\infty.\]
Then, for $t_0\in T$ and $\underset{s,t\in T}{\sup}\ d(s,t) = 2\sigma$,
\[ E \norm{W}_\infty \leq E |W(t_0)| + 4\sqrt{2}\int_0^{\sigma} \sqrt{\log\ 2N\left(u, T,d\right)}du.\]
\end{lemma}

\begin{lemma} \label{lem-v}
Let $\nu$ be the spectral measure of the squared-exponential \textsf{SqExp} process.  Set $B:=\int e^{\norm{t}_2/2}d\nu(t)$. Then, for any integer $d\ge 1$, for $c_0=e^{1/4}/\sqrt{\pi}$,
\begin{equation*} \label{eq: upper bound V}
B \le c_0 2^d.
\end{equation*}
\end{lemma}
\begin{proof}
Using the explicit expression of the spectral measure $\nu$, one can write
\[\int e^{ \norm{t}_2/2}d\nu(t)= \left(2^d \pi^{d/2}\right)^{-1} \int e^{\norm{t}_2/2-\norm{t}_2^2/4} d\lambda(t)=\frac{1}{2^{d-1}\Gamma(d/2)}\int_0^{\infty} r^{d-1} e^{r/2-r^2/4} dr,  \]
by change of variables in the integration of a radial function and using that the surface area of the sphere in dimension $d$ equals $2\pi^{d/2}/\Gamma(d/2)$. The above integral equals
\begin{align*}
&\lefteqn{ e^{1/4} \int_0^{\infty} r^{d-1} e^{-\frac{(r-1)^2}{4}} dr 
= e^{1/4} \int_{-1}^{\infty} (r+1)^{d-1} e^{-\frac{r^2}{4}} dr } \\
& \le  e^{1/4} \left[ 2^{d-1}+ \int_1^\infty (2r)^{d-1}e^{-\frac{r^2}{4}} dr \right]
\le  e^{1/4}  2^{d-1} \left[ 1 + \sqrt{4\pi}E |\cN(0,2)|^{d-1}  \right].
\end{align*}
Using the formula $E |\cN(0,1)|^p = 2^{p/2} \Gamma(\{p+1\}/2)/\sqrt{\pi}$ and that this quantity is always at least $1/\sqrt{2}$ for $p\ge 0$, the last expression under brackets in the last display is at most $2 E |\cN(0,2)|^{d-1}=  2^d\Gamma(d/2)/\sqrt{\pi}$, which concludes the proof.
\end{proof}




\section{H\"older spaces in growing dimension} \label{app:hold}
 
For a real $\be>0$ and $d$ an integer, let us denote by $\mathcal{C}^{\beta}[-1,1]^{d}$  the classical Hölder space equipped with the following norm $\norm{\,\cdot\,}_{\beta,\infty}$ (as we allow the dimension $d$ to possibly increase with $n$, the choice of norm plays a role): it consists of functions $f:[-1,1]^d\to \RR$ whose norm defined as
 \begin{equation} \label{defno_app}
  \norm{f}_{\beta,\infty}=  \max\left(\max_{ |\boldsymbol{\alpha}|\le \lfloor \beta \rfloor} \norm{\partial^{\boldsymbol{\alpha}} f}_{\infty} , \max_{\boldsymbol{\alpha}: |\boldsymbol{\alpha}|=\lfloor \beta \rfloor} \underset{\boldsymbol{x}, \boldsymbol{y}\in [-1,1]^d,\ \boldsymbol{x}\neq \boldsymbol{y}}{\sup} \frac{\left|\partial^{\boldsymbol{\alpha}} f(\boldsymbol{x}) - \partial^{\boldsymbol{\alpha}} f(\boldsymbol{y})\right|}{\left|\boldsymbol{x}- \boldsymbol{y}\right|^{ \beta-\lfloor \beta \rfloor}_{\infty}}\right) 
\end{equation}  
 is finite, 
with the multi-index notation $\boldsymbol{\alpha}=(\alpha_1,\dots,\alpha_d)\in\N^d$, $|\boldsymbol{\alpha}|\coloneqq|\boldsymbol{\alpha}|_1$ and $\partial^{\boldsymbol{\alpha}} =\partial^{\boldsymbol{\alpha}_1} \dots \partial^{\boldsymbol{\alpha}_d}$.

 \begin{example}
Suppose that $f:I^d\to\RR$ is a bounded function with all derivatives $\partial^{\boldsymbol{\al}} f$ bounded by a constant $M$ (possibly depending on $d$) on $I^d$, for $|\boldsymbol{\al}| < \lfloor \be \rfloor$ and $\left|\partial^{\boldsymbol{\alpha}} f(\boldsymbol{x}) - \partial^{\boldsymbol{\alpha}} f(\boldsymbol{y})\right|\le M\left|\boldsymbol{x}- \boldsymbol{y}\right|^{ \beta-\lfloor \beta \rfloor}_{\infty}$ for $|\boldsymbol{\al}| = \lfloor \be \rfloor$. Then 
$\norm{f}_{\beta,\infty}\le M$. 

As a special example, let us consider the class of functions $g:I^d\to \RR$ that are of the form 
\[ g(x_1,\ldots,x_d) = g_1(x_1)\cdots g_d(x_d),\]
where $g_j$ are univariate  functions and suppose 
 $\max_{j} \| g_j\|_{\be,\infty} \le M$, for the norm $\|\cdot\|_{\be,\infty}$ as defined above (now in the special case $d=1$). This means in particular that all $1$-dimensional derivatives of $g_j$'s are bounded by $M$. A simple calculation (Lemma \ref{lemder}) shows that the $\|\cdot\|_{\be,\infty}$--norm of $g$ is bounded by $2d M^d$. In particular if $M\le M_0<1$ then $\|g\|_{\be,\infty}$ is uniformly bounded in $d$. Otherwise if $M\ge 1$ we have $\|g\|_{\be,\infty}\le (C'M)^d=C^d$ for large $C,C'>0$, which shows a growth in $C^d$ for the Hölder norms of such product functions in dimension $d$.
\end{example}

\begin{lemma} \label{lemder}
Let  $\be, M>0$ and $g(x_1,\ldots,x_d) = g_1(x_1)\cdots g_d(x_d)$, 
where $g_j$ are univariate  functions with  
 $\max_{j} \| g_j\|_{\be,\infty} \le M$, for $\|\cdot\|_{\be,\infty}$ as in \eqref{defno_app}. Then
\[ \|g\|_{\be,\infty} \le 2dM^d. \]
\end{lemma}
\begin{proof}
One first notes, since $\partial^{\boldsymbol{\alpha}} g = \partial^{\al_1}g_1\cdots \partial^{\al_d}g_d$ that $\norm{\partial^{\boldsymbol{\alpha}} g}_{\infty}\le M^d$ using that, since $g_j$'s have $\|\cdot\|_{\be,\infty}$--norms bounded by $M$, each term is bounded by $M$. So, the first term in  \eqref{defno_app} is at most $M^d$. Turning to the second term in \eqref{defno_app}, suppose first $\lfloor \be \rfloor=0$. Then one writes $g(\boldsymbol{x}) - g(\boldsymbol{y})$ as a telescopic sum
\[ g(\boldsymbol{x}) - g(\boldsymbol{y}) =
\sum_{k=1}^d g(\boldsymbol{x}_k) - g(\boldsymbol{x}_{k-1})
\]
with $\boldsymbol{x}_0=\boldsymbol{x}, \boldsymbol{x}_1=(y_1,x_2,\ldots,x_d),  \boldsymbol{x}_2=(y_1,y_2,x_3,\ldots,x_d),\ldots,\boldsymbol{x}_d=\boldsymbol{y}$, and 
\[ |g(\boldsymbol{x}_k) - g(\boldsymbol{x}_{k-1})|\le |g_k(y_k)-g_k(x_k)| (\max_{i\neq k}\|g_i\|_\infty )^{d-1}
\le M^{d-1} M|x_k-y_k|^{\be}\le M^{d} \|\boldsymbol{x}-\boldsymbol{y}\|_\infty^{\be}, \]
using $\|g_i\|_{\be,\infty}\le M$ for all $i$. By the triangle inequality, this shows that if $\lfloor \be \rfloor=0$ then the last term in \eqref{defno_app} is at most $dM^d$. If now $\lfloor \be \rfloor\ge 1$, then since $|\boldsymbol{\al}| = \lfloor \be \rfloor$ for the last term in \eqref{defno_app}, there is at least one index with $\al_i\neq 0$. Without loss of generality suppose $\al_1\ge 1$. Then 
\[ \partial^{\boldsymbol{\alpha}} g(\boldsymbol{x}) - \partial^{\boldsymbol{\alpha}} g(\boldsymbol{y}) 
=\partial^{\al_1}g_1(x_1)(\partial^{\boldsymbol{\alpha_-}}g_{-}(\boldsymbol{x_-})-\partial^{\boldsymbol{\alpha_-}}g_{-}(\boldsymbol{y_-}))
+(\partial^{\al_1}g_1(x_1)-\partial^{\al_1}g_1(y_1))\partial^{\boldsymbol{\alpha_-}}g_{-}(\boldsymbol{y_-}),
\]
where $\boldsymbol{\alpha_-}=(\al_2,\ldots,\al_d), \boldsymbol{x_-}=(x_2,\ldots,x_d), \boldsymbol{y_-}=(y_2,\ldots,y_d)$ and \[g_-( \boldsymbol{x_-})=g_2(x_2)\cdots g_d(x_d).\]
Otherwise, since $\al_1\ge 1$ we have $|\boldsymbol{\alpha_-}|\le p-1$ which means $h:=\partial^{\boldsymbol{\alpha_-}}g_-$ is differentiable and its gradient has coordinates bounded by $M^{d-1}$. By the mean-value theorem $h( \boldsymbol{x_-})-h( \boldsymbol{y_-})=\nabla{h}(c)\cdot (\boldsymbol{x_-}-\boldsymbol{y_-})$ for some $c\in I^{d-1}$, so 
\[ |h( \boldsymbol{x_-})-h( \boldsymbol{y_-})|\le \|\nabla{h}(c)\|_1 \|\boldsymbol{x_-}-\boldsymbol{y_-}\|_\infty
\le (d-1)M^{d-1}\|\boldsymbol{x}-\boldsymbol{y}\|_\infty. \]
Using that $\|g_i\|_{\be,\infty}\le M$ leads to 
\[ | \partial^{\boldsymbol{\alpha}} g(\boldsymbol{x}) - \partial^{\boldsymbol{\alpha}} g(\boldsymbol{y}) | \le (d-1)M^d\|\boldsymbol{x}-\boldsymbol{y}\|_\infty+ |\partial^{\al_1}g_1(y_1)-\partial^{\al_1}g_1(x_1)|M^{d-1}.
\]
For the last absolute value, if $\al_1=\lfloor \be \rfloor$, we have $|\partial^{\al_1}g_1(y_1)-\partial^{\al_1}g_1(x_1)|\le M|x_1-y_1|^{\be-\lfloor \be \rfloor}$, otherwise if $\al_1<\lfloor \be \rfloor$ using the mean-value theorem as before leads to 
$|\partial^{\al_1}g_1(y_1)-\partial^{\al_1}g_1(x_1)|\le M|x_1-y_1|$. Putting all bounds together leads to 
\[ |\partial^{\boldsymbol{\alpha}} g(\boldsymbol{x}) - \partial^{\boldsymbol{\alpha}} g(\boldsymbol{y}) |/\|\boldsymbol{x}-\boldsymbol{y}\|_\infty^{\be-\lfloor \be \rfloor}\le  2(d-1)M^d+2M^d=2dM^d.\]
\end{proof}

The following Lemma is concerned with the extension of a Hölder function, with Hölder norm defined in \eqref{defno_app}, defined on the unit cube to the whole d-dimensional real space. It shows that such extension to a Hölder-regular function is possible, with a an additional multiplicative factor depending on the regularity and the dimension in front of the norm.

\begin{lemma}[Whitney's theorem]\label{lem: Whitney}
A function $f \in \cF(\be,K,d)$ as in \eqref{defcld} can be extended to a H\"older function on the whole $\RR^d$, and the H\"older norm of the extension is at most $\cE_\be(d)K$, for $\cE_\be(d)=c(\be) C(\be)^d$ with constants $c(\be), C(\be)$ that only depends on $\be$: in particular this bound is uniform over elements of $\cF(\be,K,d)$.
\end{lemma}
\begin{proof}
It follows from the proof of extensions of Whitney's type found in Section 6 of \cite{Stein1971}.
\end{proof}

\begin{remark}
    It has recently been proved that for integer $\be$, the constant $\cE_\be(d)$ can be taken of polynomial order in $d$ (\cite{chang17} proves $\cE_\be(d)\le c(\be)d^{5\be/2}$ for integer $\be$).
\end{remark}

\section{Dimension--dependent minimax risk: lower bound} \label{app-rate-d}

The purpose of this section is to verify that the minimax rate of convergence in the nonparametric model considered in the paper and for the squared--integrated loss is under mild conditions and slow growth of the dimension $d$ (e.g. $d=o(\sqrt{\log{n}})$) of the order $n^{-2\be/(2\be+d)}$ up to a smaller order slowly varying factor. To verify this one rewrites the classical lower bound proof for the nonparametric rate by keeping the dependence on the dimension $d$ explicit. 

Consider a $\cC^\infty$ function  $K:\RR\to \RR^{+}$ that verifies, for $K_{\given I}$ its restriction to $I=[-1,1]$,
\begin{equation} \label{bump}
\| K_{\given I} \|_{\be,\infty} \le 1/2\quad \text{and}\quad K(x)>0 \iff x\in (-1/2,1/2). 
\end{equation} 
The construction of such a function is standard, see e.g. \cite{tsybakov09}, Section 2.5 (here slightly adapted to our norm definition: in case $\lfloor \be\rfloor\ge 1$, one  checks that intermediates derivatives from $1$ to $\lfloor \be\rfloor$ are bounded by $1/2$, for a constant $a$ in Eq. (2.34) in \cite{tsybakov09} small enough).

In the following, we assume for simplicity that the variance $\sigma_0^2$ of the Gaussian noise in the regression model is known to be equal to $1$.
  
\begin{lemma} \label{lb-dim}
Let $n,d \ge 1$ and $\be, D>0$. Consider the random design regression model and suppose the distribution of $X_1$ admits a density $g$ 
on $\RR$ with $0<c_g\le g \le C_g<\infty$. Then there exists an integer $N_1= N(\be,D,d)$ such that for $n\ge N_1$,
\[ \cR_n :=\inf_T \sup_{f\in \cF(D,\be,d)} E_f \| T-f \|_{L^2(\mu)}^2 \ge C(\be,D,d) n^{-\frac{2\be}{2\be+d}},\]
where, for constants $C_0(\be), C_1(\be)$ that depends only on $\be$,
\begin{align*}  
C(\be,D,d)  & = C_0(\be) C_1(\be)^{\frac{1}{2\be+d}}  \cdot \left(\frac{c_g}{C_g}\right)^d 
 \left( D \frac{\|K\|_2^d}{4d\|K\|_{\be,\infty}^d} \right)^{\frac{2d}{2\be+d}} 
 \\
N(\be,D,d) & = C 4^\beta \frac{d^2}{D^2 C_g^d} \frac{\|K\|_{\be,\infty}^d}{\|K\|_{2}^d}.
\end{align*}
\end{lemma}

Before proving Lemma \ref{lb-dim}, we give a corollary in the main case of interest in the paper of functions depending on $d^*$ coordinates only: in this case the minimax rate is bounded from below by the corresponding quantities as in Lemma \ref{lb-dim} but with $d$ replaced by the effective dimension $d^*$. Under the assumed condition on $d^*$, it is shown in the proof of the Corollary below that the result then holds for a sufficiently large $n\ge N_0$ independent on $d,d^*$.
\begin{corollary} \label{cor-vs}
Under the conditions of Lemma \ref{lb-dim}, let $1\le d^*\le d$ and suppose, for some $\veps\in(0,1/2)$, 
\begin{equation} \label{conddst}
 d^*\le (\log{n})^{1/2-\veps}.
\end{equation} 
There exists an integer $N_0=N(\be,D)$ such that, for  constants $c_2=c_2(\be), C_2=C_2(\be)$ depending only on $\be$, $c_g, C_g$ and of the choice of kernel $K$, for all $n\ge N_0$,
\[ 
\inf_T \sup_{f\in \cF_{VS}(D,\be,d,d^*)} E_f \| T-f \|_{L^2(\mu)}^2 
\ge 
\left(C_2 c_2^{d^*} D^{\frac{2d^*}{2\be+d^*}} \right)n^{-\frac{2\be}{2\be+d^*}}.
\]
In particular under \eqref{conddst}, for fixed $D$ the minimax rate for the squared--integrated loss is bounded from below by $n^{-\frac{2\be}{2\be+d^*}}$ up to a slowly-varying multiplicative factor $c_2^{d^*}$.   
\end{corollary}
\begin{proof}[Proof of Corollary \ref{cor-vs}]
Since for any $f\in \cF_{VS}(D,\be,d,d^*)$ we have $f(x_1,\ldots,x_d)=g(x_{i_1},\ldots,x_{i_{d^*}})$ for some $g\in \cC^\be(I^{d^*})$ with $\|g\|_{\be,\infty}\le K$,  one first notes that the considered minimax risk is bounded from below by
\[ \inf_{\tilde{T}} \sup_{g\in \cF(D,\be,d^*)} E_g \| \tilde{T}-g \|_{L^2(\mu)}^2, \]
where $\tilde{T}$ is a measurable function in $L^2(\mu)$ depending only on the coordinates with indices $i_1,\ldots,i_{d^*}$: indeed, denoting by $\mathbb{T}$ the orthogonal projection of $T$ onto the (closed) subspace of $L^2(\mu)$ of functions depending only on the specified coordinates, it holds $\|T-f\|_{L^2(\mu)}\ge \|\mathbb{T}-g\|_{L^2(\mu)}$, hence the claim. The result follows from applying  Lemma \ref{lb-dim} with $d=d^*$, and by noting that the condition that $n$ is larger than $N(\be,D,d^*)$ therein, together with \eqref{conddst}, translates into $n$ larger than a large enough constant $N_0$ (depending on $\be,D$ only).
\end{proof}

\begin{proof}[Proof of Lemma \ref{lb-dim}]
The proof is similar to the proof of Theorem 2.8 in \cite{tsybakov09}, by adapting the construction in $d$ dimensions. Two further differences are the assumption on the design points $X_i$'s (here these are random of law $\mu$) and the loss function, which here is the integrated $L^2(\mu)$-loss, instead of the plain $L^2$ loss for deterministic design. This creates a few minor differences as explained below. We give the details for the sake of completeness.

One uses a classical lower bound argument via many hypotheses, with squared  loss
\[ d(f,g)^2 = \int_{-1}^{1} (f(x)-g(x))^2 d\mu(x) \]
and class of functions $\Theta=\cF(D,\be,d)$. By the arguments of Section 2.2 in \cite{tsybakov09} and Theorem 2.5 therein, it suffices to find hypotheses, for suitable $M\ge 1$,
\[ \te_j = f_j(\cdot),\ \ j=0,\ldots,M,\]
with $\te_j\in\Theta$, that verify, for any $0\le j<k\le M$ and suitable $s=s_n=s_n(\be,D,d)>0$,
\begin{equation} \label{conddist}
 d(\te_j,\te_k) \ge 2s,
\end{equation} 
and for $P_j:=P_{\te_j}$, some $\alpha\in(0,1/8)$, for $K(P,Q)$ the KL--divergence between $P$ and $Q$,
\begin{equation} \label{klcond}
 \frac1M\sum_{j=1}^M K(P_j,P_0) \le \alpha\log{M}.
\end{equation} 
Then for a universal constant $C>0$, the minimax risk $\cR_n$ is bounded from below by $Cs_n$.  

Let us set, for $c_0=c_0(D,\be,d)>0$ to be chosen below,
\begin{equation} \label{defimh}
 m:= \lceil c_0 n^{\frac{1}{2\be+d}} \rceil,\quad h_n=1/m, 
\end{equation}
and suppose $n\ge N(D,\be,d)$ is large enough so that $m\ge 2$.
For $\boldsymbol{k}=(k_1,\ldots,k_d)\in \{-m+1,\ldots,m\}^d$ a multiindex, let us set
\[ x_{\boldsymbol{k}}=\left(\frac{k_1-1/2}{m}, \ldots, \frac{k_d-1/2}{m}\right).\]
This defines a grid of points in the unit cube, indexed by $\bok$. Define, for $K$ as in \eqref{bump}, $h_n$ as in \eqref{defimh}, and $L>0$ to be specified, 
\begin{equation} \label{phis}
\vphi_{\bok}(x) = Lh_n^\be \prod_{j=1}^d K\left( \frac{x_j-x_{k_j}}{h_n} \right)
=:Lh_n^\be K_d\left( \frac{x-x_{\bok}}{h_n} \right),
\end{equation} 
where $K_d(y)=K(y_1)\cdot K(y_2)\cdots \cdot K(y_d)$ for $y\in \RR^d$. Let $\Omega$ denote the set 
\[ \Omega = \left\{ \omega=(\omega_{\bok})_{\bok\in\{-m+1,\ldots,m\}^d} \in\{0,1\}^{(2m)^d} 
\right\}. \]
For any $\omega\in\Omega$ and $x\in[-1,1]^d$, denote, for $\bok$ running in  $\{-m+1,\ldots,m\}^d$,
\begin{equation} \label{fom}
 f_\omega(x) =   \sum_{\bok} \omega_\bok \vphi_\bok(x).   
\end{equation}
To conclude, it is enough to show that: {\em [Step 1]} all $f_\omega$ belongs to $\Theta$; to find {\em [Step 2]} a number $M$ of hypotheses $f_{\omega^{(j)}}$ that are `well-separated' in terms of their respective $L^2(\mu)$--distance, with separation of at least  $2s_n$, so that \eqref{conddist} holds with $s=s_n$;
to verify {\em [Step 3]} that \eqref{klcond} holds for $P_j=P_{f_{\omega^{(j)}}}$; according to the discussion above, the final rate is then $Cs_n$.   

{\em Step 1.} The fact that all $f_\omega$ are in $\Theta$ is verified in Lemma \ref{lemfom} below for the choice $L=D/(4d\|K\|_{\be,\infty}^d)$ that we make from now on.

{\em Step 2.} One uses Varshamov--Gilbert's lemma (Lemma 2.9 in \cite{tsybakov09}) to find a number $M\ge 2^{(2m)^d/8}$ of elements $\omega^{(1)}, \ldots, \omega^{(M)}\in \Omega$, with $\omega^{(1)}=(0,0,\ldots,0)$ with Hamming distance $\rho(\omega^{(j)},\omega^{(k)})$ between any of these $(j\neq k)$ being at least $(2m)^d/8$. Now setting $f_j(x)=f_{\omega^{(j)}}(x)$ for $j=0,\ldots,M$, we have for $p\neq q$ and $I=[-1,1]$,
\begin{align*} 
d(f_p,f_q)^2 & = \sum_{\bok} (\omega^{(p)}_\bok-\omega^{(q)}_\bok)^2
\int_{I^d} \vphi_\bok^2 d\mu \\
& \ge  \rho(\omega^{(p)},\omega^{(q)}) (Lh_n^\be)^2 
\left( c_g h_n\int_I K^2\right)^d\\
& \ge  (2c_g m)^d L^2 h_n^{2\be+d} \|K\|_2^{2d}/8.
\end{align*}
This shows that $s_n$ is up to a universal constant at least equal to $L 2^d c_g^d \|K\|_2^{d}m^{-\be}$.

{\em Step 3.} The final rate $s_n$ is now determined by a specific choice of $m$ arising such that \eqref{klcond} is verified. To compute $K(P_j,P_0)$ one notes first that $P_0$ has density $\phi(y)g(x)$, where $\phi$ is the standard Gaussian  density, since $f_0=0$ from our choice $\omega^{(0)}=0$ above. Since $P_j$ has density $\phi(y-f_j(x))g(x)$, the standard formula for the KL between Gaussians gives $K(P_j,P_0)=n\int f_j^2(x) g(x)dx/2$. By a similar computation as for Step 2, using now $\|g\|_\infty\le C_g(d)$,
\begin{align*}  
K(P_j,P_0) & \le n L^2 h_n^{2\be} \sum_\bok \omega_\bok^{(j)} \left(C_g h_n \int_I K^2\right)^d/2 \\
& \le n L^2 (2C_g m)^d h_n^{2\be+d}\|K\|_2^{2d}/16
=n L^2 (2C_g)^d m^{-2\be}\|K\|_2^{2d}/16.
\end{align*} 
Using $n\le (m/c_0)^{2\be+d}$ and setting $B_d:=(2C_g)^d\|K\|_2^{2d}/16$ yields $K(P_j,P_0)  \le L^2 c_0^{-2\be-d}B_d m^d$. Using $\log M\ge m^d (\log{2})/8$ from Step 2, \eqref{klcond} is satisfied if $L^2 c_0^{-2\be-d}B_d<\alpha\log(2)/8$. Taking e.g. $\al=1/16$, one may then choose 
\[ c_0 = \left( 12L^2 B_d/\al\right)^{\frac{1}{2\be+d}}.\]
The final squared--rate is then given by $s^2=L^2\|K\|_2^{2d}m^{-2\be}(2c_g)^d$ up to a universal constant. Let us choose $n$ large enough so that $c_0n^{1/(2\be+d)}\ge 2$ and $m^{-2\be}\ge (2c_0n^{1/(2\be+d)})^{-2\be}$. This gives 
\[ s^2 \ge L^2 \|K\|_2^{2d} (2c_g)^d (2c_0)^{-2\be} n^{-2\be/(2\be+d)}. \]
Inserting the expressions of $c_0, L, B_d$ gives the result, noting that $C_g^{2\be/(2\be+d)}\le C_g$. 
\end{proof}

\begin{lemma} \label{lemfom}
Let $D>0$. Suppose $L$ in \eqref{phis} is chosen as $L=D/(4d\|K\|_{\be,\infty}^d)$. Then for any $\omega\in\Omega$, the function $f_\omega$ in \eqref{fom} belongs to $\Theta=\cF(D,\be,d)$. 
\end{lemma}
\begin{proof}
First, we note that it is enough to prove that an individual function $\vphi_\bok$ belongs to $\cF(D/2,\be,d)$. Indeed, recall that the supports of the functions $\vphi_\bok$ are disjoint (up to null sets, i.e. their boundaries) small cubes that cover the unit cube $I^d$. 
 If $x,y\in I^d$ belong to the same cube centered at $x_\bok$, then $f_\omega(x) - f_\omega(y)= \vphi_\bok(x) - \vphi_{\bok}(y)$. 
 If $x,y\in I^d$ belong to different  cubes centered at $x_\bok$ and $x_{\bok'}$ respectively, consider the segment $[x;y]$ in $I^d$. It intersects the boundary of the cube centered at $x_\bok$ (resp. $x_{\bok'}$)  at a point $z_\bok$ (resp. $z'_{\bok'}$). Since $\vphi_\bok$ vanishes at the boundary of the cube centered at $x_\bok$, we have $\vphi_\bok(z_\bok)=0=\vphi_{\bok'}(z'_{\bok'})$. We write
\[ f_\omega(x) - f_\omega(y)
= \vphi_\bok(x) - \vphi_{\bok'}(y)
= \vphi_\bok(x) - \vphi_\bok(z_\bok) + \vphi_{\bok'}(z'_{\bok'})- \vphi_{\bok'}(y). 
\]
Since $\max(\|x-z_\bok\|_\infty, \|y-z'_{\bok'}\|_\infty)\le \|x-y\|_\infty$, we deduce from the previous lines that if indeed all $\vphi_\bok$ belong to $\cF(D/2,\be,d)$ then $f_\omega$ belongs to $\Theta=\cF(D,\be,d)$.

Now thus focusing on a given function $\vphi_\bok$, note that by definition it can be written $Lh_n^\be \prod_{j=1}^d \ga_j(x_j)$, with $\ga_j(x_j)=K(\{x_j-x_{k_j}\}/h_n)$. It is enough to show that the $\|\cdot\|_{\be,\infty}$--norm of the product $\prod_{j=1}^d \ga_j$ is at most $h_n^{-\be} 2d\|K\|_{\be,\infty}^d$. Proving this is a slight variation on the proof of Lemma \ref{lemder}. The only difference is the factor $1/h_n$ present within $K$ in the definition of $\ga_j$. This factor intervenes each time that one takes a derivative or bounds the ratios in \eqref{defno_app}. Each derivative makes appear an extra factor $1/h_n$ while bounding the ratio in   \eqref{defno_app} gives an extra factor $(1/h_n)^{\be-\lfloor \be\rfloor}$. Each individual bound in the proof of Lemma \ref{lemder} then contains an extra factor either $h_n^{-\be}$ or $h_n^{-q}$ for some integer $q\le \be$. Since $h_n=1/m\le 1$ we have $h_n^{-q}\le h_n^{-\be}$ so one obtains an overall bound that contains only an extra multiplicative factor $h_n^{-\be}$ compared to the bound from Lemma \ref{lemder}, which concludes the proof. 
\end{proof}

\section{Fractional posteriors}\label{app: Frac post}

For $0<\alpha<1$,  the Rényi divergence of order $\alpha$ between two densities $f$ and $g$ on a measurable space $(E, \mathcal{A}, \La)$ is given by
	\begin{align} \label{dal}
		D_\alpha(f, g) = -\frac{1}{1-\alpha}\log \left( \int_{E} f^\alpha g^{1-\alpha}d\La \right).
	\end{align}
Further define the usual Kullback-Leibler divergence $K(f, g)=\int f \log(f/g) d\La$ and its second-variation $V(f,g)= \int f \left(\log(f/g) - K(f,g)\right)^2 d\La$. 

Consider  a generic dominated statistical model $\{P_\eta^n, \eta\in S\}$, with $dP_\eta^n=p_\eta^n d\La^n$ for $\La^n$ a dominating measure. One observes $X\sim P_{\eta_0}^n$ for some unknown true parameter $\eta_0\in S$.  
Let us set, for $\veps>0$,
	\begin{align} \label{defkln}
		B_n(p_{\eta_0}^n, \eps) = B_n(\eta_0,\eps) &=\left\{\eta: \: K(p_{\eta_0}^n, p_{\eta}^n) \leq n \eps^2,\: V(p_{\eta_0}^n,  p_{\eta}^n) \leq n \eps^2  \right\},
	\end{align}	 

The following is a slight variation on Theorem 16 of \cite{alik22}.

\begin{theorem}\label{bat_result} For any nonnegative sequence $\eps_n$ and $\rho\in(0,1)$, if 		\begin{align}\label{equation_thm_1}
			\Pi(B_n(\eta_0, \eps_n)) \geq e^{-n\rho \eps_n^2},
		\end{align}
		then for $D_\rho$ the R\'enyi divergence of order $\rho$, 
		\begin{align*}
			E_0\Pi_\rho \left( \eta: \: \frac{1}{n}D_{\rho}(p_{\eta}^n, p_{\eta_0}^n) \ge 4 \frac{\rho\eps_n^2}{1-\rho}   \given  X \right) \le e^{- n\rho\eps_n^2}  + (n\eps_n^2)^{-1}.
		\end{align*}	
	\end{theorem}

\begin{proof}
		By Lemma \ref{lem:mino_den_posterior}, on a subset $C_n$ of $P_0$-probability at least $1-1/(n\eps_n^2)$,  for any measurable set $A \subset S$,
		\begin{equation*}
		\begin{split}
			E_0 \Pi_{\rho}(A \given X)  = E_0 \frac{\int_{A} \frac{p_\eta^n(X)^{\rho}}{p_{\eta_0}^n(X)^{\rho}} d\Pi(\eta)}{\int \frac{p_\eta^n(X)^{\rho}}{p_{\eta_0}^n(X)^{\rho}} d\Pi(\eta)}
			& \leq E_0 \frac{\int_{A} \frac{p_\eta^n(X)^{\rho}}{p_{\eta_0}^n(X)^{\rho}} d\Pi(\eta)}{\Pi(B_n(\eta_0, \eps_n)) e^{-2{\rho}n\eps_n^2}}1_{C_n} + P_0(C_n^c) \\
			& =  \frac{\int_{A} \int p_\eta^n(x)^{\rho} p_{\eta_0}^n(x)^{1-\rho} d\La^n(x)d\Pi(\eta)}{\Pi(B_n(\eta_0, \eps_n)) e^{-2{\rho}n\eps_n^2}} + (n\eps_n^2)^{-1},
		\end{split}		
		\end{equation*}
		where the last equality follows from Fubini's theorem. Set
		\begin{align*}
			A_n&:= \left\{\eta: \ \int p_\eta^n(x)^{\rho} p_{\eta_0}^n(x)^{1-\rho} d\La^n(x) \leq e^{- 4n\rho\eps_n^2} \right\} \\
			&= \left\{\eta: \ -\frac{1}{n(1-\rho)}\log\int p_\eta^n(x)^{\rho} p_{\eta_0}^n(x)^{1-\rho} d\La^n(x) \geq 4\frac{\rho\eps_n^2}{1-\rho}\right\}\\
			& = \left\{\eta: \ \frac{1}{n} D_{\rho}(p_{\eta}^n, p_{\eta_0}^n) \geq 4\frac{\rho\eps_n^2}{1-\rho}\right\}.
		\end{align*}		
Substituting $A_n$ into the second-last display and using \eqref{equation_thm_1} yields
		\begin{align*}
			E_0 \Pi_{\rho}(A_n\given X) 
			& \leq \frac{\int_{A_n} e^{- 4n\rho\eps_n^2} d\Pi(\eta)}{\Pi(B_n(\eta_0, \eps_n)) e^{-2{\rho}n\eps_n^2}} +  (n\eps_n^2)^{-1} \leq e^{- n\rho\eps_n^2}  + (n\eps_n^2)^{-1}.
		\end{align*}
	\end{proof}

\begin{lemma}\label{lem:mino_den_posterior} For any distribution $\Pi$ on $S$, any $C, \eps>0$ and $0 < \rho \leq 1$, with $P_{0}$-probability at least $1-1/(C^2n\eps^2)$, we have
		\begin{align*}
			\int_{S} \frac{p_\eta^n(X)^\rho}{p_{\eta_0}^n(X)^\rho}d\Pi(\eta) \geq \Pi(B_n(\eta_0, \eps))e^{-\rho(C+1)n\eps^2}.
		\end{align*}
	\end{lemma}

\begin{proof}
		Suppose $\Pi(B_n(\eta_0, \eps))>0$ (otherwise the result is immediate), and denote by $\bar{\Pi} = \Pi(\cdot \cap B_n(\eta_0, \eps))/\Pi(B_n(\eta_0, \eps))$ the normalized prior to $B_n(\eta_0, \eps)$.
		Now let us bound from below
		\begin{align}\label{eq:ELBO_lb}
			\int_{S} \frac{p_\eta^n(X)^\rho}{p_{\eta_0}^n(X)^\rho}d\Pi(\eta) \geq \int_{B_n(\eta_0, \eps)} \frac{p_\eta^n(X)^\rho}{p_{\eta_0}^n(X)^\rho}d\Pi(\eta) = \Pi(B_n(\eta_0, \eps)) \int \frac{p_\eta^n(X)^\rho}{p_{\eta_0}^n(X)^\rho}d\bar{\Pi}(\eta). 
		\end{align}
		Since $\bar{\Pi}$ is a probability measure on $S$, Jensen's inequality applied to the logarithm gives,
		\begin{align*}
				\log\left(\int \frac{p_\eta^n(X)^\rho}{p_{\eta_0}^n(X)^\rho}d\bar{\Pi}(\eta)\right) 
				\geq  \rho  \int \log\left(\frac{p_\eta^n(X)}{p_{\eta_0}^n(X)}\right)d\bar{\Pi}(\eta).
		\end{align*}	
		Consider now the random variable $Z := \int \log\left(\frac{p_\eta^n(X)}{p_{\eta_0}^n(X)}\right)d\bar{\Pi}(\eta)$. Then
		\begin{align*}
			E_0 |Z|  
			 \leq   \int_{B_n(\eta_0, \eps)} E_0  \left\arrowvert\log\left(\frac{p_\eta^n(X)}{p_{\eta_0}^n(X)}\right)\right\arrowvert d\bar{\Pi}(\eta)
			&= \int_{B_n(\eta_0, \eps)} \int \left\arrowvert\log\left(\frac{p_\eta^n(x)}{p_{\eta_0}^n(x)}\right)\right\arrowvert  p_{\eta_0}^n(x) d\La^n(x)d\bar{\Pi}(\eta) \\ & \leq n\eps^2 + 1. 
		\end{align*}
		Thus Z is integrable and using Fubini's theorem,
		\begin{align*}
			E_0 Z
			= \int_{B_n(\eta_0, \eps)}  \int \log\left(\frac{p_\eta^n(x)}{p_{\eta_0}^n(x)}\right) p_{\eta_0}^n(x) d\La^n(x) d\bar{\Pi}(\eta) 
		 = \int_{B_n(\eta_0, \eps)} -K(p_{\eta_0}^n, p_{\eta}^n) d\bar{\Pi}(\eta) \geq -n\eps^2.
		\end{align*}
		Turning to the variance,
		\begin{align*}
			\text{Var}_0(Z) = \text{Var}_0(-Z) 
			&= E_0 \left(\int \log\left(\frac{p_{\eta_0}^n(X)}{p_\eta^n(X)}\right)d\bar{\Pi}(\eta) - \int_{B_n(\eta_0, \eps)} K(p_{\eta_0}^n, p_{\eta}^n) d\bar{\Pi}(\eta)\right)^2 \\
			&= E_0 \left(\int \log\left(\frac{p_{\eta_0}^n(X)}{p_\eta^n(X)}\right) - K(p_{\eta_0}^n, p_{\eta}^n) d\bar{\Pi}(\eta)\right)^2 \\
			&\leq  \int_{B_n(\eta_0, \eps)} E_0 \left(\log\left(\frac{p_{\eta_0}^n(X)}{p_\eta^n(X)}\right) - K(p_{\eta_0}^n, p_{\eta}^n)\right)^2 d\bar{\Pi}(\eta) 
			\leq n\eps^2
		\end{align*}
		using that $\bar{\Pi}$ is supported on $B_n(\eta_0, \eps)$. Also, 		$P_0(|Z-E(Z)| \geq Cn\eps^2) \leq 1/(C^2n\eps^2)$ by Chebychev's inequality.		Thus, on the event $\{ |Z-E(Z)| \leq Cn\eps^2\}$, which has a probability at least $1 - 1/(Cn\eps^2)$,
		\begin{align*}
			\log\left(\int \frac{p_\eta^n(X)^\rho}{p_{\eta_0}^n(X)^\rho}d\bar{\Pi}(\eta)\right) \geq \rho(Z - EZ + EZ) \geq -\rho(C+1)n\eps^2.
		\end{align*}
		Substituting this bound into \eqref{eq:ELBO_lb} then gives the result.
	\end{proof}

 In the regression model \eqref{def:rdreg}, where one wants to recover $\eta_0=f_0$, 
 the data distribution has density with respect to  $(\mu \otimes \lambda)^{\otimes n}$  given by, recalling $\mu$ is the distribution of the design points and $\lambda$ Lebesgue's measure on $\mathbb{R}$,
 \begin{equation}\label{eq; factorised density}
 p_f^n((y_1,x_1),\dots,(y_n,x_n))=\prod_{i=1}^n (2\pi\sigma_0^2)^{-1/2} e^{-\frac{(y_i-f(x_i))^2}{2\sigma_0^2}},
 \end{equation}
 for $f$ ranging over the set $C[-1,1]^d$ of continuous functions over $[-1,1]^d$. The following proposition shows that the set $B_n(p_f^n, \varepsilon)$ above contains a supremum--norm ball, which simplifies the verification of assumption \eqref{equation_thm_1} in  Theorem \ref{bat_result}.

 \begin{proposition}\label{prop: sup in KL}
For any $f \in C[-1,1]^d$, $p_f^n$ as in \eqref{eq; factorised density} and $\varepsilon\le \sigma_0^{-1}$,  for $B_n$ the neighborhood as in \eqref{defkln},
\[ B_n(p_f^n, \varepsilon) \supseteq \left\{g\in C[-1,1]^d\colon\ \norm{g-f}_\infty \leq \frac{2\sigma_0^2}{\sqrt{1+4\sigma_0^2}}\varepsilon \right\}.\]
 \end{proposition}

  \begin{proof}
For $p_f^n, p_g^n$ as in \eqref{eq; factorised density}, i.e. densities that factorise, $K(p_f^n, p_g^n)=nK(p_f^1, p_g^1)$ and $V(p_f^n, p_g^n)=nV(p_f^1, p_g^1)$. In the following, we write $p_f=p_f^1$ and $p_g=p_g^1$.
For any $f,g$ in $C[-1,1]^d$, we get
\begin{align*}
K(p_f, p_g) &= \int_{[-1,1]^d\times \R} \frac{e^{-\frac{(y_-f(x))^2}{2\sigma_0^2}}}{\sqrt{2\pi\sigma_0^2}} \frac{(y-g(x))^2-(y-f(x))^2}{2\sigma_0^2} d(\mu \otimes \lambda)(x,y) \\
&= \int_{[-1,1]^d\times \R} \frac{g(x)^2-f(x)^2-2y(g(x)-f(x))}{2\sigma_0^2} \frac{e^{-\frac{(y-f(x))^2}{2\sigma_0^2}}}{\sqrt{2\pi\sigma_0^2}}  d(\mu \otimes \lambda)(x,y)\\
&= \int_{[-1,1]^d} \frac{g(x)^2-f(x)^2-2f(x)(g(x)-f(x))}{2\sigma_0^2} d\mu(x)\\
&= \norm{f-g}_{L^2(\mu)}^2/(2\sigma_0^2)\leq \norm{f-g}_{\infty}^2/(2\sigma_0^2).
\end{align*}
Similarly,
\begin{align*}
V(p_f, p_g) &\leq \int_{[-1,1]^d\times \R}\ \frac{\left[g(x)^2-f(x)^2-2y(g(x)-f(x))\right]^2}{4\sigma_0^4} \frac{e^{-\frac{(y-f(x))^2}{2\sigma_0^2}}}{\sqrt{2\pi\sigma_0^2}}  d(\mu \otimes \lambda)(x,y) \\
&= \int_{[-1,1]^d\times \R}\ (g-f)(x)^2\frac{\left[g(x)+f(x)-2y\right]^2}{4\sigma_0^4} \frac{e^{-\frac{(y-f(x))^2}{2\sigma_0^2}}}{\sqrt{2\pi\sigma_0^2}}  d(\mu \otimes \lambda)(x,y)  \\
&= \left(\norm{g-f}_{L^2(\mu)}^2/\sigma_0^2 + \norm{g-f}_{L^4(\mu)}^4 /(4\sigma_0^4)\right)\leq \norm{f-g}_\infty^2 \frac{1+4\sigma_0^2}{4\sigma_0^4},
\end{align*}
the last inequality being true whenever $\norm{f-g}_\infty\leq 1$, which follows from $\veps\le \sigma_0^{-1}$. To conclude, we note that $2\sigma_0^2 \geq 4\sigma_0^4/(1+4\sigma_0^4)$.
 \end{proof}


\section{Verifications for specific priors on scaling parameters}
\label{sec:verif}

\begin{lemma} \label{lemexp}
Let $n\ge 1$, $\rho\in(0,1)$, $1\le d^*\le d$ and $a^*\geq 1$. Suppose 
\[ 1/a^* \le \la \le 8\sqrt{\rho n d}/\xi. \] 
Condition \eqref{condpr} is verified for $\pi$ an exponential prior $\cE(\la)$ of density $a\to \la e^{-\la a}\1_{a>0}$ 
if
\[ 
n\veps_n^2 \ge (2/\rho)\left[d\log\left(\frac{16d\sqrt{\rho n}}{\xi \la} \right)+2\la d^*a^* +\log 2\right].
\]
\end{lemma}
\begin{proof}
One bounds from below each integral in \eqref{condpr}. Starting with the second integral
\[ 
\int_{a^*}^{2a^*} \pi(a) da = e^{-\la a^*}-e^{-2\la a^*}\ge e^{-2\la a^*},
\]
using that $e^{-x}\ge 2 e^{-2x}$ if $x>\log{2}$, which holds in particular if $x>1$ (corresponding to the assumption $\la a^*\ge 1$ if one takes $x=\la a^*$). For the first integral, let us set $E:=\xi/(8\sqrt{\rho})$ and bound from below, with $e^{-x}\le 1-x/2$ for $x\le 1$  (applied here for $x=E\la/d\sqrt{n}\le 1$),
\[ \left(\int_0^{E/d\sqrt{n}} \la e^{-\la a} da\right)^{d-d^*}
 =  \left(1- e^{-E\la/d\sqrt{n}} \right)^{d-d^*}
 \ge \exp\{ d \log\left(1-(1- \frac{E\la}{2d\sqrt{n}}) \right) \}.
 \]
Combining the previous bounds gives that   \eqref{condpr}  holds under the conditions of the lemma.
\end{proof}

\begin{lemma} \label{lem-hs-fix}
Let $n\ge 1$, $\rho\in(0,1)$, $1\le d^*\le d$ and $a^*\geq 1$. Suppose 
\[ \xi/(8d\sqrt{\rho n}) \le \ta\le a^*. \] 
Condition \eqref{condpr} is verified for $\pi=\pi_\ta$ a horseshoe prior of parameter $\tau>0$ 
if
\[ 
n\veps_n^2 \ge (2/\rho)\left[d\log\left(\frac{8d\sqrt{\rho n}}{\xi e_0 \ta} \right)+ d^*\log(10a^*/\ta)+\log 2
\right],
\]
where $e_0=2\log{5}/(2\pi)^{3/2}$.
\end{lemma}
\begin{proof}
The proof is similar to that of Lemma \ref{lemexp}, but we now use  Lemma \ref{lemma: bound on the horseshoe} to bound $\pi_\ta$ from below. Using \eqref{hs-smalltau} with $\xi/(8d\sqrt{\rho n}) \le \ta$ and \eqref{eq: prob bound smoothing high} with $\ta\le a^*$, one obtains that the left-hand side of \eqref{condpr} is bounded from below by
\[ \left(\frac{e_0\xi}{8\ta d\sqrt{\rho n}}\right)^{d} \exp(-d^*\log(10a^*/\ta)), \]
which gives the result by rearranging.

\end{proof}

\begin{lemma} \label{lem-hs-van}
Let $n\ge 1$, $\rho\in(0,1)$, $1\le d^*\le d$ and $a^*\geq 1$. 
Consider $\pi=\pi_\ta$ a horseshoe prior of parameter $\tau>0$. Then  \eqref{condpr} is verified for  large enough $n$ if
\begin{equation} 
10a^*e^{-n\rho\varepsilon_n^2/4d^*}  \le \tau \le 
\frac{\xi}{d^2}\frac{1}{\sqrt{\rho n}}
\end{equation}
In particular, for $a^*, \veps_n$ as in \eqref{optastar}--\eqref{vepsfinal}, the latter display holds for large enough $n$ and fixed $K$ if one sets for $s>0$
\begin{equation*} 
\ta = \ta^*:=(n^{1+s} d^4)^{-1/2}.
\end{equation*}
\end{lemma}
\begin{proof}
We proceed as for the proof of Lemma \ref{lem-hs-fix}, now using the lower bound \eqref{eq: prob bound small K} of $\pi_\ta$. Setting $\delta=\xi/(8d\sqrt{\rho n})$,
\[ \left( \int_0^\delta \pi_\tau(a)da \right)^{d-d^*}
\ge \left(1-4\tau/(\delta\sqrt{2\pi^3})\right)^d, \]
choosing $\tau\le \delta\sqrt{2\pi^3}/8$ so that the term in brackets in the last display is at least $1/2$. Since $\log(1-x)\ge -x(\log{2})$ for $x\le 1/2$ by concavity, one deduces
\[ \left( \int_0^\delta \pi_\tau(a)da \right)^{d-d^*}
\ge \exp\{ - d(\log{2})(32/\sqrt{2\pi^3})(d\sqrt{\rho n}/\xi)\tau \}
\ge \exp\left\{ - d^2 \frac{5\sqrt{\rho n }}{\xi} \tau \right\}. \]
This is larger than $2\exp(-\rho n\veps_n^2/4)$ for large enough $n$ under the condition of the lemma, using that $n\veps_n^2\to \infty$ as $n\to\infty$. 

The bound from below of the second integral in \eqref{condpr} is exactly the same as in Lemma \ref{lem-hs-fix}, and the lower bound on $\tau$ of the present lemma follows by asking that the obtained bound $\exp(-d^*\log(10a^*/\ta))$ is at least $\exp\{-\rho n\veps_n^2/4\}$.
\end{proof}

\section{Deep prior with deterministic $q_{max}, d_{max}$}
\label{app: fixed dmax and qmax}

Instead of taking for random $q$ and $d_i$, $i=1,\dots,q$, another possibility is to fix deterministic $q_{max}$ and $d_{max}$; recall the definition of the simpler prior \textsf{Deep--HGP}$(q_{max},d_{max})$
 \begin{equation}\label{eq:dgpmax_app}
 \begin{split}  
g_{ij}\  &\overset{\text{ind.}}{\sim} \textsf{HGP}(\tau_i) \\
 f\ |\ (g_{ij}) &=g_q\circ \dots\circ g_0
 \end{split}  
 \end{equation}
where $\tau_i>0$ for $i=0,\ldots, q$ and $g_i=(g_{ij})_j$ with $g_{ij}:I^{d_{max}}\to I$, $1\leq j\leq d_{max}$, for $i=1,\ldots, q$. For $i=0$, we still consider that $g_{0j}\colon I^{d}\to I$. There is no need to sample from $q$ and $d_i$'s in this posterior (since those are now fixed in the prior to the values $q_{max}$ and $d_{max}$ respectively), which simplifies significantly computational complexity. Though this prior is seemingly less flexible, the next theorem shows that it can still performs adaptation. In particular, as long as the true regression function belongs to a class $\cF_{deep}(\la,\be,K)$, for a parameter $\la=(q, (d_1,\ldots,d_q) , (t_0,\ldots,t_q) )$, whenever $q\le q_{max}$ and $d_i\le d_{max}$ for $i=1,\ldots,q$, this deterministic choice enables adaptation.
 
 \begin{theorem}\label{theorem: posterior contraction rates hdgp fixed comp}
 For $q_{max}\geq 0$ and $d_{max}\geq 1$, let $\Pi$ be the \textsf{Deep--HGP} prior from \eqref{eq:dgpmax_app} with fixed parameters $\tau_i>0$.
 For $\la=(q, (d_1,\ldots,d_q) , (t_0,\ldots,t_q) )$ such that $q\le q_{max}$ and $d_i\le d_{max}$ for $i=1,\ldots,q$, $\mathbb{\beta}=(\beta_0,\dots,\beta_{q})$, $d_0=d\geq1$, and $K\geq1$, suppose $f_0\in \mathcal{F}_{\text{deep}}(\lambda,\mathbb{\beta},K)$. Denoting
 $\alpha_i = \prod_{l=i+1}^q (\beta_l\wedge 1)$, let us set 
  \[\veps_n = \max_{i=0,\dots,q} 
n^{-\frac{\alpha_i\be_i}{2\alpha_i\be_i+t_i}}. \]
 Then, for any $0<\rho<1$,  $\Pi_\rho[\cdot\given X,Y]$ contracts to $f_0$ at the rate in $\norm{\,\cdot\,}_{L^2(\mu)}$ distance,  for some $\gamma\geq 0$ and any $M_n \to \infty$,
  \[ E_{f_0} \Pi_\rho[f:\ \norm{f-f_0}_{L^2(\mu)} \geq M_n \log^\gamma(n) \varepsilon_n\given X,Y]\to 0. \]
 \end{theorem}
\begin{proof}

    The proof is similar to the one of Theorem \ref{theorem: posterior contraction rates hdgp}. We therefore only highlight the main modifications. We have to pay attention to two new aspects of the problem: the hyperparameters $q_{max}$ and $d_{max}$ are fixed in the prior, and these may not be equal to those in the set $\mathcal{F}_{\text{deep}}(\lambda,\mathbb{\beta},K)$. In the following, we decompose $f_0=h_q\circ\dots\circ h_0\in \cF_{deep}(\la,\be,K)$, for $\lambda=(q,d_1,\dots,d_q,t_0,\dots,t_q)$ and
$\beta=(\beta_0,\dots,\beta_{q})$, where $h_i\in \cF_{VS}(K,\beta_i, d_i,t_i)$.

We proceed in two steps. First, we construct a function $\tilde{f}_0$ that coincides with $f_0$ in a pointwise sense, but which can be written as a well-chosen composition of $q_{max}+1$ functions with multivariate functions inside the composition (for $1\le i\le q_{max}$) that each act on $d_{max}$ dimensions. We will also check that $\tilde{f}_0$ belongs to a class $\cF_{deep}(\La,\tilde{\be},K)$ over which the target convergence rate coincides with $\veps_n$ (up to logarithmic factors, which will be silently understood in the rest of the proof) as in the statement of the Theorem.  Second, 
it will then suffice to follow the proof of Theorem \ref{theorem: posterior contraction rates hdgp}, implying posterior contraction at the desired rate $\veps_n$, again up to a logarithmic factor.


For the first step, let us note that $f_0$ is also equal (pointwise) to $\tilde{f}_0=H_{q_{max}}\circ\dots\circ H_0$ where the consecutive $H_{i}$ are defined as follows:
 \begin{itemize}
\item $H_0\colon I^{d}\to I^{d_{max}}$ is such that $H_{0j}=h_{0j}$ for $j\leq d_1$, and $H_{0j}=0$ otherwise;
\item for $1\leq i\leq q-1$, $H_i\colon I^{d_{max}}\to I^{d_{max}}$ is such that $H_{ij}(x_1,\dots,x_{d_{max}})=h_{ij}(x_1,\dots,x_{d_i})$ for $j\leq d_{j+1}$, and $H_{ij}=0$ otherwise;
\item for $q\leq i\leq q_{max}-1$, $H_i\colon I^{d_{max}}\to I^{d_{max}}$ is the identity function;
\item $H_{q_{max}}\colon I^{d_{max}}\to I$ is such that $H_{q_{max}}(x_1,\dots,x_{d_{max}})=h_{q}(x_1,\dots,x_{d_q})$.
 \end{itemize}
Indeed, for any $x\in I^d$,
\begin{align*}
\Tilde{f}_0(x)&=H_{q_{max}}\circ\dots\circ H_0(x)\\
&=H_{q_{max}}\circ\cdots\circ H_q \circ H_{q-1}\circ\dots\circ H_0(x)\\
&=H_{q_{max}}\circ H_{q-1}\circ\dots\circ H_0(x)\\
&=h_q\circ\dots\circ h_0(x)=f_0(x).
\end{align*}
Introducing
\begin{align*}
\Lambda=(&q_{max},(\tilde{d}_{1},\dots,\tilde{d}_{q_{max}}),(\tilde{t}_{0},\dots,\tilde{t}_{q_{max}})),\\
\coloneqq(&q_{max},\\
&(\underbrace{d_{max},\dots,d_{max})}_{q_{max} \text{ times}},\\
&(t_0,t_1\dots,t_{q-1},\underbrace{d_{max},\dots,d_{max}}_{q_{max}-q \text{ times}},t_q)),
\end{align*}
and
\[\tilde{\beta}=(\beta_0,\beta_1,\dots,\underbrace{\beta_{-1},\dots,\beta_{-1}}_{q_{max}-q \text{ times}},\beta_{q}),\] for arbitrarily large $\beta_{-1}>0$ (the identity function has infinitely large regularity), we note that by definition $H_{ij}\in \cF_{VS}(1\vee K,\tilde{\be}_i, \tilde{d}_i,\tilde{t}_i)$ -- observe that the extra coordinate functions (equal to the null function, which is infinitely smooth) do not imply any change in the regularities $\be_i$ -- and as a consequence, $\Tilde{f}_0\in \cF_{deep}(\La,\tilde{\be},1\vee K)$. 
Let us now express the target contraction rate over this class. Using a similar  notation for the updated regularity parameters, let $\tilde{\al}_i=\prod_{j=i+1}^{q_{max}} (\tilde{\be}_i \wedge 1)$. So, the target rate is, up to logarithmic terms,
\[n^{-1/\left(2+\underset{{0\leq i\leq q_{max}} }{\max}\ \frac{\tilde{t}_i}{\tilde{\al}_i \tilde{\be}_i}\right)},\]and for $\beta_{-1}$ large enough, this reduces to 
\begin{equation}\label{equrate}
n^{-1/\left(2+\underset{{0\leq i\leq q_{max}} }{\max}\ \frac{\tilde{t}_i}{\tilde{\alpha}_i\tilde{\beta}_i}\right)}=n^{-1/\left(2+\underset{{0\leq i\leq q} }{\max}\ \frac{t_i}{\alpha_i\beta_i}\right)},
\end{equation}
the order of the minimax rate over $\cF_{deep}(\la,\be,K)$. To verify this, note that none of the factors $\tilde{\alpha}_i=\prod_{j=i+1}^{q_{max}}\left(\tilde{\beta}_j\wedge 1\right)$ depends on $\beta_{-1}>1$, and one has $\tilde{\al_i}=\al_i$ for $1\le i\le q-1$ and $\tilde{\al}_{q_{max}}=\al_q$.This means that $\underset{{0\leq i\leq q-1} }{\max}\ \frac{\tilde{t}_i}{\tilde{\alpha}_i\tilde{\beta}_i}=\underset{{0\leq i\leq q-1} }{\max}\ \frac{t_i}{\alpha_i\beta_i}$ as long as $\beta_{-1}\geq1$. Also, we note that $\frac{\tilde{t}_{q_{max}}}{\tilde{\alpha}_{q_{max}}\tilde{\beta}_{q_{max}}}=\frac{t_q}{\alpha_q\beta_q}$. Finally, for $q+1\le i\le q_{max}-1$, we have $\frac{\tilde{t}_i}{\tilde{\alpha}_i\tilde{\beta}_i}=\frac{d_{max}}{(1\wedge \be_{q})\be_{-1}}$ so that these ratios in \eqref{equrate} can be taken as small as desired (since $\be_{-1}$ can be taken arbitrarily large) and thus do not change the value of the maximum over all $1\le i\le q_{max}$, so the claim is verified.
 
For the second step, we can follow the proof of Theorem \ref{theorem: posterior contraction rates hdgp}, where now $\tilde{f}_0$ plays the role of $f_0$ therein, and with an updated number of compositions equal to $q_{max}$ and inner dimensions equal to $d_{max}$. The  number of compositions and dimensions being now deterministic, we can directly start from, recalling $\tilde{f}_0=f_0$ pointwise,
\begin{equation*}
 \Pi\left[f:\ \norm{f-\tilde{f_0}}_\infty <\xi \varepsilon_n\right]=  \Pi\left[\norm{\Psi(g_{q_{max}})\circ \dots\circ \Psi(g_0)-H_{q_{max}}\circ\dots\circ H_0}_\infty <\xi\varepsilon_n\right]. 
 \end{equation*}
The proof is then the same, except that the involved constants now depend on $q_{max}, d_{max}$ (instead of $q$ and $d_i$'s). For instance, the prior-mass terms corresponding to restricting to intervals $I_j$ (defined within the proof of Theorem \ref{theorem: posterior contraction rates hdgp}) are of the form $\prod_{j=1}^{d_{max}} \pi(I_j)$, and this is bounded from below by $c^{d_{max}}$ for small enough $c>0$. 
 Since $q_{max}, d_{max}$ are fixed (non-$n$-dependent), all such terms are bounded from below by strictly positive universal constants, so  the desired rate $\veps_n$ is obtained, which concludes the proof. 
\end{proof}

\section{Results for standard posteriors}

\begin{proof}[Proof of Proposition \ref{proppost}]
Let $B:=\{f:\ \|f-f_0\|\le M\veps_n\}$, for a suitable constant $M$ to be chosen large enough. We wish to show that $\Pi\left[B\times \RR^+\given (X,Y)\right]$ goes to $1$ in probability under $P_{f_0,\si_0^2}$. By Bayes formula, recalling that $b\in(0,1)$ is the parameter in the definition of the prior \eqref{priorsig} on $\si^2$, and setting $s^2:=b\si^2$,
\begin{align*}
\Pi[B\times \RR^+& \given X,Y]  
= \frac{\int \int_B (\si^2)^{-\frac{n}{2}} \exp\left\{-\sum_{i=1}^n (Y_i-f(X_i))^2/(2\si^2)\right\} d\Pi(f,\si^2)}{\int \int (\si^2)^{-\frac{n}{2}} \exp\left\{-\sum_{i=1}^n (Y_i-f(X_i))^2/(2\si^2)\right\} d\Pi(f,\si^2)}\\
& = \frac{\int \int_B (\si^2)^{-b\frac{n}{2}}(\si^2)^{-(1-b)\frac{n}{2}}  \exp\left\{- \sum_{i=1}^n (Y_i-f(X_i))^2/(2\si^2)\right\} d\Pi(f,\si^2)}{\int \int (\si^2)^{-b\frac{n}{2}}(\si^2)^{-(1-b)\frac{n}{2}} \exp\left\{-\sum_{i=1}^n (Y_i-f(X_i))^2/(2\si^2)\right\} d\Pi(f,\si^2)}\\
& = \frac{\int \int_B (\si^2)^{-b\frac{n}{2}}  \exp\left\{- \sum_{i=1}^n (Y_i-f(X_i))^2/(2\si^2)\right\} \,d\Pi_f(f)\, e^{-b \si^2}d\si^2}{\int \int (\si^2)^{-b\frac{n}{2}} \exp\left\{-\sum_{i=1}^n (Y_i-f(X_i))^2/(2\si^2)\right\}\,d\Pi_f(f)\, e^{-b \si^2}d\si^2}
\end{align*}   
where we have used the definition \eqref{priorsig} which gives \[d\pi_{\sigma^2}(\sigma^2)=b^{(1-b)\frac{n}{2}+1}\Gamma\left((1-b)\frac{n}{2}+1\right)^{-1}(\sigma^2)^{(1-b)\frac{n}{2}}e^{-b\sigma^2}d\sigma^2.\] Next one does the change of variables $s^2=b\sigma^2$, so that $d\sigma^2=ds^2/b$ (the integral domain for $s^2$ is $\RR^+$, remaining unchanged through the change of variables), which  leads to
\begin{align*}
\Pi[B\times \RR^+& \given X,Y]  
& = \frac{\int \int_B (s^2)^{-b\frac{n}{2}} \exp\left\{-b \sum_{i=1}^n (Y_i-f(X_i))^2/(2s^2)\right\} \,d\Pi_f(f)\, e^{-s^2}ds^2}{\int \int (s^2)^{-b\frac{n}{2}} \exp\left\{-b\sum_{i=1}^n (Y_i-f(X_i))^2/(2s^2)\right\} \,d\Pi_f(f) \, e^{-s^2} ds^2}.
\end{align*}   

For $Q$ a prior distribution on $(f,\si^2)$ and $b\in(0,1)$, let $Q_b$ denote the fractional posterior distribution 
\[ Q_b[B\times C\given X,Y] = \frac{\int_B\int_C p_{f,\si^2}(X,Y)^b dQ(f,\si^2)}{\int\int p_{f,\si^2}(X,Y)^b dQ(f,\si^2)}.\]
Define the product distribution 
\[ Q := \Pi_f \otimes \pi_1,\] where $\pi_1=\text{Exp}(1)$  the standard exponential distribution. Examining the expression of $\Pi[B\times \RR^+ \given X,Y]$ above, one notes 
\begin{equation} \label{eqpq}
\Pi[B\times \RR^+ \given X,Y]  = Q_b[ B \times \RR^+ \given X,Y].
\end{equation}
It is therefore enough to show that the fractional posterior mass in the last display goes to $1$ in probability. For this, it is enough to verify a prior mass condition only: we use Theorem \ref{bat_result}, which gives a result in terms of the $D_b$--Rényi divergence. In a final step of the proof, one links this distance to the $L^2(\mu)$--norm.\\

{\em Verification of the prior mass condition.} We claim that, for $B_n((f_0,\si_0^2),\veps)$ the KL-type neighborhood as defined in \eqref{defkln}, one can find a fixed constant $C>0$ sufficiently large such that it holds, for any $\veps_n=o(1)$ as $n\to\infty$,
\begin{equation} \label{relkl}
\left\{(f,\si^2):\ \|f-f_0\|_\infty\le \veps_n\, ,\, |\si^2-\si_0^2|\le \veps_n^2 \right\}
\subset B_n((f_0,\si_0^2),C\veps_n).
\end{equation}
This follows from Lemma \ref{kllem}: to bound from above  the quantities appearing therein, one can use, for instance, the bounds $\log(\si_0/\si)\le (\si_0-\si)/\si\le|\si_0-\si|/\si$ and $\log^2(\si_0/\si)\le 2(\si_0/\si-1)^2$ provided $\si$ is in a small enough neighborhood of $\si_0$, which is granted since one works under $|\si^2-\si_0^2|\le \veps_n^2$ (and the latter implies $|\si-\si_0|\le \veps_n^2/\si_0^2$). For the parts involving $f$, one uses $\int (f-f_0)^pd\mu\le \|f-f_0\|_\infty^p$ (for $p=2, 4$).

Now using the definition of $Q$, for $C$ as above and $\veps_n$ the rate in the Lemma,
\begin{align*}
 Q\left[B_n( (f_0,\si_0^2),C\veps_n) \right] & \ge \Pi_f(\|f-f_0\|_\infty\le \veps_n)
\pi_1\left([\si_0^2\pm \veps_n^2]\right) \\
& \ge C_2\veps_n^2 e^{-Dn\veps_n^2}\ge \exp(-C_3n\veps_n^2), 
\end{align*}
for a large enough $C_3>0$. 

We now apply Theorem \ref{bat_result} from Appendix \ref{app: Frac post}  with $\rho=b$ and for $\veps_n$ replaced by $C_4\veps_n$ and $C_4\ge \max(C,\sqrt{C_3/\rho})$ a large enough constant. Since \eqref{equation_thm_1} holds for this choice, one deduces that the fractional posterior $Q_b[\cdot\given X,Y]$ converges at rate $C_5\veps_n^2$, with $C_5=4bC_4/(1-b)$, in terms of the $b$-R\'enyi divergence:
\begin{equation} \label{rencv}
E_{f_0,\si_0^2}Q_b\left[D_b(p_{f,\si^2},p_{f_0,\si_0^2})\ge C_5\veps_n^2
\given X,Y \right] = o(1).
\end{equation}  
\vspace{.5cm}

{\em Relating $D_b$--divergence and $L^2(\mu)$--norm.} Let us first express $D_b(p_{f,\si^2},p_{f_0,\si_0^2})$ more explicitly. By using the standard formula expressing the $D_b$--divergence between two arbitrary Gaussians in one dimension 
\[ D_b(\cN(f(x),\si^2),\cN(f_0(x),\si_0^2)) =
b\frac{(f(x)-f_0(x))^2}{2\si_b^2}+\frac{1}{1-b}\log \frac{\si_b}{\si^{1-b}\si_0^b},\]
where $\si_b^2:=(1-b)\si^2+b\si_0^2$, see e.g. \cite{vanerven14}, eq. (10), one has
\begin{align*}
D_b(p_{f,\si^2},p_{f_0,\si_0^2})& =\frac{1}{b-1}\log \int e^{ (b-1)D_b\left(\cN(f(x),\si^2)\,,\,\cN(f_0(x),\si_0^2)\right) } d\mu(x) \\
& = \frac{1}{1-b}\log\left(\frac{\si_b}{\si^{1-b}\si_0^b}\right)
+\frac{1}{b-1}\log 
\int e^{- b(1-b)(f(x)-f_0(x))^2/(2\si_b^2) }d\mu(x),\\
& =: \qquad \quad\ \ (I)\qquad \qquad+\qquad\quad (II).
\end{align*}
We have both $(I)\ge 0$ (e.g. by concavity of the logarithm) and $(II)\ge 0$. It follows from \eqref{rencv} that both terms converge to $0$ at rate $C_5\veps_n^2$ under the fractional posterior, and so  $|\si^2-\si_0^2|$ converges at rate $C'\veps_n$ under the fractional posterior, using Lemma \ref{funconc}. In particular, the event 
$\{\si_b^2>2\si_0^2\}$ has vanishing probability. Since $(II)=D_b(p_{f,\si_b^2},p_{f_0,\si_b^2})$, we have $(II)\ge D_b(p_{f,2\si_0^2},p_{f_0,2\si_0^2})$ with high probability, by using that the expression of (II) decreases if the variance parameter increases. That is, for some $C>0$,
\begin{equation} \label{renf}
E_{f_0,\si_0^2}Q_b\left[D_b(p_{f,2\si_0^2},p_{f_0,2\si_0^2})\ge C\veps_n^2| X,Y\right]=o(1).
\end{equation}
Noting that the event in the last display does not put any constraint over $\si^2$, and so is of the form $\{(f,\si^2):\ (f,\si^2)\in B\times\RR^+\}$, one can use \eqref{eqpq} to 
obtain the first statement of the lemma. 

For the second statement, one further bounds $D_b(p_{f,2\si_0^2},p_{f_0,2\si_0^2})$  from below using $-\log(x)\ge 1-x$ and next $1-e^{-x}\ge xe^{-x}$, so that
\begin{align*}
(II) &\ge \frac{b}{4\si_0^2}\int  (f(x)-f_0(x))^2e^{- b(1-b)(f(x)-f_0(x))^2/(4\si_0^2) }d\mu(x), 
\end{align*} 
and so the last term converges at rate $\veps_n^2$ under the $Q_b$ fractional posterior. By using \eqref{eqpq}, one deduces that the same is true under the original posterior. Now using that $\|f_0\|_\infty$ is finite and that $\|f\|_\infty$ is bounded by a fixed constant under the posterior, the exponential in the last display can be bounded below under the posterior by $c\int (f(x)-f_0(x))^2 d\mu$, which proves the second part of the lemma.
\end{proof}

%
%
%
 

\begin{lemma} \label{kllem}
Let $\si,\si_0>0$ and denote $\|f-f_0\|_\mu^2=\int (f-f_0)^2d\mu$, then
\begin{align*}
KL(P_{f_0,\si_0^2},P_{f,\si^2}) &
= \log(\si/\si_0)+\frac{(\si_0/\si)^2-1}{2} + \frac{\|f-f_0\|_\mu^2}{2\si^2},\\
V(P_{f_0,\si_0^2},P_{f,\si^2}) &
= \frac{\int (f-f_0)^4(x)d\mu(x)-\|f-f_0\|_\mu^4}{4\si^4}
+ \frac{((\si_0/\si)^2-1)^2}{2}+\frac{\si_0^2\|f-f_0\|_\mu^2}{\si^4}.
\end{align*}
\end{lemma}
\begin{proof} 
One first notes that $\log(dP_{f_0,\si_0^2}/P_{f,\si^2})(X,Y)$ can be written $\log(\si/\si_0)-Z_0(X,Y)^2/2+Z(X,Y)^2/2$, where we have set
\[ Z_0= (Y-f_0(X))/\si_0,\qquad Z= (Y-f(X))/\si.\]
Under $P_{f_0,\si_0^2}$, we have $Z_0=:\eta\sim\cN(0,1)$ and $Z=(f_0-f)(X)/\si+\si_0\eta/\si=\Delta+\si_0\eta/\si,$ setting $\Delta:=(f_0-f)(X)/\si$. The formula for the KL--divergence follows. Next
\[ V(P_{f_0,\si_0^2},P_{f,\si^2}) 
= E_0\left\{ (Z^2-E_0Z^2) -(Z_0^2-E_0Z_0^2)\right\}^2/4,
\]
where $E_0$ denotes the expectation under $P_{f_0,\si_0^2}$. In terms of $\eta, \Delta$, the square in the last display can be written
\[ \left\{ (\Delta^2-E_0\Delta^2) +\left[ (\si_0/\si)^2-1 \right](\eta^2-1)-2\si_0\eta\Delta/\si \right\}^2. \]
Now expanding the square over the three terms, and noting that crossed terms are zero since $\eta\sim\cN(0,1)$ and independent of $\Delta$, one gets the desired formula. 

\end{proof}

\begin{lemma} \label{funconc}
Let $\psi=\psi_b$ be the function defined on $(0,\infty)$, for $b\in(0,1), \si_0\in(0,\infty)$, by 
\[ \psi(\si^2):= \log\left\{ \frac{(1-b)\si^2+b\si_0^2}{(\si^2)^{1-b}(\si_0^2)^{b}} \right\}. \]
There exists a constant $d=d(\si_0^2,b)>0$ such that, for any $0<\veps<1$,
\[ \left\{\si^2:\ \psi(\si^2)\le d\veps^2\right\} \subset \left\{\si^2:\ |\si^2-\si_0^2|<\veps\right\}. \]
\end{lemma}
\begin{proof}
Denoting $\si_b^2:=(1-b)\si^2+b\si_0^2$, it follows from simple algebra that 
\[ \psi'(\si^2)= b(1-b)(1-(\si_0/\si)^2)\si_b^{-2}, \]
so that $\psi'(\si_0^2)=0$, $\psi$ is decreasing from $+\infty$ to $\psi(\si_0^2)=0$ on $(0,\si_0^2]$ and increasing from $0$ to $+\infty$ on $[\si_0^2,\infty)$. Also,
\[ \psi''(\si^2)=\frac{1-b}{(\si^2)^2} - \frac{(1-b)^2}{(\si_b^2)^2},\]
so that $\psi''(\si_0^2)=b(1-b)/(\si_0^2)^2=:2m>0$. By continuity of $\psi"$, there exists an interval $[\si_0^2\pm \delta]$, for some $\delta=\delta(\si_0^2,b)$, such that $\psi''(y)\ge m$ for all $y\in [\si_0^2\pm \delta]$. Set
\[ d:=d(\si_0^2,b)=\min\left(m/2, \psi(\si_0^2-\delta), \psi(\si_0^2+\delta)\right). \]
Let $\si^2$ verify $\psi(\si^2)\le d\veps^2 < d$. Then by monotonicity of $\psi$ on each side of $\si_0^2$, one must have $|\si^2-\si_0^2|\le \delta$ (otherwise $\psi(\si^2)$ would be larger than either $\psi(\si_0^2-\delta)$ or $\psi(\si_0^2+\delta)$). Expanding  $\psi$ around $\si_0^2$ gives, for all $\si^2\in[\si_0^2\pm\delta]$ and some $\zeta\in[\si_0\pm \delta]$,
\[\psi(\si^2)=\psi(\si_0^2)+ \psi'(\si_0^2)(\si^2-\si_0^2)+ \psi''(\zeta)(\si^2-\si_0^2)^2/2
\ge (m/2)(\si^2-\si_0^2)^2,
\]
where we use that $\psi''$ is bounded from below by $m$ on the considered interval. Then $|\si^2-\si_0^2|\le \sqrt{(2d/m)}\veps$ which gives the result using $2d/m\le 1$ by definition.
\end{proof}

\end{appendix}



\end{document}